\documentclass[11pt,reqno]{amsart}
\usepackage{xifthen}
\usepackage{subfig}
\usepackage{pkgfile}

\usepackage[normalem]{ulem}

\newcommand{\Tr}{\mathrm{Tr} \,}

\newcommand{\oneperp}{{\langle {\mathbf {1}} \rangle^{\perp}}}
\DeclareMathOperator{\Sparse}{Sparse}
\DeclareMathOperator{\Flat}{Flat}
\newcommand{\Log}{{\rm Log}}
\newcommand{\imag}{\mathrm{i}}

\newcommand{\expo}[1]{\exp \left( #1 \rule{0mm}{3mm}\right)}
\DeclareMathOperator{\bpi}{\bm{\pi}}

\DeclareMathOperator{\tpi}{\widetilde{\pi}}
\DeclareMathOperator{\Id}{Id}
\newcommand{\tA} {\widetilde{S}_n}
\newcommand{\tR}{\widetilde{R}}
\newcommand{\eqd}{\,{\buildrel d \over =}\,}
\newcommand{\me}{{m}}
\def\corO{}

\def\corA{}
\def\corAa{}

\def\corN{}

\newcommand\nc\newcommand
\nc\dmo\DeclareMathOperator
\nc\bs\boldsymbol
\nc{\Blue}[1]{\textcolor[rgb]{0,0,1}{[#1]}}
\nc{\sph}{\mathbb{S}^{n-1}}
\nc{\uball}{\mathbb{B}^n}
\dmo{\un}{\mathbf{1}}
\def \lf {\lfloor}
\def \rf {\rfloor}
\dmo{\Comp}{Comp}

\allowdisplaybreaks[2]
\begin{document}

\title[Circular law for the sum of random permutation matrices]{\corN{Circular law for the sum of random permutation matrices}}

\author[A.\ Basak]{Anirban Basak$^*$}\thanks{${}^*$Partially supported by
 grant 147/15 from the Israel Science Foundation}
 \address{$^*$Department of Mathematics, Weizmann Institute of Science 
 \newline\indent POB 26, Rehovot 76100, Israel}
 \author[N.\ Cook]{Nicholas Cook$^\ddagger$}\thanks{${}^\ddagger$Partially supported by NSF postdoctoral fellowship DMS-1606310}
 \address{$^\ddagger$Department of Mathematics, {University of California
\newline\indent Los Angeles, CA 90095-1555}}
\author[O.\ Zeitouni]{Ofer Zeitouni$^{\mathsection}$}\thanks{${}^{\mathsection}$Partially 
supported by  grant 147/15 from the Israel Science Foundation}
\address{$^{\mathsection}$Department of Mathematics, Weizmann Institute of Science 
 \newline\indent POB 26, Rehovot 76100, Israel}



\date{\today}



\maketitle

\begin{abstract}
Let $P_n^1,\dots, P_n^d$ be $n\times n$ permutation matrices drawn independently and uniformly at random, and set $S_n^d:=\sum_{\ell=1}^d P_n^\ell$.
We show that if $\log^{12}n/(\log \log n)^{4} \le d=O(n)$, then the empirical spectral
distribution of $S_n^d/\sqrt{d}$ converges weakly to
the circular law in probability as $n \to \infty$.
\end{abstract}

\section{Introduction}
For an $n \times n$ matrix $M_n$ let $\lambda_1(M_n), \lambda_2(M_n), \ldots,\lambda_n(M_n)$ be its eigenvalues. We define the {\em empirical spectral distribution} (\abbr{ESD}) of $M_n$ as follows:
\[
L_{M_n}:= \f{1}{n}\sum_{i=1}^n \delta_{\lambda_i(M_n)}.
\]

For a sequence of random probability measures
$\{\mu_n\}_{n \in \N}$, supported on the complex plane,
we say that $\mu_n$ converges weakly to a limiting probability measure $\mu$,
in probability, if for every bounded continuous function
$f: \C \mapsto \R$,
 \beq
 \int f d\mu_n - \int f d\mu \ra 0 \qquad \text{ as } n \ra \infty,\label{eq:in_probab_defn}
 \eeq
in probability. If \eqref{eq:in_probab_defn} holds almost surely
we say that $\mu_n$ converges weakly to $\mu$, almost surely.

We are concerned in this paper with the \abbr{ESD} of certain random, non-normal
matrices, defined as follows. For a positive integer $n$,
let $\pi^i_n$, $i=1,2,\ldots$
denote i.i.d.~permutations,
distributed uniformly on the symmetric group $\mathbb{S}_n$.
Let $P^i_n$
denote the associated permutation matrices, {i.e.},
$P^\ell_n (i,j):= \bI(\pi^\ell_n(i)=j)$ for $\ell\in [d]$, $i,j\in[n]$ \corAa{where for any integer $m$ we denote $[m]:=\{1,2,\ldots,m\}$}. For $d$ an integer,
define $S_n^d$ as
\begin{equation}	\label{def:S}
S_n^d(i,j):= \sum_{\ell=1}^d P_n^\ell(i,j)=\sum_{\ell=1}^d \bI(\pi^\ell_n(i)=j).
\end{equation}
Note that $S_n^d$ can be viewed as the adjacency matrix of a $d$-regular
directed multigraph.

For two sequences of positive reals $\{a_n\}$ and $\{b_n\}$
we say that $a_n= O(b_n)$ (or $a_n=o(b_n)$)
if for some universal constant $C$,
$\limsup_{n \ra \infty} a_n/b_n \le C$ (respectively, $=0$).
We say that $a_n=\omega(b_n)$ if $b_n=o(a_n)$.
The main result of this paper is the following theorem.
\begin{thm}\label{thm:main}
If $ \log^{12}n/(\log \log n)^{4} \le d=O(n)$
then the \abbr{ESD} of
$S_n^d/\sqrt{d}$ converges weakly to the uniform distribution
on the unit disk in the complex plane, in probability, as $n \ra \infty$.
\end{thm}
We refer to this result as the weak circular law for sums of permutations.

\begin{rmk}
 One expects the conclusion of Theorem \ref{thm:main} to hold almost surely. However,
 the estimate on the smallest singular value of $S_n^d/\sqrt{d}-zI$  contained
 in Theorem \ref{thm:ssv0} below is not sharp enough to allow for the application
 of the Borel--Cantelli lemma. On the other hand, other estimates in the paper, and in particular the
 concentration inequalities and the estimates on moderately small singular values, see Section
 \ref{sec:prelim_outline} for definitions, are
 not an obstacle to the application of Borel--Cantelli.
\end{rmk}

\begin{rmk}
  \label{remarkomega1}
Theorem \ref{thm:main} is established for $d\ge \log^{12}n/(\log \log n)^{4}$.
One expects its conclusion to hold as soon as
$d =\omega(1)$.
Obvious obstacles to proving this by our methods are that the minimal singular value estimate, Theorem \ref{thm:ssv0} below, requires $d=\omega(\log^8n)$ to be useful, and our loop equations main theorem, Theorem
\ref{thm:smallish_sing_control}, is only effective
when $d$ grows like a power of
$\log n$.
Proving Theorem
\ref{thm:main} for $d=\omega(1)$ remains a major challenge and seems to require new ideas. It is possible that one could use
the methods of \cite{LLTTY}
to relax the assumptions in Theorem \ref{thm:ssv0} to $d=\omega(1)$.
\end{rmk}

\subsection{Background: \abbr{ESD}'s for non-normal matrices}
The study of the \abbr{ESD} for random Hermitian matrices can be traced
back to Wigner
\cite{Wigner55, wigner1958distribution} who showed that the \abbr{ESD}'s of
$n \times n$ Hermitian matrices
with i.i.d.~centered entries of variance
$1/n$ (upper diagonal) satisfying appropriate moment bounds (e.g.,
Gaussian) converge to the semicircle distribution. The conditions on
finiteness of moments were removed in subsequent work, see e.g.~
\cite{bai2010spectral,pastur1972spectrum} and the references therein.
We refer to the texts \cite{mehta1967statistical, forrester2010log, tao2012topics, agz, bai2010spectral} for further background
and a historical perspective.

Wigner's proof employed the method of moments: one notes that the moments
of the semicircle law determine it, and then one computes by combinatorial
means
the expectation (and variance) of the trace of powers of the matrix.
This method (as well as related methods based
on evaluating the Stieltjes transform of the \abbr{ESD})
fails for non-normal matrices since moments do not determine
the \abbr{ESD}.

An analogue of Wigner's semicircle law
in the non-normal regime is the following {\em circular law} theorem:

\subsection*{Circular law} Let $M_n$ be an $n \times n$ matrix with i.i.d.~entries of zero mean and unit variance. Then the \abbr{ESD} of $M_n/\sqrt{n}$ converges to the uniform distribution on the unit disk on the complex plane.

\vskip10pt

The circular law was posed as a conjecture
based on numerical evidence in the $1950$'s. \corN{For the case that the
entries have a complex Gaussian distribution it can be derived from
Ginibre's explicit formula for the joint density function of the eigenvalues
\cite{Ginibre65,mehta1967statistical}. }The case of real Gaussian
entries, where a similar formula is available, was settled by
Edelman \cite{edelman1988eigenvalues}.
For the general case when there is no such formula,
the problem remained open for a very long time. An approach to
the problem, which eventually
played an important role in the resolution of the conjecture,
was suggested by Girko in the 1980's \cite{girko1984circular}, but mathematically it contained
significant gaps. The first non-Gaussian
case (assuming existence of density for the entries)
was rigorously treated by Bai \cite{bai1997circular}, and
after a series of partial results
(see \cite{bordenave2012around} and the references therein),
the circular law conjecture was established in its full generality
in the seminal work of Tao and Vu \cite{tao_vu_2010}.
 \begin{thm}[Circular law for i.i.d.~entries {\cite[Theorem 1.10]{tao_vu_2010}}]
 Let $M_n$ be an $n \times n$ random matrix whose entries are 
 i.i.d.\ copies of a fixed (independent of $n$) complex random variable $x$ 
 with zero mean and unit variance.
 Then the \abbr{ESD} of $\f{1}{\sqrt{n}}M_n$ converges weakly to the uniform distribution on the unit disk on the complex plane,
both in probability and in the almost sure sense.
\label{thm:circ_iid}
 \end{thm}
 
 \vskip5pt
A remarkable feature of Theorem \ref{thm:circ_iid} is its {\em universality}:
the asymptotic behavior of the \abbr{ESD} is insensitive
to the specific details of the entry distributions as long as
they are i.i.d.~and have zero mean and unit variance.  It also extends to the sparse set-up. Namely consider a matrix of i.i.d.~entries where each entry is the product of a zero mean and unit variance random variable, and an independent Bernoulli$(p)$ random variable.
From the two concurrent works of G{\"o}tze and Tikhomirov
\cite{gotze2010circular} and Tao and Vu \cite{tao2008random} it follows
that if $p$ decays polynomially in $n$, i.e.~$p \ge n^{\vep-1}$
for some $\vep>0$, then the limit is still the circular law.
Later Wood \cite{wood2012universality} relaxed the moment
assumptions of the entries. A recent article by Basak and
Rudelson \cite{b_rudelson_spase_circ}
shows that the same limit continues to hold when $p$ decays at a poly-logarithmic rate. 
In all these works the entries of the matrix
still enjoys the independence and this feature plays a key role in the proofs.
In \cite{bordenave2014spectrum}, Bordenave, Caputo and Chafa\"i studied random Markov generators where one puts i.i.d.~entries on the off-diagonal positions and sets each diagonal to be the negative of the corresponding row-sum, showing that the limit law is a {\em free additive convolution} of the circular law and a Gaussian random variable. Their result covers sparse ensembles, including the Markov generator for a directed Erd\H{o}s--R\'enyi graph with edge probability $p(n)=\omega(n^{-1}\log^6n)$.

Circular laws for matrices with less independence between
entries were subsequently proved in
\cite{bordenave2012circular},
\cite{adamczak2011marchenko},
\cite{nguyen2013circular},
\cite{adamczak2016circular}, and
\cite{nguyen2014random}. In particular, in \cite{nguyen2014random} Nguyen showed
that the \abbr{ESD} of
a uniformly chosen random {\em doubly stochastic matrices} converges weakly
to the circular law. Since the adjacency matrix of a random $d$-regular
directed graph (digraph) is a random doubly
stochastic matrix, one is naturally led to the
question of establishing the limits of the \abbr{ESD} for such matrices.
This was addressed in recent work of the second author \cite{cook-rrd}, where it was shown that the circular law holds for adjacency matrices of random regular digraphs assuming a poly-log($n$) lower bound on the degree.

A completely different story emerges when one replaces the
Ginibre matrices by other models whose distribution is
invariant under
the action of some large group (note that Ginibre matrices are indeed invariant under right or left multipliction by unitary matrices).
The study of such invariant models was initiated by
Feinberg and Zee \cite{feinberg1997non}, who evaluated non-rigorously
the limit of the \abbr{ESD} for
such matrices and showed various properties of the limit,
e.g.~that it is supported on a single ring in the complex plane.
By using a variant of Girko's method adapted to the unitary
group, this was put on a rigorous basis  by Guionnet, Krishnapur and Zeitouni
\cite{gkz}, who
evaluated the limit of the \abbr{ESD} for
a matrix of the form $UD$ where $D$ is diagonal satisfying some assumptions
and $U$ is a random
Haar-distributed unitary, and showed that it coincides
with the Brown measure of the associated limiting operators
(an improved version appears in
\cite{rudelson2014invertibility}).
Building on this and closer to the topic of this paper,
Basak and Dembo
\cite{b_dembo_unitary}
showed
that the \abbr{ESD}s of the sum $\hat U_n^d$ of $d$ i.i.d.~Haar
distributed Unitary/Orthogonal matrices converge to a probability measure
$\mu_d$ whose density with respect to Lebesgue measure is given
by
\begin{equation}
  \label{eq-nickofer}
f_d(z):=\f{1}{\pi}\f{d^2(d-1)}{(d^2- |z|^2)^2}\bI(|z| \le \sqrt{d}),
\end{equation}
which coincides with the Brown measure of a sum of $d$ free Haar unitaries.
Note that from this
one easily concludes
the existence of a sequence $d=d(n)$ so that
the \abbr{ESD} of $\hat U_n^{d(n)}/\sqrt{d(n)}$ converges to the circular law.

We finally get to our model: it sits at the intersection of sparse
models of regular directed (multi)-graphs and the sum of unitaries treated
in \cite{b_dembo_unitary}. Indeed, from the point of view of the latter
we replace unitary matrices
which are Haar-distributed on the full unitary group by
unitaries which are Haar-distributed on the subgroup of permutation matrices.
In this case a formal application of Girko's method
leads one to expect convergence to
$\mu_d$ (if $d$ is fixed, see e.g. \cite{bordenave2012around})
or to the circular law when \corO{$d=\omega(1)$}
(after rescaling by $\sqrt{d}$).
 The goal of this paper is to establish that
\corO{the latter indeed holds, at least when $d$ does not grow
too slowly or too rapidly.}

\begin{rmk}
  Our methods are not sharp enough to handle the case of $d$ constant,
  both for the reasons mentioned in Remark \ref{remarkomega1} and the fact that
  the loop equations for fixed $d$ are much more complicated. See however
  the recent work \cite{BHY:km} for progress in this direction for random
$d$-regular graphs of sufficiently large fixed degree.
\end{rmk}

We end this section by pointing out that for \emph{fixed} $d$, the random regular digraph model considered in \cite{cook-rrd} is contiguous with the sum of permutations model conditioned to have no parallel edges (i.e.\ with the matrix conditioned to have no entries larger than 1, an event which occurs with positive probability) \cite{MRRW, Janson:contiguity}. 
However, we are unaware of any quantitative contiguity results that allow $d$ to grow with $n$.
Given such a result (allowing $d$ to grow faster than $\log^{12}n$) it could be possible to deduce the main result of \cite{cook-rrd} from Theorem \ref{thm:main}, for some range of $d$; however, this would require a quantitative version of Theorem \ref{thm:main} with failure probability smaller than the probability for the sum of permutations to yield a 0/1 matrix, which is of order $\exp(-cd^2)$.

\subsection{Outline of the paper}
In Section \ref{sec:prelim_outline} we provide a brief outline of the proof techniques of Theorem \ref{thm:main}.
  We begin Section \ref{sec:prelim_outline} by a short description of
  Girko's method, which in a nutshell consists of focusing attention
  on the \textit{logarithmic potential} of the \abbr{ESD} of $S_n^d/\sqrt{d}$.
  This is done by analyzing
  the Hermitian matrix
  $T_n(z)=(z-S_n^d/\sqrt{d})^*(z-S_n^d/\sqrt{d})$
  with $z\in \C$ \corAa{(hereafter, for any $n \times n$ matrix $B_n$ and $z \in \C$, for brevity, we often write $z- B_n$ to denote $zI_n -B_n$)}.
  To implement Girko's method one requires good control
  on the smallest singular value of $T_n(z)$ as well as on its
  {\em smallish}
  singular values.
  The required control on the smallest
  singular value is derived in Theorem \ref{thm:ssv} and
  an outline of its proof can be found in Section
  \ref{sec:smallest_sing_outline}. The desired control on the
  smallish singular values is obtained in
  Theorem \ref{thm:smallish_sing_control} by controlling the
  difference of the {\em Stieltjes transform} of the \abbr{ESD} of $T_n^{1/2}(z)$
  at the finite $n$ level and
  at the putative limit, all  the way up to (almost)
  the real line.
  An outline of the proof of Theorem \ref{thm:smallish_sing_control} is
  given in Section \ref{sec:stieltjes_outline}.

For Theorem \ref{thm:ssv}, to control  the smallest singular value of a matrix $A_n$ we need to control the infimum of $\|A_n u\|_2$ over all $u$ in the unit sphere. To this end, we break the sphere into the set of ``flat" vectors and its complement, where a vector is said to be flat if it is close in $\ell_2$ norm to a vector with a large number of equal components (for a precise formulation see Definition \ref{dfn:mzafv}). The infimum over flat vectors is taken care of in Section \ref{sec:structured} and the infimum over
the remaining vectors is treated in Section \ref{sec:unstructured}.

  Section \ref{sec-traces}
  and Section \ref{sec:conc_ineq} are devoted to control certain traces of
  polynomials in
$S_n^d$ and to derive concentration inequalities for Lipschitz functions of sum of permutations, respectively.
We then turn to the
control on the Stieltjes transform of the \abbr{ESD} of $T_n^{1/2}(z)$.
  In Section \ref{sec:sub-loop-eqn}
we show that the Stieltjes transform satisfies
an (approximate) fixed point equation, first in expectation and then,
using the concentration results of Section \ref{sec:conc_ineq}, also with high probability.
In Section \ref{sec:prf-thm-26} we then finish the proof
of Theorem  \ref{thm:smallish_sing_control} using the stability of the fixed point equation, apriori lower bound on Stieltjes transform of the \abbr{ESD} of $T_n^{1/2}(z)$ far away from the real line, and
a bootstrap argument.

Finally in Section \ref{sec:proof_thmmain} combining Theorem \ref{thm:ssv}, Theorem \ref{thm:smallish_sing_control}, and using a {\em replacement principle} (see Lemma \ref{lem:replacement}) we finish the proof of Theorem \ref{thm:main}.

\subsection{Notational conventions}
We write $\C^J$ for the subspace of vectors in $\C^n$ 
supported on $J\subset[n]$, and write $\mathbb{B}^J, \mathbb{S}^J$ for the 
closed Euclidean 
unit ball and sphere in this subspace. \corO{If $J=[n]$, we write $\mathbb{B}^n, \mathbb{S}^{n-1}$ for
brevity.}
Given $v\in \C^n$ and $J\subset[n]$, $v_J$ denotes the projection of $v$ to $\C^J$.
$\un=\un_n$ denotes the $n$-dimensional vector with all components equal to one, and consequently $\un_J$ denotes the vector with $j$th component equal to 1 for $j\in J$ and 0 otherwise.
For $x,y\in \R$ we sometimes write $x\wedge y$ to mean $\min(x,y)$.

\section{Preliminaries and proof outline}
\label{sec:prelim_outline}
\subsection{Proof overview}

In this section we provide an outline of the proof of Theorem \ref{thm:main}. As we go along we introduce necessary definitions and notation.

The standard technique to analyze the asymptotics of the \abbr{ESD}
of a non-normal matrix is Girko's method \cite{girko1984circular}.
The basis of this method is the following identity which is a consequence
of
Green's theorem:  for any polynomial
$P(z)= \prod_{i=1}^n (z - \lambda_i)$ and any test function
$\psi \in C_c^2(\C)$,
\beq
\sum_{j=1}^n \psi(\lambda_j) =
\f{1}{2 \pi} \int_\C \Delta \psi(z) \log |P(z)| d\gm(z),
\notag
\eeq
where $\gm$ is the Lebesgue measure on $\C$ and $\Delta$
denotes the two-dimensional Laplacian.
Applying this identity with the characteristic polynomial
$P(\cdot)$ of a matrix $M_n$ yields
\begin{align}
\int_\C \psi(z) dL_{M_n}(z) =  &
\f{1}{2 \pi n} \int_\C \Delta \psi(z) \log | \det  (z I_n-M_n)| d\gm(z)\label{eq:greens_theorem}\\
= & \f{1}{4 \pi n} \int_\C \Delta \psi(z) \log
\det  [(z I_n-M_n)(z I_n-M_n)^*] d\gm(z).\notag
\end{align}
Next, associate with any $n$-dimensional
non-Hermitian matrix $M_n$ and every $z \in \C$
the $2n$-dimensional {Hermitian} matrix
\beq\label{eq:herm-def}
{\bm M}_n^z:=
 \begin{bmatrix}
  0 & (z I_n-M_n) \\
  (z I_n-M_n)^* & 0
 \end{bmatrix} \,.
\eeq
The eigenvalues of ${\bm M}_n^z$ are merely
$\pm 1$ times the singular values of $z I_n- M_n$. Therefore,
denoting by $\nu_n^z$ the \abbr{ESD} of ${\bm M}_n^z$, we have that
\beq
\f{1}{n} \log \det[(z I_n-M_n)(z I_n-M_n)^*] = \f{1}{n} \log |\det {\bm M}_n^z|
= 2 \langle \Log, \nu_n^z \rangle \, , \notag
\eeq
where for any probability measure $\mu$ on $\R$, $\langle \Log, \mu \rangle := \int_\R \log |x| d\mu(x)$. Therefore we have the following  key identity
\beq
\int_\C \psi(z)dL_{M_n}(z) = \f{1}{2 \pi}
\int_\C \Delta \psi(z) \langle \Log, \nu_n^z \rangle d\gm(z).
\label{eq:girko_key_identity}
\eeq
The utility of Eqn.~(\ref{eq:girko_key_identity}) lies in the following
general recipe for proving convergence of $L_{M_n}$ of a
given family of non-Hermitian random matrices $\{M_n\}$:
\vskip5 pt

\noindent
 {\bf Step 1}: Show that for (Lebesgue almost)
every $z \in \C$, as $n \ra \infty$, the measures
$\nu_n^z$ converge weakly, in probability, to
some measure $\nu^z$.

\vskip5pt

\noindent
{\bf Step 2}: Justify that
$\langle \Log, \nu_n^z \rangle \to \langle \Log, \nu^z \rangle$
in probability.

\vskip5pt

\noindent
{\bf Step 3}: A uniform integrability argument allows one to convert
the $z$-a.e. convergence of $\langle \Log, \nu_n^z \rangle$ to
the convergence of $\int_\C \Delta \psi(z)\langle \Log, \nu_n^z \rangle d\gm(z)$, for a suitable collection
$\cS \subseteq C_c^2(\C)$ of (smooth) test functions $\psi$.
Consequently, it then follows from (\ref{eq:girko_key_identity}) that
for each fixed, non-random $\psi \in \cS$,
\beq
\int_\C \psi(z) dL_{M_n}(z) \ra \f{1}{2 \pi} \int_\C \Delta \psi(z)
\langle \Log, \nu^z \rangle
d\gm(z)\, ,
\label{eq:step3}
\eeq
in probability.

\vskip5pt
\noindent
{\bf Step 4}: Upon checking that $f(z) := \langle \Log, \nu^z \rangle$
is smooth enough to justify the integration by parts, one has that
for each fixed, non-random $\psi \in \cS$,
\beq
\label{eq:step4}
\int_\C \psi(z) dL_{M_n}(z)  \ra \f{1}{2 \pi}\int_\C
\psi(z) \Delta f(z)
d\gm(z) \,,
\eeq
in probability. For $\cS$ large enough, this implies the weak convergence of the \abbr{ESD}s
$L_{M_n}$ to a limit which has the density $\frac{1}{2\pi} \Delta f$
with respect to Lebesgue measure on $\C$, in probability.

\vskip10pt

To prove Theorem \ref{thm:main} our plan is to establish
{\bf Steps 1--4} for $M_n=S_n^d/\sqrt{d}$.
As has been the case for other models of random matrices,
{\bf Step 2} is the most challenging part.
Since $\nu_z$ is the \abbr{ESD} of a Hermitian matrix one can use tools such as the method of moments or the Stieltjes transform to deduce {\bf Step 1}. However $\log (\cdot)$ being unbounded both near zero and infinity the conclusion of {\bf Step 1} is not enough to establish {\bf Step 2}. One needs additional control on the large as well as small singular values of $S_n^d/\sqrt{d}-z$. To this end, we first note that the limit of the \abbr{ESD} of $S_n^d/\sqrt{d}$, the circular law, is compactly supported. Therefore one can actually check that establishing {\bf Steps 1--4} for $z$ in a large ball in the complex plane is enough to  complete the proof of Theorem \ref{thm:main}.

Next note that each row-sum and column-sum of $S_n^d$ is $d$
and hence the maximal singular value of $S_n^d/\sqrt{d}-z$ is $O(\sqrt{d})$
for any $z$ in a large ball.
One can also easily show that
the trace of $S_n^d (S_n^d)^*/nd$ is bounded with high probability
(see Section
\ref{sec-traces}), which can be used to show that $\nu_n^z$ integrates $x^2$, and hence $\log(x)$, 
near infinity.

Most of this paper is devoted to
obtaining bounds on the small singular values of $S_n^d/\sqrt{d}-z$.
First, one needs to have a lower bound on the smallest singular value.
This is derived in Theorem \ref{thm:ssv0}.  The idea behind the proof of
Theorem \ref{thm:ssv0} is outlined in Section \ref{sec:smallest_sing_outline}.

Next we need to show that there are not too many singular values near zero. Equivalently, we need to show that the total mass
of a small interval $I$ around zero under the \abbr{ESD} of ${\bm M}_n^z$
is not too large.
That mass can be estimated
by obtaining bounds on the Stieltjes transform of
the \abbr{ESD} at a distance from the real line which is
comeasurate with the length of $I$
(for example, see Lemma \ref{lem:stieltjes_bound}).
In Section \ref{sec:stieltjes_outline} we provide an outline
on how to achieve the desired bounds on the Stieltjes
transform of ${\bm M}_n^z$ (see Theorem \ref{thm:smallish_sing_control}).

\subsection{Control on the smallest singular value}
\label{sec:smallest_sing_outline}

The following result provides the required lower bound on the smallest singular value of $\frac1{\sqrt{d}}S_n^d-z$.

\begin{thm}	\label{thm:ssv0}
Fix any $R>0$ and let $z \in B_\C(0,R)\corAa{:=\{z' \in \C: |z'| \le R\}}$. Assume $1\le d\le n^{100}$.
There exists $C_{\ref{thm:ssv0}}<\infty$ depending only on $R$ and an absolute constant $\ol{C}_{\ref{thm:ssv0}}>0$ such that
\begin{equation}	\label{main:bound0}
\P\left\{ s_n\Big(\frac{1}{\sqrt{d}}S_n^d-z I_n\Big) \le n^{-\ol{C}_{\ref{thm:ssv0}} \log_{{d\wedge n}}n}\right\}  \le C_{\ref{thm:ssv0}} \f{\log^4 n}{\sqrt{{d\wedge n}}},
\end{equation}
\corAa{where $s_n(\cdot)$ denotes the smallest singular value.}
\end{thm}

\vskip10pt
We deduce Theorem \ref{thm:ssv0} from the following more general result. First we introduce some notation. For  an $n \times n$ matrix $M_n$ we write
\begin{equation}	\label{def:Wnorm}
\|M_n\|_{\oneperp} := \sup_{u\in \sph \cap \oneperp} \|M_n u\|_2,
\end{equation}
where we recall $\sph:=\{u \in \C^n: \|u\|_2=1\}$ and $\|\cdot\|_2$ denotes $\ell_2$ norm.

\begin{thm}	\label{thm:ssv}
Fix an arbitrary $\gamma_0\ge 1$.
Let $1\le d\le n^{\gamma_0}$, and let $Z_n$ be a deterministic $n\times n$ matrix such that $\|Z_n\|_{\oneperp}\le n^{\gamma_0}$ and $Z_n\un=\zeta\un$, $Z_n^*\un = \overline{\zeta}\un$ for some $\zeta \in \C$. There exists $C_{\ref{thm:ssv}}<\infty$ depending only on $\gamma_0$ and an absolute constant $\ol{C}_{\ref{thm:ssv}}<\infty$ such that
\begin{equation}	\label{main:bound}
\P\left\{ s_n(S_n^d+Z_n) < n^{-\ol{C}_{\ref{thm:ssv}}\gamma_0\log_{{d\wedge n}} n}\wedge |d+\zeta| \right\} \le C_{\ref{thm:ssv}} \frac{\log^{4}n}{\sqrt{{d\wedge n}}}.
\end{equation}	
\end{thm}

\vskip10pt
By taking $Z_n= \sqrt{d} z I_n$ we immediately deduce Theorem \ref{thm:ssv0} from Theorem \ref{thm:ssv}.

\begin{rmk}	\label{rmk:reduced}
In the proof of Theorem \ref{thm:ssv} it will be convenient to assume $d\le n$. 
We now show how to reduce to this case (in fact we could reduce assuming $d\le c_0n$ for any fixed constant $c_0>0$). 
Suppose $d>n$, and let
\[
Z_n'= Z_n + S_n^d - S_n^n.
\]
Condition on $\pi^{n+1}_n,\dots, \pi^{d}_n$ to fix $Z_n'$.
Then we have
\begin{itemize}
\item $Z'_n \un = (\zeta + d-n)\un=: \zeta'\un$,
\item $(Z'_n)^* \un= \ol{\zeta'} \un$,
\item $|n+\zeta'|= |d+\zeta|$,
\item $\|Z'_n\|_\oneperp \le \|Z_n\|_\oneperp + \|S_n^d-S_n^{n}\|\le n^{\gamma_0}+d\le 2n^{\gamma_0}$.
\end{itemize}
Thus, after modifying $\gamma_0$ slightly, we see that it is enough to prove Theorem \ref{thm:ssv} under the additional assumption that $d\le n$.
\end{rmk}

\vskip10pt

On a high level, the proof of Theorem \ref{thm:ssv} follows the general strategy of the recent work \cite{cook-rrd} of the second author, which establishes a similar result with $S_n^d$ replaced by a uniform random 0--1 matrix constrained to have all row and column sums equal to $d$. We now motivate some of the main ideas of this strategy.

From the definition of the smallest singular value we have
\begin{equation}	\label{ssv:inf}
s_{n}(S_n^d +Z_n)= \inf_{u \in \sph} \norm{(S_n^d+Z_n)u}_2.
\end{equation}
We note that $\un$ is an eigenvector of $(S_n^d+Z_n)^*(S_n^d+Z_n)$ with eigenvalue $|d+\zeta|^2$. A short argument then shows that to obtain \eqref{main:bound} it suffices to control the infimum of $\norm{(S_n^d+Z_n)u}_2$ for $u \in \sph \cap \oneperp$.
Denoting the rows of $S_n+Z_n$ by $R_1,\dots, R_n$, we have
\[
(S_n^d+Z_n)u = (R_1\cdot u , \dots, R_n\cdot u).
\]
Thus, for a fixed vector $u\in \sph\cap \oneperp$, the task of controlling the probability that $(S_n^d+Z_n)u$ concentrates near the origin will involve bounding the probability that the scalar random variables $R_i\cdot u$ concentrate near zero.

First we briefly review the argument from \cite{RuVe} for the case where $S_n^d$ is replaced by a matrix $X_n$ with i.i.d.\ centered entries $\xi_{ij}$ of unit variance. In this case we have
\begin{equation}	\label{walk:iid}
R_i\cdot u = w + \sum_{j=1}^n \xi_{ij}u_j,
\end{equation}
$w\in \C$ is a deterministic quantity involving the entries of $u$ and $Z_n$. Then we can bound $\P(|R_i\cdot u|\le t)$ for small $t>0$ using standard anti-concentration estimates. For instance, we have the following Berry--Ess\'een-type bound (see Lemma \ref{lem:BE}): for fixed nonzero $v\in \C^n$ and any $r\ge 0$,
\begin{equation}	\label{be:above}
\sup_{z\in \C} \P\bigg( \bigg| z+ \sum_{j=1}^n \xi_j v_j \bigg| \le r\bigg) =O\left(\frac{r+ \|v\|_\infty}{\|v\|_2}\right).
\end{equation}
For this bound to be effective when applied to $u$, we need $u$ to be ``spread" in the sense that there is
a set $J\subset[n]$ with $|J|\ge cn$ such that $|u_j|\sim 1/\sqrt{n}$ for all $j\in J$. After conditioning on the variables $\xi_{ij}$ with $j\notin J$, \eqref{be:above} gives
\begin{equation}	\label{spreadbound}
\P( |R_i\cdot u| \le t ) = O\left( t + \frac1{\sqrt{n}}\right).
\end{equation}
This motivates partitioning the sphere into \emph{compressible} and \emph{incompressible} vectors, which we now define.
Denote $\supp(v) := \{j\in [n]:v_j\ne 0\}$, and for $m\in[n]$ define the set of \emph{$m$-sparse vectors}
\begin{equation}	\label{def:sparse}
\Sparse(m) := \big\{ v\in \C^n: |\supp(v)|\le m\big\}.
\end{equation}
For $m\in [n]$ and $\rho>0$, the set of $(m,\rho)$-\emph{compressible} unit vectors is defined to be the $\rho$-neighborhood of the set of $m$-sparse vectors in the sphere:
\[
\Comp(m,\rho) := \sph \cap \big( \Sparse(m) + \rho\uball \big).
\]
For $m\ge cn$ and $\rho$ of constant order, one can show that \emph{incompressible} vectors $u\in \sph\setminus \Comp(m,\rho)$ are spread in the above sense, \emph{i.e.}\ $|u_j|\sim 1/\sqrt{n}$ for $\ge c' n$ elements $j\in [n]$ for some constant $c'>0$. Thus, \eqref{spreadbound} is effective for incompressible vectors.
While we only have a crude anti-concentration bound for compressible vectors, the bound can be \emph{tensorized} to show $\P(\|Mu\|_2\le c\sqrt{n})\le e^{-cn}$ for any fixed compressible vector $u$. Then, from the fact that $\Comp(m,\rho)$ has low metric entropy (i.e.\ it can be covered by a relatively small number of small balls) one can apply the union bound over a suitable net to show $\inf_{u\in \Comp(c_1n,c_2)}\|(X_n+Z_n)u\|_2\ge c'\sqrt{n}$ with high probability if $c_1,c_2$ are sufficiently small constants.

After obtaining uniform control on $\|(X_n+Z_n)u\|_2$ for $u\in \Comp(c_1n,c_2)$, an averaging argument shows that in order to obtain an estimate of the form
\[
\P\big( s_n(X_n+Z_n) \le t/\sqrt{n}\big) =O(t) + o(1),
\]
it suffices to get a bound of the form $\P(|R_i\cdot u| \le t) = O(t)+o(1)$ for an arbitrary \emph{fixed} row $R_i$ and $u\in \sph\setminus\Comp(c_1n,c_2)$. But this now follows from \eqref{spreadbound}.
See \cite{RuVe} for the detailed presentation of this argument.

The distribution of $S_n^d$ necessitates a somewhat modified approach, and in particular a different notion of structure than compressibility.
In order to make use of the anti-concentration estimate \eqref{be:above} we will consider \emph{pairs} of rows $R_{i_1},R_{i_2}$. For each $\ell\in [d]$, conditioning on the remaining $n-2$ rows of $P_n^\ell$ fixes $\pi_n^\ell(\{i_1,i_2\})$. It follows that the $i_1$-st row of $P_n^\ell$ is $e_{\bs j}$ where $\bs j$ is drawn uniformly from $\pi_n^\ell(\{i_1,i_2\})$, and $e_k$ denotes the $k$-th standard basis vector.
Since the matrices $\{P_n^\ell\}_{\ell \in [d]}$ are independent, it is then possible to express
\begin{equation}	\label{walk:pairs}
R_{i_1}\cdot u = w + \sum_{\ell=1}^d \xi_\ell (u_{\pi_n^\ell(i_1)}- u_{\pi_n^\ell(i_2)})
\end{equation}
where $\{\xi_\ell\}_{\ell\in [d]}$ are i.i.d.\ Rademacher variables and $w\in \C$ is some quantity that is deterministic under conditioning on the rows $[n]\setminus \{i_1,i_2\}$ of all of the matrices $\{P_n^\ell\}_{\ell\in [d]}$.
By the discussion under \eqref{walk:iid}, we can then get a bound on $\P(|R_{i_1}\cdot u|\le t)$ for small $t>0$ via the Berry--Ess\'een-type bound \eqref{be:above}, which will be effective when the vector of \emph{differences} $(u_{\pi_n^\ell(i_1)}- u_{\pi_n^\ell(i_2)})_{\ell\in [d]}$ is spread. This motivates the following:

\begin{dfn}\label{dfn:mzafv}
For $m\in [n]$ and $\rho\in (0,1)$, define the set of \emph{$(m,\rho)$-flat vectors}
\begin{align}
\Flat(m,\rho) &:= \sph\cap \left(  \rho \uball + \bigcup_{\lambda\in \C} \big(\lambda\un+ \Sparse(m)\big) \right) \notag\\
&= \big\{ u\in \sph:  \exists\, v\in \Sparse(m),\, \lambda\in \C \text{ with } \|u-v-\lambda \un\|_2 \le \rho \big\}	\label{def:Aflat}
\end{align}
(where the set $\Sparse(m)$ was defined in \eqref{def:sparse}).
We denote the mean-zero flat vectors by
\begin{equation}	\label{def:Aflat0}
\Flat_0(m,\rho) :=  \Flat(m,\rho)\cap \langle \un\rangle^\perp.
\end{equation}
For non-integral $x\ge0$ we will sometimes abuse notation and write $\Sparse(x)$, $\Flat(x,\rho)$, etc.~to mean $\Sparse(\lf x\rf)$, $\Flat(\lf x\rf,\rho)$.
\end{dfn}

\vskip10pt

Our first task is get a lower bound on $\inf_{u\in \Flat_0(m,\rho)}\|(S_n^d + Z_n)u\|_2$ holding with high probability for a suitable choice of $m,\rho$, which we obtain in Proposition \ref{prop:struct} below.
For a parameter $K\ge 1$
define the \emph{boundedness event}
\begin{equation}	\label{def:bdd}
\cB(K) := \Big\{ \, \|S_n^d+Z_n\|_\oneperp \le K\sqrt{d}\,\Big\}
\end{equation}
(recall our notation \eqref{def:Wnorm}). We will eventually take $K=n^{\gamma_0}$ for an arbitrary fixed $\gamma_0\ge1$ (cf.~Section \ref{sec:ssv_conclude}).
For $m\in [n]$ and $\rho\in (0,1)$
(possibly depending on $n$),
define the event
\begin{equation}	\label{def:eventK}
\cE_K(m,\rho):=  \cB(K) \cap \Big\{ \, \exists u \in \Flat_0(m,\rho): \, \|(S_n^d+Z_n)u\|_2\le \rho K\sqrt{d}\, \Big\}.
\end{equation}

\begin{prop}[Invertibility over flat vectors]	\label{prop:struct}
There exist absolute constants $C_{\ref{prop:struct}},c_{\ref{prop:struct}},\ol{c}_{\ref{prop:struct}}>0$ such that the following holds.
Let $\gamma\ge1$ and $1\le K\le n^\gamma$. Assume $\log^3n\le d\le n$. Then 
\begin{equation}\label{eq:struct}
\P\left\{ \cE_K\left( \frac{c_{\ref{prop:struct}}n}{\gamma \log^2n}, \, n^{-C_{\ref{prop:struct}}\gamma \log_dn)} \right) \right\} \le e^{-\ol{c}_{\ref{prop:struct}}d}
\end{equation}
for all $n$ sufficiently large depending on $\gamma$.
\end{prop}
Section \ref{sec:structured} is devoted to the proof of Proposition \ref{prop:struct}, and we defer discussion of the proof ideas to that section.

The remainder of the proof of Theorem \ref{thm:ssv} is given in Section \ref{sec:unstructured}.
Having obtained control on flat vectors, our aim will then be to reduce the problem to obtaining an anti-concentration estimate on $R_{i_1}\cdot u$, which we express as in \eqref{walk:pairs}, for a fixed row $R_{i_1}$ and fixed $u\in \sph \cap \oneperp\cap \Flat(m,\rho)^c$.  (Actually we will consider dot products of the form $(R_{i_1}-R_{i_2})\cdot u$, but these can also be expressed in the form \eqref{walk:pairs}.)
As in the i.i.d.\ setting discussed above this can be accomplished by an averaging argument, but the argument here is more delicate due to the dependencies among the entries of $S_n^d$. We adapt an approach used in \cite{LLTTY} for the invertibility problem for random regular digraphs.
The vector $u$ must be chosen to be almost-orthogonal to the span of rows $\{R_i: i\notin \{i_1,i_2\}\}$, and we want to ensure that the differences $u_{\pi_n^\ell(i_1)}- u_{\pi_n^\ell(i_2)}$ are large for a large number of $\ell\in [d]$. If the indices $\pi_n^\ell(i_1),\pi_n^\ell(i_2)$ were independent of $u$ then it would be relatively easy to show that because $u$ is non-flat, a random choice of $i_1,i_2$ will give us a large number of differences, on average. However, since both $u$ and $\pi_n^\ell(i_1),\pi_n^\ell(i_2)$ are fixed by conditioning on $\{\pi_n^\ell(i): i\in [n]\setminus \{i_1,i_2\},\ell\in [d]\}$ the argument requires some care.
See Lemma \ref{lem:avg} for the details.

Having reduced to consideration of a random walk of the form \eqref{walk:pairs} with a large number of large differences $u_{\pi_n^\ell(i_1)}- u_{\pi_n^\ell(i_2)}$, we can conclude using the Berry--Ess\'een-type bound \eqref{be:above}; this is done in Lemma \ref{lem:walk}. In Section \ref{sec:ssv_conclude} we combine all of these elements to complete the proof of Theorem \ref{thm:ssv}.

\subsection{Control on the Stieltjes transform}
\label{sec:stieltjes_outline}
We begin this section by fixing some notation.
Denote $\C^+:=\{\xi \in \C: \Im \xi >0\}$. 
Fixing any $z \in B_\C(0,R)$, for some $R >0$, and $\xi \in \C^+$ we define the Green function as follows:
\[
G(S_n^d):=G(S_n^d,\xi,z):= \left(\xi  - \left(z-\f{S_n^d}{\sqrt{d}}\right) \left(z -\f{S_n^d}{\sqrt{d}}\right)^*\right)^{-1}.
\]
Instead of working with the Green function $G_n(\cdot)$,
we will see that it will be easier to work with its symmetrized version
\beq\label{eq:wt-G-def}
\wt{G}(S_n^d):=\wt{G}(S_n^d,\xi,z):= \left[\xi - \begin{bmatrix} 0 & \left(z-\f{S_n^d}{\sqrt{d}}\right)\\ \left(z- \f{S_n^d}{\sqrt{d}}\right)^* & 0 \end{bmatrix}\right]^{-1}.
\eeq
We next define the Stieltjes transform of \corA{the \abbr{ESD} of} $(z-S_n^d/\sqrt{d})(z-S_n^d/\sqrt{d})^*$ and its symmetrized version,
\[
m_n(\xi) :=m_n(\xi,z):=\f{1}{n} \Tr \, G(S_n^d,\xi,z), \quad \wt{m}_n(\xi):=\f{1}{2n}\Tr \,\wt{G}(S_n^d,\xi,z).
\]

Recall that the eigenvalues of the matrix
\beq\label{eq:bmS}
{\bm S}_n^{d,z}:=\begin{bmatrix} 0 & \left(z- \f{S_n^d}{\sqrt{d}}\right) \\ \left(z- \f{S_n^d}{\sqrt{d}}\right)^* & 0 \end{bmatrix}
\eeq
are $\pm s_i(z-S_n^d/\sqrt{d})$ where $s_i(z- S_n^d/\sqrt{d})$ are the singular values of $z- S_n^d/\sqrt{d}$. Therefore,
$\wt{m}_n(\xi)$ is the Stieltjes transform of the symmetrized version of the empirical measure of the singular values of $z-S_n^d/\sqrt{d}$, and one has
\beq
\wt{m}_n(\xi)= \xi m_n(\xi^2).
\eeq
Our goal is to show that  $\wt{m}$ converges to a limit
$\wt{m}_\infty$ which is the Stieltjes transform of a probability measure
on $\R$ and satisfies the equation
\beq\label{eq:tilde_m_infty}
\wt{m}_\infty(\xi) (\xi - \wt{m}_\infty(\xi))^2 + \wt{m}_\infty(\xi) (1-|z|^2) -\xi=0.
\eeq
As explained above, we need a bit more:
we need to control the difference $|\wt{m}_n(\xi)-\wt{m}_\infty(\xi)|$
for all $\xi \in \C^+$ such that $\Im \xi\ge ( \log^2 n)^{-1}$. 
The proof of
  Theorem \ref{thm:main} only requires such control for
  $\xi$ purely imaginary.
This is achieved in Theorem \ref{thm:smallish_sing_control} below.
\begin{thm}\label{thm:smallish_sing_control}
Fix any sufficiently small $\vep >0$
and $z \in B_\C(0,1-\vep)$.
Take any sequence of reals $\{\varpi_n\}_{n \in \N}$
such that $\varpi_n \ra \infty$ as $n \ra \infty$.
Then there exist a constant
$\wt{C}_{\ref{thm:smallish_sing_control}}$, depending only on $\vep$, absolute constants ${c}_{\ref{thm:smallish_sing_control}}, \bar{C}_{\ref{thm:smallish_sing_control}}, {C}_{\ref{thm:smallish_sing_control}}$,
and an event $\Omega_n$ with
\[
\P(\Omega_n^c) \le
{C}_{\ref{thm:smallish_sing_control}} \exp(-{c}_{\ref{thm:smallish_sing_control}} (\log n)^2)+\exp(-c_{\ref{thm:smallish_sing_control}}d),
\]
such that for all large $n$, on the event $\Omega_n$ we have
\[
  |\wt{m}_n(\xi) - \wt{m}_\infty(\xi)| \le \wt{C}_{\ref{thm:smallish_sing_control}} \max \left\{ \f{1}{d^{1/2}}, \f{\log n}{n^{1/4}}\right\} (\Im \xi)^{-3}
\]
for all $\xi \in \sS_{\vep,\varpi}$ where
\[
\sS_{\vep,\varpi}:= \left\{ \xi = {\rm i} \eta:  \eta \in (0,\bar{C}_{\ref{thm:smallish_sing_control}}], \,  
\eta^3 \min\{{d}^{1/2}, n^{1/4}(\log n)^{-1}\} 
\ge \varpi_n \right\}.
\]
\end{thm}

\begin{rmk}
In Theorem \ref{thm:smallish_sing_control} we treat the case when $\xi$ is purely imaginary, which simplifies some of the
computations. One can use a similar idea as in the proof of Theorem \ref{thm:smallish_sing_control} to control the difference of $\wt m_n(\xi)$ and $\wt m_\infty(\xi)$ for all $\xi \in \C^+$ when $\Im \xi \ge (\log n)^{-C}$ for some $C>0$. The key is to establish stability of the equation 
\eqref{eq:tilde_m_infty} for all $\xi \in\C^+$. Since the proof of Theorem \ref{thm:main} does not require such control we do not attempt it here.
\end{rmk}
The key to the proof of Theorem \ref{thm:smallish_sing_control} is to establish that $\wt{m}_n(\xi)$ satisfies an approximate version of the equation \eqref{eq:tilde_m_infty}. That is we need to show that $\wt{P}(\wt{m}_n(\xi))\approx 0$ 
where $\wt{P}(m):= {m} (\xi - {m})^2 + {m}(\xi) (1-|z|^2) -\xi$. To show 
this,
it is easier to work with $\wh m_n(\xi)$, the Stieltjes transform of the symmetrized version of the empirical measure of the singular values of $z-\wh S_n^d/\sqrt{d}$ where the entries of $\wh S_n^d$ are now centered. Then concentration bounds for Lipschitz functions of permutations under the Hamming metric also allow us only to consider $\wt P(\E \wh m _n(\xi))$. 

To show that $\wt P(\E \wh m _n(\xi)) \approx 0$ 
we start with a function related to $\wt G(\wh S_n^d)$, where $\wt G(\wh S_n^d)$ is defined by replacing $\wt S_n^d$ with $\wh S_n^d$ in \eqref{eq:wt-G-def}. Then we use the resolvent identity and the fact that $\{P_n^\ell\}$ are independent to identify the dominant and negligible terms. This yields an approximate equation involving $\E \wh m_n(\xi)$ and an auxiliary variable. To remove the auxiliary variable we derive another approximate equation. 

However, this alone does not yield Theorem \ref{thm:smallish_sing_control}. Because $\wt P(\cdot)$ is cubic polynomial,
bounds on $\wt P (\cdot)$ do not translate to  bounds on $|\wt m_n(\xi) - \wt m_\infty(\xi)|$. Moreover, the bound on $\wt P(\wt m_n(\xi))$ depends implicitly on an bound on $\wt m_n(\xi)$ (see Lemma \ref{lem:loop-eqn}). To overcome this difficulty, in Lemma \ref{lem:stability} we show that if $\wt m_n(\xi)$ if bounded below then a bound on $\wt P (\cdot)$ can be translated to a bound on the difference between $\wt m_n(\xi)$ and $\wt m_\infty(\xi)$. On other hand, we can easily show that the desired bounds on $\wt m_n(\xi)$ hold when $ \xi$ is far away from the real line. This gives Theorem \ref{thm:smallish_sing_control} when $\xi$ away from the real line.

To propagate the above bound for all $\xi \in \sS_{\vep,\varpi}$ we use a 
\textit{bootstrap} argument.  In the random matrix literature the
bootstrap argument has already been used on many occasions to prove
local law for different random matrix ensembles.
Specifically,  Erd\H{o}s, Schlein, and Yau \cite{erdHos2009local} used
it to prove the local semicircle law for Wigner matrices
down to the optimal scale. Subsequently it was generalized to prove local
laws for other ensembles of random matrices
(see \cite{knowles2016lecnotes} and references therein).

To carry out the above scheme for $\xi \in \C^+$ such that $\Im \xi$ is small we note that by Lipschitz continuity and the boundedness property of 
$\wt{m}_\infty(\xi)$ derived in Lemma \ref{lem:prop_wtm_infty}, the bounds on $\wt{m}_n(\xi)$ translates to a bound on the same 
with $\xi$ replaced by $\xi'$, whenever $|\Im \xi- \Im \xi'|$ is small.
These bounds on $\wt m_n(\xi')$ together with Lemma \ref{lem:loop-eqn} yield the desired bound on $|\wt{m}_n(\xi')- \wt{m}_\infty(\xi')|$. 
Repeating this scheme we obtain the desired 
result for all $\xi \in \sS_{\vep,\varpi}$.

We note that in the work \cite{cook-rrd} on the spectrum of the adjacency matrix $A_{n,d}$ for a random $d$-regular digraph, a completely different argument is used to obtain quantitative control on the Stieltjes transforms $g_{\xi,z}(A_{n,d})=\frac1n\Tr G(A_{n,d}, \xi,z)$. 
There the approach is by comparison, first replacing $A_{n,d}$ with an i.i.d.\ 0--1 Bernoulli matrix $B_{n,p}$ with entries of mean $p=d/n$, and then replacing $B_{n,p}$ with a suitably rescaled real Ginibre matrix $G_n$ (for which the desired bounds are known to hold), showing that $g_{\xi,z}$ changes by a negligible amount at each replacement. 
The comparison between $g_{\xi,z}(B_{n,p})$ and $g_{\xi,z}(G_n)$ is done using the standard Lindeberg swapping argument, whose use in random matrix theory goes back to Chatterjee \cite{Chatterjee:invariance2}.
The comparison of $g_{\xi,z}(A_{n,d})$ with $g_{\xi,z}(B_{n,p})$ is done by conditioning, basically showing that $g_{\xi,z}(B_{n,p})$ concentrates near its expected value with failure probability smaller than the probability that $B_{n,p}$ lies in  $\cA_{n,d}$, the set of adjacency matrices for $d$-regular digraphs. 
Since $A_{n,d}$ is uniform in $\cA_{n,d}$, obtaining a lower bound for the latter probability amounts to the enumerative problem of estimating the cardinality of $\cA_{n,d}$, which can be solved with known techniques. 
It is possible that this comparison approach could be adapted to the current setup, first replacing $S_n^d$ with a discrete i.i.d.\ matrix $M_n^d$ having i.i.d.\ Poisson entries, and then replacing $M_n^d$ with a Gaussian matrix. However, as $S_n^d$ is not drawn uniformly from a set of matrices the first step would not reduce to an enumeration problem as it did for $A_{n,d}$, and hence this step appears more challenging. Instead we would need a coupling between $S_n^d$ and $M_n^d$, together with a lower bound on the probability that they are close in an appropriate norm. It is likely that a proof along these lines, even if
doable, 
would be somewhat lengthier than the approach taken in the present article.

\section{Invertibility over flat vectors}	\label{sec:structured}

In this section we prove Proposition \ref{prop:struct}.
Throughout this section and Section \ref{sec:unstructured} we let $S_n^d$ and $Z_n$ be as in the statement of Theorem \ref{thm:ssv}, except that some lemmas and propositions are stated under additional assumptions on the range of $d$. (Recall from Remark \ref{rmk:reduced} that we are free to assume $d\le n$; also note that Theorem \ref{thm:ssv} trivially holds for $d\le \log^8n$.)

The general approach is similar to the proof in \cite{cook-rrd}, and indeed we make use of two lemmas from that work (Lemma \ref{lem:flatnet} and Lemma \ref{lem:gap}). However, the differences between the distribution of $S_n^d$ and the adjacency matrix of a uniform random regular digraph $A_{n,d}$ cause the proof here to differ on most of the particulars. We have attempted to structure the proof in roughly the same way as in \cite{cook-rrd}, and use Lemma \ref{lem:image} to encapsulate the parts of the proof which are most different from that work.
On a technical level, the proof here is somewhat simpler as the joint independence of the permutations $\pi_n^\ell$ allows us to avoid the difficult coupling constructions of \cite{cook-rrd}, as well as the use of heavy-powered graph discrepancy results.

\subsection{Anti-concentration for the image of a fixed vector}
To lighten notation we will drop subscripts $n$ from $\bpi_n, \pi_n^\ell$ in this section.

We begin by obtaining lower tail bounds for the norm of $(S_n^d+Z_n)u$ for a fixed vector $u \in \sph$.

\begin{lem}[Image of a fixed vector]	\label{lem:image}
There exist absolute constants $c_{\ref{lem:image}},\ol{c}_{\ref{lem:image}}>0$ such that the following holds.
Let $d\ge 1$, and let $u\in \C^n$ be such that there are disjoint sets $J_1,J_2\subset[n]$ and $\rho>0$, with $|J_1|=|J_2|=m$, such that
\begin{equation}\label{eq:u_j_1j_2_diff}
|u_{j_1}-u_{j_2}|\ge \frac{\rho}{\sqrt{n}}\qquad \forall j_1\in J_1,\; j_2\in J_2.
\end{equation}
Then
\begin{equation}	\label{bd:image}
\P\left( \|(S_n^d+Z_n)u\|_2\le c_{\ref{lem:image}}\rho\min\left\{\sqrt{\frac{md}{n}},  1 \right\} \right) \le \expo{ -\ol{c}_{\ref{lem:image}}\min(md,n)}.
\end{equation}
\end{lem}

\vskip10pt

\begin{rmk}
We note that \eqref{bd:image} is essentially 
optimal when $md$ is small compared with $n$, at least for the case $Z_n=0$ (and we are aiming for estimates that are uniform in $Z_n$).
Indeed, $\|S_n^d u\|_2^2=\sum_{i=1}^n |R_i\cdot u|^2$, where $R_i$ is the $i$th row of $S_n^d$. When $md=o(n)$ the number of ``good" rows $R_i$ 
whose support overlaps the support of $u$ will be roughly $md$ on average (in fact it concentrates near $md$, as will be shown in the proof). For each good row $R_i$ we will have $\E|R_i\cdot u|^2 \approx 1/n$, since the overlap of supports is of order 1 on average, and coordinates $u_j$ are typically of size $1/\sqrt{n}$. 
This means we should expect $\|S_n^du\|_2^2\approx md/n$, and \eqref{bd:image} gives a lower bound at this scale. 
However, the bound is suboptimal when $m\approx n$, in which case there will be roughly $ \asymp n$ 
good rows with overlaps of order $d$, which suggests $\E\|S_n^d u\|_2^2\approx d$ in this case. Thus, we expect \eqref{bd:image} to hold with $\min(\sqrt{md/n}, 1)$ replaced with $\sqrt{md/n}$. The proof could be extended to give such a bound by exploiting the randomness of all $d$ permutations within each row (in the proof we only use one permutation per row) but such a refinement is not necessary for our purposes.
\end{rmk}

The above lemma is a quick consequence of Lemma \ref{lem:WX} below. First we need some notation.
We write $J:=J_1\cup J_2$, and for each $k\in [d]$ we set $I_k :=(\pi^k)^{-1}(J)= I_k^1\cup I_k^2$ with $I_k^a := (\pi^k)^{-1}(J_a)$ for $a=1,2$. Note that $I_k^1, I_k^2$ are disjoint sets of size $m$.
We also denote 
\begin{equation}	\label{def:pik}
\bpi^{<k}:=(\pi^\ell)_{\ell\in[k-1]}, \qquad \bpi^{>k}:=(\pi^\ell)_{\ell\in[k+1,d]}, \qquad \bpi^{(k)}:=(\pi^\ell)_{\ell\in [d]\setminus \{k\}}
\end{equation}
and
\begin{equation}	\label{def:uk}
U_{<k}:=\bigcup_{\ell\in [k-1]}I_\ell, \qquad U_{>k} := \bigcup_{\ell\in[k+1,d]} I_\ell, \qquad U_{(k)} := \bigcup_{\ell\in [d]\setminus\{k\}} I_\ell
\end{equation}
(with $U_{<1}=U_{>d} := \emptyset$). We further write $\pi^{\le k}:= \pi^{<k+1}$, $U_{\ge k}:= U_{>k-1}$, etc.
For fixed $u\in \C^n$ and for $k\in [d]$ let
\[
W_k(u) := \sum_{i\in U_{\le k}} |R_i\cdot u|^2, \qquad X_k(u) := \exp\left( -\frac{n}{\rho^2}W_k(u)\right).
\]

\begin{lem}	\label{lem:WX}
Let $J_1,J_2, u$, and $m$ be as in Lemma \ref{lem:image}.
There are absolute constants $c_0,c_1>0$ such that for any $d_0\le \min(d, c_0n/m)$,
\begin{equation}
  \label{eq:sunday1}
\E  X_{d_0}(u)   \le e^{-c_1md_0}.
\end{equation}

\end{lem}

\begin{proof}[Proof of Lemma \ref{lem:image}]
Let $c_0,c_1>0$ be as in Lemma \ref{lem:WX} and let $d_0= \lfloor \min(d, c_0n/m)\rfloor$. 
For any $c_2>0$ we have
\begin{align*}
\P\left( \|(S_n^d+Z_n)u\|_2 \le c_2\rho\sqrt{\frac{md_0}{n}} \right) 
&= \P\left( \sum_{i=1}^n|R_i\cdot u|^2 \le c_2^2\rho^2\frac{md_0}{n}\right) \\
&\le \P\left( W_{d_0}(u) \le c_2^2\rho^2\frac{md_0}{n}\right).
\end{align*}
Using the pointwise bound $\bI_{[0,\infty)}(x)\le \exp((n/\rho^2)x)$ followed by Lemma \ref{lem:WX}, we can bound the last expression by
\[
e^{ c_2^2 md_0} \E  X_{d_0}(u)
\le \exp( (c_2^2-c_1)md_0).
\]
Taking $c_2=c_1^{1/2}/2$, the claim follows.
\end{proof}

\begin{proof}[Proof of Lemma \ref{lem:WX}]
Fix $u$ as in the statement of the lemma. To lighten notation we will drop the dependence on $u$ from $X_k(u), W_k(u)$ and write $X_k, W_k$.

First we note that for any $\ell\in [d]$, $j_1 \in J_1$, and $j_2 \in J_2$, if $i\in (\pi_n^\ell)^{-1}(\{j_1,j_2\})$, then we have
\begin{equation}\label{eq:row_inner_prod_bound}
\P\left\{ |R_i\cdot u| \le \frac{\rho}{4\sqrt{n}} \ \middle|\  \bpi^{(\ell)}, ((\pi^\ell)^{-1}(j))_{j\notin\{j_1,j_2\}}   \right\} \le \frac12\,.
\end{equation}
Indeed, fixing $((\pi^\ell)^{-1}(j))_{j\notin\{j_1,j_2\}}$ we see that for any
$i \in (\pi^\ell)^{-1}(\{j_1,j_2\})$, either
$\pi^\ell(i)=j_1$ or $j_2$ with equal probability.  Thus, under the conditioning in \eqref{eq:row_inner_prod_bound} we have
\[
|R_i \cdot u| = \Delta_i + u_{j_1} \qquad \text{ or } \qquad \Delta_i+ u_{j_2},
\]
with equal probability, where $\Delta_i$ is some non-random quantity depending on $\bpi^{(\ell)}, ((\pi^\ell)^{-1}(j))_{j\notin\{j_1,j_2\}}$. 
Using the assumption \eqref{eq:u_j_1j_2_diff} and the triangle inequality we immediately deduce \eqref{eq:row_inner_prod_bound}. Now using \eqref{eq:row_inner_prod_bound}, 
\begin{align}
&\E \left( \expo{ -\frac{n}{\rho^2}|R_i\cdot u|^2} \ \middle|\ \bpi^{(\ell)}, ((\pi^\ell)^{-1}(j))_{j\notin\{j_1,j_2\}}  \right)\notag \\
&\qquad\qquad= \int_0^1 \P\left\{ e^{-\frac{n}{\rho^2} |R_i\cdot u|^2} \ge s \ \middle|\ \bpi^{(\ell)}, ((\pi^\ell)^{-1}(j))_{j\notin\{j_1,j_2\}} \right\} ds	\notag\\
&\qquad\qquad\le  \frac12(1-e^{-1/16}) + e^{-1/16}
=:1-q. \label{flat:onerow}
\end{align}

Now we establish the claim for the case $m=1$, in which case
$J=\{j_1,j_2\}$, and for each $k\in [d]$  we set $I_k=\{i_k^1,i_k^2\}$. 
The sets $I_k$ are i.i.d.\ uniform random subsets of $[n]$ of size $2$. 
Let $d_1=c_3d$ for some $c_3>0$ to be determined later. For $2\le k\le d_1$, say that $k$ is ``bad" if $I_k\cap U_{<k} \ne \varnothing$, and let $\cB_k$ be the event that $k$ is bad. 
Then for each $2\le k\le d_1$, 
\begin{equation}	\label{ebk1}
\E \bI(\cB_k) \le 2d_1\cdot n /{n\choose 2}= O(c_3)
\end{equation}
(recall our assumption $d\le n$ from Remark \ref{rmk:reduced}).
Thus we have that $B:=\sum_{k=2}^{d_1}\bI(\cB_k)$ is stochastically dominated by a sum of i.i.d.\ indicator variables with expectation $O(c_3)$. 
From the Chernoff bound it thus follows that
\begin{equation}	\label{ebk2}
\P( B>d_1/2) \le e^{-cd},
\end{equation}
taking $c_3$ sufficiently small. 
Let us denote the complement of this event by $\cG$.
On $\cG$, there exists a set $G\subset[d_1]$ with $|G|\ge d_1/2$ such that the sets $\{I_k\}_{k\in G}$ are pairwise disjoint. We take $G$ to be the largest such set (in the event of a tie we pick one in some measurable fashion). 
We have
\begin{align*}
\E \left[ X_{d_1} \ \middle| \ G, (\pi^k)_{k\notin G}, (I_k)_{k\in G} \right] \bI(\cG) &\le \E\left[ \expo{ -\frac{n}{\rho^2}\sum_{k\in G} |R_{i_k^1}\cdot u|^2} \ \middle| \ G, (\pi^k)_{k\notin G}, (I_k)_{k\in G} \right] \bI(\cG) \\
&=\prod_{k\in G} \E\left[ \expo{ -\frac{n}{\rho^2}|R_{i_k^1}\cdot u|^2} \ \middle| \ G, (\pi^k)_{k\notin G}, (I_k)_{k\in G} \right] \bI(\cG)\\
&\le (1-q)^{|G|}\bI(\cG).
\end{align*}
Thus, since $d_1\le d_0$,
\[
\E X_{d_0} \le \E X_{d_1} \le \P(\cG^c)+ \E X_{d_1} \bI(\cG) \le e^{-cd} + (1-q)^{d_1/2} \le e^{-c'd}
\]
for some constant $c'>0$, establishing the lemma for the case $m=1$.

Now assume $m\ge 2$. In fact we are now free to assume $m\ge C_0$ for some absolute constant $C_0>0$ to be specified later. 
Indeed, for $m\le C_0$ we can simply pass to singleton subsets of $J_1,J_2$ and apply the case $m=1$ (adjusting the constant $c_1$). 


We next show that for any fixed $k\in [d]$,
\begin{align}
\E \left[ \expo{ -\frac{n}{\rho^2} \sum_{i\in I_k\setminus U_{(k)}} |R_i\cdot u|^2} \ \middle| \ \bpi^{(k)} ,\, (\pi^k(i))_{i\in U_{(k)}}\right]
&\le \E \left[ (1-q)^{ M_k } \ \middle| \ \bpi^{(k)} ,\, (\pi^k(i))_{i\in U_{(k)}}\right],	\label{flat:stepk}
\end{align}
where 
\begin{equation}	\label{def:Mk}
M_k:=\min(| I_k^1 \setminus U_{(k)}| , | I_k^2 \setminus U_{(k)}|).
\end{equation}
Note that the expectation in \eqref{flat:stepk} is only taken over part of the randomness of the permutation $\pi^k$. The idea for the proof is that after some further conditioning we can reduce to using only the randomness of $\pi^k$ on $M_k$ pairwise disjoint sets $T_1,\dots, T_{M_k}\subset I_k\setminus U_{(k)}$ of size two, and the action of $\pi^k$ on these sets can be realized as the application of $M_k$ independent transpositions. Thus, we can extract a subsequence of $M_k$ rows $R_{i_j}$ that are jointly independent under the conditioning, and apply the bound \eqref{flat:onerow} to each one.

We turn to the details. Fix $k\in [d]$ and write $\hat{I}_k^a:=I_k^a\setminus U_{(k)}$ for $a=1,2$. 
For given $m_0\in \N$ and $U\subset [n]$ let $\cT(m_0,U)$ be the collection of all sequences $\mathbf{T}:=(T_j)_{j=1}^{m_0}$ of pairwise disjoint 2-sets $T_j:=\{i_j^1,i_j^2\}\subset[n]\setminus U$.
Given $\mathbf{T}\in \cT(m_0,U)$, define the set 
\[
\bS_{\mathbf{T}}:=\bigcap_{j\in [m_0]}\{\pi\in \mathbb{S}_n:  |T_j\cap (\pi^{-1}(J_1))\setminus U|=|T_j\cap (\pi^{-1}(J_2))\setminus U|=1\}.
\]
(Since $\pi^{-1}(J_1),\pi^{-1}(J_2)$ are disjoint, this is the event that they bisect each of the sets $T_j$ for $1\le j\le m_0$.) 
Conditional on $\bpi^{(k)}$ and $M_k$, for any $\mathbf{T}\in \cT(M_k, U_{(k)})$,
\begin{align*}
&\E \left[ \expo{ -\frac{n}{\rho^2} \sum_{i\in I_k\setminus U_{(k)}} |R_i\cdot u|^2} \ \middle| \ \bpi^{(k)} ,\, (\pi^k(i))_{i\in U_{(k)}},\, M_k \right] \bI(\pi^k\in \bS_{\mathbf{T}}) \\
&\qquad\qquad\le \E \left[ \expo{ -\frac{n}{\rho^2} \sum_{j=1}^{M_k} |R_{i_j^1}\cdot u|^2} \ \middle| \ \bpi^{(k)} ,\, (\pi^k(i))_{i\in U_{(k)}},\, M_k \right] \bI(\pi^k\in \bS_{\mathbf{T}}) \\
&\qquad\qquad= \prod_{j=1}^{M_k} \E \left[ \expo{ -\frac{n}{\rho^2}|R_{i_j^1}\cdot u|^2} \ \middle| \ \bpi^{(k)} ,\, (\pi^k(i))_{i\in U_{(k)}},\, M_k \right] \bI(\pi^k\in \bS_{\mathbf{T}})\\
&\qquad\qquad\le (1-q)^{M_k} \bI(\pi^k\in \bS_{\mathbf{T}}),
\end{align*}
where in the penultimate line we noted that under the conditioning and restriction to $\pi^k\in \bS_{\mathbf{T}}$ the pairs of rows $\{(R_{i_j^1}, R_{i_j^2})\}_{j=1}^{M_k}$ are jointly independent, and in the last line we applied \eqref{flat:onerow}.
Now letting $\cT':=\cT'(M_k, U_{(k)})\subset\cT(M_k, U_{(k)})$ be a sub-collection such that $\{\mathbb{S}_{\mathbf{T}}\}_{\mathbf{T}\in \cT'}$ partitions the range of $\pi^k$ under the conditioning on   $\bpi^{(k)} , (\pi^k(i))_{i\in U_{(k)}}, M_k$, we have
\begin{align*}
\E \left[ \expo{ -\frac{n}{\rho^2} \sum_{i\in I_k\setminus U_{(k)}} |R_i\cdot u|^2} \ \middle| \ \bpi^{(k)} ,\, (\pi^k(i))_{i\in U_{(k)}},\, M_k \right] 
&\le \sum_{\mathbf{T}\in \cT'} (1-q)^{M_k} \bI(\pi^k\in \bS_{\mathbf{T}}) \le (1-q)^{M_k}.
\end{align*}
Undoing the conditioning on $M_k$ yields \eqref{flat:stepk} as desired.

Define the decreasing sequence of sigma algebras
\begin{equation}
\cF_k:= \big\langle \pi^{>k}, (\pi^\ell(i))_{\ell\le k, i\in U_{>k}}\big\rangle, \quad k=1,\dots, d-1
\end{equation}
and set $\cF_d$ to be the trivial sigma algebra.
In words, conditioning on $\cF_k$ fixes the permutations $\pi_{k+1},\dots, \pi_d$, along with the values $\pi^\ell(i)$ for $\ell\le k$ and all $i$ in the preimages of $J=J_1\cup J_2$ under $\pi_{k+1},\dots, \pi_d$. 
Note that $\cF_1\supset \cdots \supset \cF_d$.  
Note also that for any $k\in [d]$ the random variable $M_k$ defined in \eqref{def:Mk} is $\cF_{k-1}$-measurable. 
Indeed, conditioning on $\cF_{k-1}$ fixes $I_k^1, I_k^2, U_{\ge k}$, and $\pi^\ell(i)$ for all $\ell \le k-1$ and $i\in I_k\subset U_{\ge k}$, which in turn determine 
\[
I_k^a\cap U_{<k} = \bigcup_{\ell\le k-1}\{i\in I_k^a: \pi^\ell(i) \in J\} ,\qquad a=1,2,
\]
so $I_k^a\cap U_{<k}^c\cap U_{>k}^c=I_k^a\setminus U_{(k)}$ are fixed as well for $a=1,2$. 

From \eqref{flat:stepk}, 
\begin{align}
\E[X_1|\cF_1] 
&= \E \left[ \expo{ -\frac{n}{\rho^2} \sum_{i\in I_1} |R_i\cdot u|^2} \ \middle| \  \cF_1 \right]	\notag\\
&\le  \E \left[ \expo{ -\frac{n}{\rho^2} \sum_{i\in I_1\setminus U_{(1)}} |R_i\cdot u|^2} \ \middle| \ \cF_1 \right]	\le \E \left[ (1-q)^{M_1} \ \middle| \ \cF_1\right].	\label{EX1F1}
\end{align}
Now letting $2\le k\le d_0$, we have
\begin{align*}
\E[X_k|\cF_k]
&= \E \left[ X_{k-1} \E\left[ \expo{ -\frac{n}{\rho^2} \sum_{i\in I_k\setminus U_{(k)}} |R_i\cdot u|^2} \ \middle| \ \bpi^{(k)}, (\pi^k(i))_{i\in U_{(k)}} \right] \ \middle| \ \cF_k \right] \\
&\le  \E \left[ X_{k-1} \E\left[  (1-q)^{M_k}  \ \middle| \ \bpi^{(k)}, (\pi^k(i))_{i\in U_{(k)}} \right] \ \middle| \ \cF_k \right] \\
&= \E \left[ X_{k-1}   (1-q)^{M_k}  \ \middle| \ \cF_k \right] 
= \E \left[   \E[X_{k-1}|\cF_{k-1}]   (1-q)^{M_k} \ \middle| \ \cF_k \right], 
\end{align*}
where the penultimate equality follows upon noting that
\[
\cF_k \subset \langle \bpi^{(k)}, (\pi^k(i))_{i \in U_{(k)}} \rangle
\]
and applying the tower property of the conditional expectation, and in the last step we have used that $M_k$ is $\cF_{k-1}$-measurable. 
Iterating this bound over $2\le k\le d_0$ and combining with \eqref{EX1F1} we obtain
\begin{align*}
\E[X_{d_0}|\cF_{d_0}] \le \E \left[ \prod_{\ell=1}^{d_0} (1-q)^{M_\ell}  \ \middle| \ \cF_{d_0}\right]
\le \E \left[ \prod_{\ell=1}^{d_0} \left( (1-q)^{m/2} + \bI\left( M_\ell <\frac{m}2\right) \right) \ \middle| \ \cF_{d_0}\right].
\end{align*}
Thus,
\begin{equation}	\label{EK.bin}
\E X_{d_0}
\le \E\prod_{\ell=1}^{d_0} \left( (1-q)^{m/2} + \bI\left( M_\ell <\frac{m}2\right) \right) = \sum_{k=0}^{d_0} (1-q)^{\frac12m(d_0-k)} \sum_{L\in {[d_0]\choose k}} \P\left( \cE_L\right) 
\end{equation}
where
\[
\cE_L:=\bigcap_{\ell\in L} \left\{ M_\ell <\frac{m}2\right\}.
\]
Next we will show that for any $L\subset[d_0]$, 
\begin{equation}	\label{CEL}
\P(\cE_L) \le e^{-cm|L|},
\end{equation}
for some absolute constant $c>0$. Assuming \eqref{CEL}, we have from
\eqref{EK.bin} that
\begin{align*}
\E X_{d_0}  
\le \sum_{k =0}^{d_0} {d_0\choose k }(1-q)^{\frac12m(d_0-k) }  e^{-cmk } =((1-q)^{m/2}+ e^{-cm})^{d_0}\le e^{-c'md_0}, 
\end{align*}
where the last inequality is obtained 
by taking the constant $C_0>0$ sufficiently large and thus $m\geq C_0$.
This yields \eqref{eq:sunday1} and hence Lemma \ref{lem:WX}.

It only remains to establish \eqref{CEL}.
Since the variables $M_\ell$ are exchangeable we may take $L=[k]$. 
As $I_\ell=I_\ell^1\cup I_\ell^2$ with $|I_\ell^1|=|I_\ell^2|=m$, on $\cE_L$ we have $|I_\ell\setminus U_{(\ell)}|<3m/2$ for each $\ell\le k$. 
Hence,
\[
|U_{\le k}| = \left| \bigcup_{\ell\le  k} I_\ell\right| = \sum_{\ell\le  k} |I_\ell\setminus U_{<\ell}| <\frac32m k.
\]
On the other hand,  
\[
|U_{<\ell}|\le |U_{\le  k}| \le 2m k\le 2md_0\le 2c_0n,
\]
and since the sets $I_\ell$ are independent and uniformly distributed over ${[n]\choose 2m}$, we have
\[
\E |I_\ell\setminus U_{<\ell}| \ge 2m\left( 1- 2c_0\right)\ge 1.9m
\]
for each $\ell\le  k$, where we took the constant $c_0$ sufficiently small. Hence,
\[
\E|U_{\le k}| = \sum_{\ell=1}^k \E |I_\ell \setminus U_{< \ell}| \ge 1.9m k.
\]
We have thus shown
\[
\P(\cE_L) \le \P(|U_{\le  k}| < 0.99 \E|U_{\le k}|).
\]
The latter probability can be shown to be at most $e^{-cm k}$ by an argument using stochastic domination and the Chernoff bound similar to what was done in \eqref{ebk1}--\eqref{ebk2}.
This gives \eqref{CEL} and hence the claim.
\end{proof}

\subsection{Weak control on flat vectors}

In this subsection we establish the following, which already implies Proposition \ref{prop:struct} when $d\ge n/\log n$, but is weaker for smaller values of $d$.
Recall the events $\cE_K(m,\rho)$ from \eqref{def:eventK}.

\begin{lem}[Invertibility over flat vectors, weak version]	\label{lem:struct}
There are absolute constants $c_{\ref{lem:struct}}$, $\ol{c}_{\ref{lem:struct}}$, $c'_{\ref{lem:struct}}>0$ such that the following holds.
Let $\gamma\ge 1$ and $1\le K\le n^\gamma$. Assume $d\ge 1$.
Then for any $1\le m_0\le c'_{\ref{lem:struct}}d/\gamma\log n$,
\begin{equation}
\P\left\{ \cE_K\left( m_0,\, \frac{c_{\ref{lem:struct}}}{K\sqrt{m_0}}\right)\right\} \le e^{-\ol{c}_{\ref{lem:struct}}d}.
\end{equation}
\end{lem}

We will need the following lemma from \cite{cook-rrd}.

\begin{lem}[Metric entropy for flat vectors, cf.\ {\cite[Lemma 3.3]{cook-rrd}}]		\label{lem:flatnet}
Let $1\le m\le n/10$ and $\rho\in (0,1)$.
There exists $\Sigma_0 := \Sigma_0(m,\rho)\subset \Flat_0(m,\rho)$ such that $\Sigma_0$ is a $\rho$-net for $\Flat_0(m,\rho)$ and $|\Sigma_0| \le  \big( \frac{C_{\ref{lem:flatnet}}n}{m\rho^2}\big)^m$ for some absolute constant $C_{\ref{lem:flatnet}}>0$. 
\end{lem}

\begin{proof}[Proof of Lemma \ref{lem:struct}]
Our plan is to use Lemma \ref{lem:image} first to obtain a bound on $\|(S_n^d+Z_n)u\|_2$ for any arbitrary but fixed $u \in \Flat(m_0,\rho_0)$, 
where $\rho_0:=c/K\sqrt{m_0}$ for some $c$ to be determined determined during the course of the proof. 
Then using Lemma \ref{lem:flatnet} we claim that the metric entropy of $\Flat(m_0,\rho_0)$ is small enough to allow us to take a union bound.

In order to apply Lemma \ref{lem:image} we need to find disjoint sets $J_1$ and $J_2$ such that $|u_{j_1} - u_{j_2}| $ is large for every $j_1 \in J_1$ and $j_2 \in J_2$. To this end,
consider an arbitrary vector $u\in \Flat_0(m_0,\rho_0)$.
By definition, there exists $\lambda\in \C$, $v\in \Sparse(\corAa{m_0})$ and $w\in \rho_0 B_\C(0,1)$ such that $u=v+\frac{\lambda}{\sqrt{n}}\un+w$. First we claim that
\begin{equation}	\label{very:vwlb}
\|v+w\|_2\ge 1/2.
\end{equation}
Indeed, by the triangle inequality,
\begin{equation}	\label{very:1}
|\lambda| = \left\|\frac{\lambda}{\sqrt{n}}\un\right\|_2 \ge \|u\|_2-\|v+w\|_2 = 1-\|v+w\|_2.
\end{equation}
On the other hand by the assumption $u\in \sph\cap\oneperp$ and applying
the Cauchy--Schwarz inequality we get
\[
|\lambda|\sqrt{n} = \left| \sum_{j=1}^n (v_j+w_j)\right| \le \|v+w\|_2\sqrt{n}
\]
and so
\[
|\lambda| \le \|v+w\|_2.
\]
Combined with \eqref{very:1} this gives \eqref{very:vwlb}.

Let $J\subset[n]$ with $|J|=m_0$ such that $\supp(v)\subset J$.
Shrinking $\rho_0$, if necessary, from \eqref{very:vwlb} we obtain
\[
\frac18 \le \frac14 - \norm{w}_2^2\le \sum_{j \in J} |v_j+w_j|^2 \le
m_0 \max_{j\in J} \left| u_j -\frac{\lambda}{\sqrt{n}}\right|^2.
\]
It follows that there exists $j_1\in J$ with
\begin{equation}	\label{vj1}
  \left| u_{j_1}-\frac{\lambda}{\sqrt{n}}\right| \ge \frac{1}{2\sqrt{2 m_0}}.
\end{equation}
On the other hand, since $\sum_{j\in J^c} |w_j|^2 \le \|w\|^2_2\le \rho^2_0$
it follows from the pigeonhole principle that there exists $j_2\in J^c$ such that
\[
\left| u_{j_2}-\frac{\lambda}{\sqrt{n}}\right| =|w_{j_2}| \le \frac{\rho_0}{\sqrt{n-m_0}} \le \frac{1}{5\sqrt{m_0}},
\]
where we have used the fact that $m_0=o(n)$ and the definition of $\rho_0$. Now using the triangle inequality we have
\begin{equation}	\label{very:du}
|u_{j_1} - u_{j_2}| \ge \frac{1}{4\sqrt{m_0}}.
\end{equation}
To complete the proof of the lemma we then apply Lemma \ref{lem:image} with $J_1=\{j_1\}$, $J_2=\{j_2\}$, $m=1$ and $\rho = \frac14\sqrt{n/m_0}$. Recalling that $u\in \Flat_0(m_0,\rho_0)$ was abitrary, we conclude the bound
\begin{equation}	\label{crude:fixed}
\sup_{u\in \Flat_0(m_0,\rho_0)}\P\left\{ \|(S_n^d+Z_n)u\|_2\le \f{c_{\ref{lem:image}}}{4}\sqrt{\frac{d}{m_0}}\right\} \le e^{-\ol{c}_{\ref{lem:image}}d},
\end{equation}
where we also use the fact that $d \le n$.

Now by Lemma \ref{lem:flatnet} we may fix a $\rho_0$-net $\Sigma_0(m_0,\rho_0)\subset \Flat_0(m_0 ,\rho_0)$ for $\Flat_0(m_0 ,\rho_0)$ of cardinality at most $ (C_{\ref{lem:flatnet}}n/m_0 \rho_0^2)^{m_0 }$.
On the event $\cE_K(m_0 ,\rho_0)$ we have $\|(S_n^d+Z_n)v\|_2\le \rho_0 K\sqrt{d}$ for some $v\in \Flat_0(m_0 ,\rho_0)$. Letting $u\in \Sigma_0(m_0 ,\rho_0)$ such that $\|u-v\|_2\le \rho_0$, by the triangle inequality we have
\[
\|(S_n^d+Z_n)u\|_2\le \|(S_n^d+Z_n)v\|_2 + \|(S_n^d+Z_n)(u-v)\|_2 \le \rho_0 K\sqrt{d} + \rho_0 \|S_n^d+Z_n\|_{\oneperp}\le 2\rho_0 K\sqrt{d},
\]
where in the last step we have used the fact that $\cE_K(m_0,\rho_0) \subset \cB(K)$.
Thus, by the union bound,
\begin{align*}
\P\left\{ \cE_K(m_0 ,\rho_0)\right\}  \le \sum_{u\in \Sigma_0(m_0 ,\rho_0)} \P\left\{ \|(S_n^d+Z_n)u\|_2\le 2\rho_0 K\sqrt{d}\right\}.
\end{align*}
We choose $c_{\ref{lem:struct}}$ such that $ c_{\ref{lem:struct}} \le c_{\ref{lem:image}}/2$ and hence $2 \rho_0 K \sqrt{d} \le c_{\ref{lem:image}} \sqrt{\f{d}{m_0}}$. Therefore,
by \eqref{crude:fixed},
\begin{align*}
\P\left\{ \cE_K(m_0 ,\rho_0)\right\}
\le |\Sigma_0(m_0,\rho_0)| e^{-\ol{c}_{\ref{lem:image}}d}& \le \left(\frac{C_{\ref{lem:flatnet}}n}{m_0\rho_0^2}\right)^{m_0} e^{-\ol{c}_{\ref{lem:image}}d}\\
& = \left(\frac{C_{\ref{lem:flatnet}}nK^2}{c_{\ref{lem:struct}}^2}\right)^{m_0} e^{-\ol{c}_{\ref{lem:image}}d}\\
&\le \expo{  (1+2 \gamma)m_0 \log n +m_0 \log(C_{\ref{lem:flatnet}}/c_{\ref{lem:struct}}^2) -\ol{c}_{\ref{lem:image}}d}\\
& \le \expo{ -\f{\ol{c}_{\ref{lem:image}}}{3}d},
\end{align*}
where in the last step we choose $c'_{\ref{lem:struct}}$ suffciently small. The proof of the lemma thus completes.
\end{proof}

\subsection{Proof of Proposition \ref{prop:struct}}




In this subsection we upgrade the weak control on flat vectors obtained in Lemma \ref{lem:struct} to obtain Proposition \ref{prop:struct} by iterative application of Lemma \ref{lem:increment} below.
The idea is that once we have shown $S_n^d+Z_n$ is well-invertible over $\Flat(m_0,\rho_0)$ for some small value of $m_0\in [n]$ we can exploit the improved anti-concentration properties of vectors in $\sph\setminus\Flat(m_0,\rho_0)$.
(Here and in the sequel, by saying that a matrix $A$ is
  well-invertible over a subspace $B$ we mean that with high probability a good
lower bound on $\|A u\|_{\corAa{2}}$ holds for all $u\in B$.)
  This allows us to beat the increased metric entropy cost for $\Flat(m_1,\rho_1)$ for some $m_1>m_0$ that exceeds $m_0$ by a factor (essentially) $d$, and some $\rho_1>0$ somewhat smaller than $\rho_0$. We can iterate this roughly $\log_dn$ times to obtain control on $\Flat(m,\rho)$ with $m$ essentially size $n$ (up to log corrections). A similar iterative approach was used in 
  the sparse i.i.d.\ setup in \cite{gotze2010circular} (with the sets $\Flat(m_0,\rho_0)$ replaced by sets of vectors lying close to $m_0$-sparse vectors).

For deducing the improved anti-concentration properties as we increment the parameter $m$ we will need the following lemma from \cite{cook-rrd}.

\begin{lem}[Locating a bimodal component, cf.\ {\cite[Lemma 3.5]{cook-rrd}}]	\label{lem:gap}
Let $u\in \sph\setminus \Flat(m^\star,\rho)$.
There exist disjoint sets $J_1,J_2\subset[n]$ such that $|J_1|\ge m^\star$, $|J_2| \ge c_{\ref{lem:gap}}(n-m^\star)$ and
\begin{equation}	\label{twolevels.weak}
|u_{j_1}-u_{j_2}| \ge \frac{\rho}{4\sqrt{n}} \qquad \forall\; j_1\in J_1,\;j_2\in J_2,
\end{equation}
where $c_{\ref{lem:gap}}>0$ is some absolute constant.
\end{lem}

\begin{lem}[Incrementing control on flat vectors]	\label{lem:increment}
There exists absolute constants $c_{\ref{lem:increment}}, c'_{\ref{lem:increment}},\ol{c}_{\ref{lem:increment}}>0$ such that the following holds.
Let $\gamma\ge 1$ and $1\le K\le n^\gamma$. Assume $1\le d\le n$.
Let
\begin{equation}	\label{lb:rho_floor}
e^{-\gamma \log^2n} \le \rho^\star<1 \qquad \text{ and } \qquad 1\le m^\star \le \min\left( \frac1d, \frac{c_{\ref{lem:gap}}}{1+c_{\ref{lem:gap}}}\right) n
\end{equation}
and let $m',\rho'$ satisfy
\begin{equation}	\label{bds:ratios}
m^\star<m' \le \left(\frac{c_{\ref{lem:increment}}'d}{\gamma\log^2n}\right) m^\star\;, \qquad 0<\rho'\le \left( \frac{{c}_{\ref{lem:increment}}\sqrt{m^\star d}}{Kn}\right) \rho^\star.
\end{equation}
Then
\begin{equation}
\P\left\{ \cE_K(m',\rho')\setminus \cE_K(m^\star,\rho^\star)\right\} \le \expo{ -\ol{c}_{\ref{lem:increment}}m^\star d}.
\end{equation}
\end{lem}

\begin{proof}
Let $m^\star,m',\rho^\star,\rho'$ be as in the statement of the lemma (note that the lemma holds vacuously for $d\le \log^2n$ by the assumptions \eqref{bds:ratios}). Since the event $\cE_K(m,\rho)$ is monotone in the parameters $m,\rho$, we may and will assume that the upper bounds \eqref{bds:ratios} hold with equality.

First we will argue
\begin{equation}	\label{incr:fixed}
\sup_{u\in \Flat_0(m',\rho')\setminus \Flat_0(m^\star,\rho^\star)} \P\left\{ \|(S_n^d+Z_n)u\|_2 \le \f{c_{\ref{lem:image}}\rho^\star}{4} \sqrt{\frac{m^\star d}{n}}\right\} \le \expo{-\ol{c}_{\ref{lem:image}}m^*d}.
\end{equation}
Indeed, consider an arbitrary fixed element $u\in \Flat_0(m',\rho')\setminus \Flat_0(m^\star,\rho^\star)$.
Note that
\begin{align*}
\Flat_0(\me',\rho')\setminus \Flat_0(\me^\star,\rho^\star)
&= \sph\cap \oneperp \cap \Flat(\me',\rho') \cap \Flat(\me^\star,\rho^\star)^c\\
&\subset \sph\setminus \Flat(\me^\star,\rho^\star).
\end{align*}
By the assumed upper bound on $m^\star$ we can apply Lemma \ref{lem:gap} to obtain disjoint sets $J_1,J_2\subset [n]$ with $|J_1|\ge m^\star$, $|J_2|\ge c_{\ref{lem:gap}}(n-m^\star)\ge m^\star$, such that
\begin{equation}
|u_{j_1}-u_{j_2}|\ge \frac{\rho^\star}{4\sqrt{n}}	\qquad 	\forall\; j_1\in J_1,\; j_2\in J_2.
\end{equation}
By deleting elements from
$J_1$ and $J_2$ we may assume $|J_1|=|J_2|=m^\star$. Now we apply Lemma \ref{lem:image} to obtain
\[
\P\left\{ \|(S_n^d+Z_n)u\|_2 \le \f{c_{\ref{lem:image}}\rho^\star}{4} \sqrt{\frac{m^\star d}{n}}\right\} \le \expo{-\ol{c}_{\ref{lem:image}}m^*d}
\]
where we have used the fact that $m^\star d\le n$.
Since $u$ was arbitrary, \eqref{incr:fixed} follows.

As in the proof of Lemma \ref{lem:struct} we conclude by application of the union bound.
Indeed, using Lemma \ref{lem:flatnet} we fix a $\rho'$-net $\Sigma_0'\subset \Flat_0(m',\rho')$ for $\Flat_0(m',\rho')$ with $|\Sigma_0'| \le (C_{\ref{lem:flatnet}}n/m'\rho'^2)^{m'}$.
By similar reasoning as in the proof of Lemma \ref{lem:struct}, on the event $\cE_K(m',\rho')$, there exists $u\in \Sigma_0'$ such that $\|(S_n^d+Z_n)u\|_2\le 2\rho'K\sqrt{d}$. Since $d \le n$, choosing ${c}_{\ref{lem:increment}}$ sufficiently small we also have that $2 \rho' K \sqrt{d} \le (c_{\ref{lem:image}}\rho^\star/4) \sqrt{m^\star d/n}$. Therefore, applying the union bound and \eqref{incr:fixed} we deduce,
\begin{align*}
\P\left\{ \cE_K(m',\rho')\setminus \cE_K(m^\star,\rho^\star)\right\}
&\le \sum_{u\in \Sigma_0'} \P\left( \cE_K(m^\star,\rho^\star)^c\cap \big\{ \|(S_n^d+Z_n)u\|_2\le 2\rho'K\sqrt{d}\big\}\right)\\
&\le \sum_{u\in \Sigma_0'\setminus \Flat_0(m^\star,\rho^\star)} \P\left(\|(S_n^d+Z_n)u\|_2\le 2\rho' K\sqrt{d}\right)\\
&\le \left( \frac{C_{\ref{lem:flatnet}} n}{m'\rho'^2}\right)^{m'} \expo{ -\ol{c}_{\ref{lem:image}}m^\star d}\\
&\le \expo{ m' \left( \log (n^3K^2) + 2\log \frac{1}{\rho^\star} + \log \left(\f{C_{\ref{lem:flatnet}}}{c_{\ref{lem:increment}}'^2 {c}_{\ref{lem:increment}}^2}\right)\right)-\ol{c}_{\ref{lem:image}}m^\star d }.
\end{align*}
Since $K \le n^\gamma$, and $\rho^\star$ and $m'$ satisfies \eqref{lb:rho_floor} and \eqref{bds:ratios} respectively we further obtain that
\[
\P\left\{ \cE_K(m',\rho')\setminus \cE_K(m^\star,\rho^\star)\right\}  \le \expo{ 3c_{\ref{lem:increment}}' m^\star d - \ol{c}_{\ref{lem:image}}m^\star d}.
\]
Now we choose $c_{\ref{lem:increment}}'$ sufficiently small to complete the proof of the lemma.
\end{proof}

\vskip10pt

\begin{proof}[Proof of Proposition \ref{prop:struct}]
We may and will assume throughout that $n$ is sufficiently large depending on $\gamma$.
In the sequel, we will frequently apply the observation that the events $\cE_K(m,\rho)$ are monotone increasing in the parameters $m$ and $\rho$.


For $k\ge 0$, set
\begin{equation}
m_k := \left(\frac{c_{\ref{prop:struct}}d}{\gamma\log^2n}\right)^{k}\;,\qquad \tilde{\rho}_k := n^{-10\gamma k},
\end{equation}
where $c_{\ref{prop:struct}}:=c_{\ref{lem:struct}}' \wedge c_{\ref{lem:increment}}'$, and denote
\[
\cE_k:= \cE_K(m_k,\tilde{\rho}_k).
\]
Note that $m_k$ is an increasing sequence by our assumption $d\ge \log^3n$.
From Lemma \ref{lem:struct} and monotonicity of $\cE_K(m, \cdot)$, we have
\begin{equation}	\label{struct:0}
\P( \cE_1) \le e^{-\ol{c}_{\ref{lem:struct}}d}.
\end{equation}
Let $k^*\ge 0$ be such that
\begin{equation}
\frac{n}{d} \in [m_{k^*},m_{k^*+1}).
\end{equation}
From the definitions of $k^*$ and $m_k$ and using the fact that $d \ge \log^3 n$ we see that
\begin{equation}	\label{struct:kasymp}
 k^* \le  \frac{C\log n}{\log d}
\end{equation}
for a sufficiently large constant $C>0$.
By monotonicity of $\cE_K(\cdot,\rho)$,
\[
\cE_K\left( \frac{n}{d}, \tilde{\rho}_{k^*+1}\right) \subset \cE_{k^*+1}.
\]
Thus, applying the union bound,
\begin{equation}	\label{eq:cE_K_bd}
\P\left\{ \cE_K\left(\frac{c_{\ref{prop:struct}}n}{\gamma\log^2n}, \tilde{\rho}_{k^*+2}\right)\right\} 
 \le  \P\left( \cE_K\left(\frac{c_{\ref{prop:struct}}n}{\gamma\log^2n},\tilde{\rho}_{k^*+2}\right) \setminus \cE_K\left( \frac{n}{d},\tilde{\rho}_{k^*+1}\right) \right) 
+ \P( \cE_1) + \sum_{k=1}^{k^*} \P( \cE_{k+1}\setminus \cE_k) 
\end{equation}
where we interpret the last sum as zero if $k^*=0$.
From \eqref{struct:kasymp} we have
\[
\tilde{\rho}_{k^*+1} = n^{-10(k^*+2)\gamma}  \ge \expo{ - \gamma\log^2n},
\]
for $n$ sufficiently large. Thus, we can apply Lemma \ref{lem:increment} with $m^\star=n/d$ and $\rho^\star= \tilde{\rho}_{k^*+1}$ to bound
\begin{align}
& \P\left\{ \cE_K\left(\frac{c_{\ref{prop:struct}}n}{\gamma\log^2n},\tilde{\rho}_{k^*+2}\right) \setminus \cE_K\left( \frac{n}{d},\tilde{\rho}_{k^*+1}\right) \right\} \notag\\
 &\qquad\qquad\le \P\left\{ \cE_K\left( \frac{n}{d} \times \frac{c_{\ref{prop:struct}} d}{\gamma\log^2n}, \tilde{\rho}_{k^*+1}\times \frac{c_{\ref{lem:increment}}\sqrt{(n/d) \times d}}{Kn}\right) \setminus \cE_K\left( \frac{n}{d},\tilde{\rho}_{k^*+1}\right) \right\} \notag\\
&\qquad\qquad \le e^{-\ol{c}_{\ref{lem:increment}}n} \le e^{ -\ol{c}_{\ref{lem:increment}}d}.	\label{struct:last}
\end{align}
For the case that $k^*\ge1$, since
\[
\frac{m_{k+1}}{m_k} \le \frac{c_{\ref{lem:increment}}'d}{\gamma\log^2n}\;, \qquad \frac{\tilde{\rho}_{k+1}}{\tilde{\rho}_k}= n^{-10\gamma} \le \frac{{c}_{\ref{lem:increment}}\sqrt{m_k d}}{Kn}
\]
we may similarly apply Lemma \ref{lem:increment} to deduce
\begin{equation}	\label{struct:middle}
\P( \cE_{k+1}\setminus \cE_k) \le e^{-\ol{c}_{\ref{lem:increment}}d},
\end{equation}
for each $1\le k\le k^*$.
Combining \eqref{struct:0} and \eqref{struct:last}--\eqref{struct:middle}, from \eqref{eq:cE_K_bd} and our assumption $d\ge \log^3n$ we conclude
\[
\P\left\{ \cE_K\left(\frac{c_{\ref{prop:struct}}n}{\gamma\log^2n}, \tilde{\rho}_{k^*+2}\right)\right\}   \le 4k^* e^{-cd} \le e^{-cd/2},
\]
where $c$ is a sufficiently small positive constant. From \eqref{struct:kasymp} we have $\tilde{\rho}_{k^*+2} \ge n^{-C'\gamma \log_dn}$ for a sufficiently large constant $C'>0$. This completes the proof of the proposition.
\end{proof}

\section{Invertibility over non-flat vectors}	\label{sec:unstructured}

Having shown that $S_n^d+Z_n$ is well-invertible over vectors in $\Flat_0(m,\rho)$ with $m$ essentially of size $n$ (up to log factors), it remains to control the infimum of $\|(S_n^d+Z_n)u\|_2$ over the non-flat vectors $u\in \sph\cap\oneperp \cap \Flat(m,\rho)^c$. The metric entropy of non-flat vectors is too large to take union bounds, so a different approach must be used for reducing to consideration of $(S_n^d+Z_n)u$ for a fixed vector $u$.
We follow \cite{RuVe} by using an averaging argument, which in the setting of i.i.d.\ matrices reduces the problem to consideration of a dot product $R_i\cdot u$ for a single row vector $R_i$ and a unit vector $u$ that is orthogonal to the span of the remaining rows (and hence may be treated as fixed).

In the present setting, in order to use random transpositions we must consider a fixed pair of rows $R_{i_1},R_{i_2}$ and the dot product $(R_{i_1}-R_{i_2})\cdot u$. Here $u$ is a unit vector that is (almost) orthogonal to the remaining $n-2$ vectors as well as $R_{i_1}+R_{i_2}$.
The lack of independence between the rows makes the argument considerably more delicate than in \cite{RuZe}.
In particular, the vectors $R_{i_1},R_{i_2}$ and $u$ all depend on the rows $\{R_i:i\ne i_1,i_2\}$, and we want to avoid the event that, after conditioning on these $n-2$ rows, the vector $u$ is not flat on the supports of $R_{i_1}$ and $R_{i_2}$. To overcome this we will adapt an argument of Litvak et al.\ that was used to bound the singularity probability for adjacency matrices of random regular digraphs \cite{LLTTY}. Specifically, we define ``good overlap events" $\cO_{i_1,i_2}$ on which we may select an appropriate (almost-) normal vector $u$ that has ``high variation" on the supports of $R_{i_1},R_{i_2}$, see Definition \ref{def:goodo}. In Lemma \ref{lem:avg} we show that, if we restrict to the events that
\begin{enumerate}
\item $S_n^d+Z_n$ is well-invertible over flat vectors, and
\item $S_n^d$ has \emph{no holes} in the sense that the nonzero entries are uniformly distributed in all sufficiently large submatrices,
\end{enumerate}
then the events $\cO_{i_1,i_2}$ hold for a constant proportion of pairs $i_1,i_2\in [n]$. Event (1) holds with high probability by Proposition \ref{prop:struct}, while the no-holes property (2) is shown to hold with high probability in Section \ref{sec:discrepancy}. We can then restrict to $\cO_{i_1,i_2}$ for some fixed $i_1,i_2$ by an averaging argument, at which point we can control the dot product $(R_{i_1}-R_{i_2})\cdot u$ using a Berry--Ess\'een-type bound.
As with the previous section, the arguments are similar to those in the work \cite{cook-rrd} for random regular digraphs, but differ in many particulars due to the different nature of the distribution of $S_n^d$.

\subsection{The no-holes property}	\label{sec:discrepancy}

In the graph theory literature, a graph is said to enjoy a \emph{discrepancy property} if the number of edges between all sufficiently large pairs of vertex sets $U,V$ is roughly $\delta|U||V|$, where $\delta$ is the density of the graph.
In terms of the adjacency matrix this says that all sufficiently large submatrices have roughly the same density.
We will need a one-sided version of this property, called the \emph{no-holes property}, to hold for $S_n^d$ with high probability -- namely, that all sufficiently large submatrices have density at least half of the expected value. In fact, we will need this property to hold for all matrices $\{S_n^T: T\subset [d]\}$ obtained by summing only the permutation matrices $P_n^\ell$ with $\ell\in T$. (Note that $S_n^T$ can be interpreted as the adjacency matrix for a random regular directed multigraph.)

For $L\subset[d]$ and $I,J\subset[n]$, write
\begin{equation}
e_L(I,J):= \sum_{\ell\in L} \sum_{i\in I} \bI(\pi^\ell_n(i)\in J).
\end{equation}
Since the permutations $\pi^\ell_n$ have uniform distribution, by linearity of the expectation,
\begin{equation}
\E e_L(I,J) = \frac1n|L||I||J|.
\end{equation}
For $k_0\in[d]$, $n_0\in [n]$ we define the \emph{no-holes event}
\begin{equation}	\label{def:discrep}
\cD(k_0,n_0) := \bigcap_{\substack{L\subset[d]:\\ |L|\ge k_0}} \bigcap_{\substack{I,J\subset[n]:\\|I|,|J|\ge n_0}} \bigg\{\, e_L(I,J) \ge \frac{|L||I||J|}{2n}\,\bigg\}.
\end{equation}
(This event actually only enforces a one-sided discrepancy property.)

\begin{lem}[No-holes property]	\label{lem:discrepancy}
Assume $1\le d\le 10n$. If $k_0n_0^2 \ge C_{\ref{lem:discrepancy}}n^2$ for a sufficiently large absolute constant $C_{\ref{lem:discrepancy}}>0$, then
\begin{equation}
\P\left( \cD(k_0,n_0)\right) \ge 1-e^{-n}.
\end{equation}
\end{lem}

\begin{proof}
The proof follows from a result of \cite{Cook:discrep} upon taking the union bound. Indeed, from \cite[Theorem 1.13]{Cook:discrep} we have that for any fixed $L\subset[d], I,J\subset[n]$,
\begin{equation}
\P\left( e_L(I,J) \le \frac{|L||I||J|}{2n}\right) \le 2\expo{ -\frac{1}{10 n}|L||I||J|}.
\end{equation}
Combining this with the union bound,
\begin{align*}
\P( \cD(k_0,n_0)^c) &= \P\left\{ \exists L\subset [d], I,J\subset[n]: \;|L|\ge k_0, |I|,|J| \ge n_0,\; e_L(I,J) \le \frac{|L||I||J|}{2n}\right\} \\
& \le 2^{d+1}4^n \expo{ -\frac{ k_0 n_0^2}{10 n}}.
\end{align*}
Since $d \le n$ the result immediately follows. 
\end{proof}

\corN{
\begin{rmk}
It is interesting to note that the dual property that $S_n^d$ has no \emph{dense} patches with high probability was a crucial ingredient in the work of Kahn--Szemer\'edi \cite{FKS:gap} on the mirror problem of proving an upper tail bound for the \emph{second largest} singular value of $S_n^d$ (i.e.\ the operator norm of the centered matrix $S_n^d-\f{d}n\un\un^*$).
\end{rmk}
}

\subsection{Good overlap via an averaging argument}	\label{sec:avg}

In this and the next subsection we make use of the following notation: for distinct $i_1,i_2\in [n]$ we denote
\begin{equation}	\label{def:cF}
\cF(i_1,i_2):=  \langle \{\pi^\ell_n(i): i \ne i_1,i_2\}\rangle
\end{equation}
that is, the sigma algebra of events generated by all but the $i_1$-st and $i_2$-nd rows of each permutation matrix $P_n^\ell$, $\ell\in [d]$.

\begin{dfn}[Good overlap events]	\label{def:goodo}
For $i_1,i_2\in [n]$ distinct, $\rho,t>0$ and $k\ge 1$, we define the \emph{good overlap event} $\cO_{i_1,i_2}(k,\rho,t)$ to be the event that there exist $u\in \sph\cap\oneperp$ and $L\subset[d]$ with $|L|\ge k$ such that the following properties hold:
\begin{enumerate}[(a)]
\item $ \big| u_{\pi^\ell_n(i_1)} - u_{\pi^\ell_n(i_2)} \big| \ge \frac{\rho}{\sqrt{n}}$ for all $\ell\in L$,
\item $\big\|(S_n^d+Z_n)^{(i_1,i_2)}u \big\|_2 \le \frac{t}{\sqrt{n}}$, and
\item $\big|(R_{i_1}+R_{i_2}) \cdot u\big| \le \frac{2t}{\sqrt{n}} $.
\end{enumerate}
Here $(S_n^d+Z_n)^{(i_1,i_2)}$ denotes the matrix obtained by removing rows $i_1,i_2$ from $S_n^d+Z_n$.
We note that the event $\cO_{i_1,i_2}(k,\rho,t) $ is $\cF(i_1,i_2)$-measurable. Indeed, conditioning on $\cF(i_1,i_2)$ fixes the $(S_n^d+Z_n)^{(i_1,i_2)}$ as well as the pairs $\{\pi^\ell_n(i_1),\pi_n^\ell(i_2)\}_{\ell\in [d]}$, and the latter determine the vector $R_{i_1}+R_{i_2}$ and the differences $\{| u_{\pi^\ell_n(i_1)} - u_{\pi^\ell_n(i_2)} |\}_{\ell\in [d]}$.

For each pair of distinct indices $i_1,i_2\in [n]$
we choose an $\cF(i_1,i_2)$-measurable random vector $u^{(i_1,i_2)}\in \sph\cap \oneperp$ and an $\cF(i_1,i_2)$-measurable random set $L(i_1,i_2)\subset[d]$ which, on the event $\cO_{i_1,i_2}(k,\rho,t)$, satisfy the stated properties (a)--(c) for $u, L$; off this event we define $u^{(i_1,i_2)}$ and $L(i_1,i_2)$ arbitrarily (but in an $\cF(i_1,i_2)$-measurable way).
\end{dfn}

For $m\ge1$ and $\rho,t>0$ we define the ``good" event that $(S_n^d+Z_n)$ is well-invertible over mean-zero flat vectors:
\begin{equation}	\label{def:notflat}
\cG(m,\rho,t) := \left\{ \forall u, v\in \Flat_0(m,\rho),\; \min(\|(S_n^d+Z_n)u\|_2,\|(S_n^d+Z_n)^*v\|_2)> \frac{t}{\sqrt{n}}\right\}.
\end{equation}


\begin{lem}[Good overlap on average]	\label{lem:avg}
Assume $d\ge1$ and let $1\le m\le \frac{c_{\ref{lem:gap}} }{1+c_{\ref{lem:gap}}}n$.
For all $\rho>0$ and $0<t\le |d+\zeta|\sqrt{n}$,
\begin{align}
&\P\left(\left\{ s_n(S_n^d+Z_n)< \frac{t}{\sqrt{n}}\right\}\cap \cG(m,\rho,t) \cap \cD\left( \frac{c_{\ref{lem:avg}}md}{n},\frac{m}{4}\right) \right) \notag\\
&\qquad\qquad \le \frac{2}{mn} \sum_{i_1,i_2=1}^n \P\left( \cO_{i_1,i_2}\left( \frac{c_{\ref{lem:avg}}md}{n},\frac{\rho}{4},t\right) \cap \left\{ \big| (R_{i_1}-R_{i_2})\cdot u^{(i_1,i_2)}\big| \le \frac{8t}{\rho}\right\}\right)	\label{bd:avg}
\end{align}
for some absolute constant $c_{\ref{lem:avg}}>0$.
\end{lem}
\begin{rmk}
The condition $t\leq |d+\zeta|\sqrt{n}$ is needed in order to bypass
  the possibility that $\bf 1$ is an approximate minimal singular eigenvector
of $S_n^d+Z_n$. This can be best seen if one chooses $\zeta=-d$.
\end{rmk}

\begin{proof}[\corAa{Proof of Lemma \ref{lem:avg}}]
Suppose the event on the left hand side of \eqref{bd:avg} holds.
Let $u,v\in \sph$ be the respective eigenvectors of $(S_n^d+Z_n)^*(S_n^d+Z_n)$, $(S_n^d+Z_n)(S_n^d+Z_n)^*$ with eigenvalue
$(s_n(S_n^d+Z_n))^2$.
By our assumptions on $Z_n$ we have that $\un$ is also an eigenvector of these matrices with eigenvalue $|d+\zeta|^2$. Then since
\[
s_n(S_n^d+Z_n) < \frac{t}{\sqrt{n}} \le |d+\zeta|
\]
by assumption, it follows that $u$ and $\un$ are associated to distinct eigenvalues of $(S_n^d+Z_n)^*(S_n^d+Z_n)$ and hence $u\perp \un$; we similarly have that $v\perp \un$.
We have thus located vectors $u,v\in \sph\cap\oneperp$ such that
\begin{equation}	\label{avg:assume.uv}
\|(S_n^d+Z_n)u\|_2, \, \|(S_n^d+Z_n)^*v\|_2 \le \frac{t}{\sqrt{n}}.
\end{equation}
Furthermore, by the restriction to $\cG(m,\rho,t)$ we have that $u,v\in \sph\cap\oneperp\cap \Flat(m,\rho)^c$.

In the first stage of the proof, we show that there is a large number of ``good'' pairs $(i_1,i_2)\in [n]^2$ such that (1) $|v_{i_1}-v_{i_2}|$ is reasonably large, and (2) $|u_{\pi^\ell_n(i_1)}-u_{\pi^\ell_n(i_2)}|$ is reasonably large for a large number of $\ell\in [d]$.

We begin with (2), counting pairs $(i_1,i_2)$ that are ``good" with respect to $u$.
Since $u \in \sph \setminus \Flat(m,\rho)$, by Lemma \ref{lem:gap} there exist disjoint sets $J_1,J_2\subset[n]$ with $|J_1|=m$ and
\begin{equation}	\label{J2lb}
|J_2|\ge c_{\ref{lem:gap}}(n-m) \ge \f{c_{\ref{lem:gap}}}{1+ c_{\ref{lem:gap}}} n \ge m
\end{equation}
such that
\begin{equation}	\label{avg:Js}
|u_{j_1}-u_{j_2}| \ge \frac{\rho}{4\sqrt{n}} \quad \forall j_1\in J_1,\; j_2\in J_2.
\end{equation}
For $i\in [n]$ and $\alpha\in \{1,2\}$, write
\[
L_\alpha(i) := \{ \ell \in [d]: \pi^\ell_n(i) \in J_\alpha\}.
\]
Fixing $c_{\ref{lem:avg}} < \frac{c_{\ref{lem:gap}}}{4(1+c_{\ref{lem:gap}})}$, define
\begin{equation}	\label{def:mI}
\cI(u) := \left\{ (i_1,i_2)\in [n]^2: \; |L_1(i_1)\cap L_2(i_2)| > \frac{c_{\ref{lem:avg}}dm}{n} \right\}.
\end{equation}
We will use our restriction to the
no-holes
event $\cD(c_{\ref{lem:avg}}md/n,m/4)$ to show that $\cI(u)$ is large. 
First, let
\[
I_1:= \left\{ i\in [n]: |L_1(i)| \ge \frac{dm}{2n}\right\}.
\]
We claim
\begin{equation}	\label{avg:claim1}
|I_1| > n-\frac{m}4.
\end{equation}
Indeed, suppose $|I_1^c| \ge m/4$. By our restriction to $\cD(c_{\ref{lem:avg}}md/n, m/4)$ and the fact that $|J_1|=m> m/4$, we have
\[
\frac{d|I_1^c|m}{2n} \le e_{[d]}(I_1^c,J_1) = \sum_{i\in I_1^c} |L_1(i)| < \frac{dm|I_1^c|}{2n},
\]
a contradiction. Hence, \eqref{avg:claim1} holds.
Now for $i_1\in [n]$ let
\[
I_2(i_1) := \left\{ i\in [n]: |L_1(i_1)\cap L_2(i)| \ge \frac{c_{\ref{lem:avg}}dm}{n}\right\}.
\]
We claim that for any $i_1\in I_1$,
\begin{equation}	\label{avg:claim2}
|I_2(i_1)| > n-\frac{m}4.
\end{equation}
Indeed, suppose towards a contradiction that $|I_2(i_1)^c|\ge m/4$ for some $i_1\in I_1$. 
From \eqref{J2lb} we have $|J_2|\ge m$, so by our restriction to $\cD(c_{\ref{lem:avg}}md/n, m/4)$,
\[
\frac{|L_1(i_1)||I_2(i_1)^c||J_2|}{2n} \le e_{L_1(i_1)}(I_2(i_1)^c, J_2) = \sum_{i\in I_2(i_1)^c} |L_1(i_1)\cap L_2(i)| < |I_2(i_1)^c| \frac{c_{\ref{lem:avg}} dm}{n},
\]
which rearranges to
\[
|L_1(i_1)||J_2|< 2c_{\ref{lem:avg}}dm.
\]
Since $|J_2|\ge \f{c_{\ref{lem:gap}}}{1+c_{\ref{lem:gap}}}n$ and $c_{\ref{lem:avg}} < \frac{c_{\ref{lem:gap}}}{4(1+c_{\ref{lem:gap}})}$, we have $|L_1(i_1)|< dm/2n$, which contradicts the fact that $i_1\in I_1$. This establishes \eqref{avg:claim2}.
From \eqref{avg:claim1} and \eqref{avg:claim2} it follows that
\begin{equation}	\label{my:bound}
|\cI(u)|  \ge |\{(i_1,i_2): i_1\in I_1, i_2\in I_2(i_1)\}| > \left( n- \frac{m}{4}\right)^2 > n^2 - \frac{mn}{2}.
\end{equation}

Now we count pairs that are ``good" with respect to $v$.
For $i_1\in [n]$ write
\[
J_v(i_1) :=\left\{ i\in [n]: |v_{i_1}-v_i| \ge \frac{\rho}{\sqrt{n}} \right\}.
\]
Since $v\in \sph\setminus \Flat(m,\rho)$ we must have that $|J_v(i_1)|> m$ for any $i_1\in [n]$.
Indeed, suppose $|J_v(i_1)|\le m$ for some $i_1\in [n]$. Denoting $w:=(v-v_{i_1}\un)_{J_v(i_1)}$ \corAa{(for any vector $v'$ and $J' \in [n]$ we write $v'_{J'}$ to denote the projection of the vector $v'$ onto coordinates indexed by $J'$)}, we have
\[
\|v-v_{i_1}\un - w\|_2 = \|(v-v_{i_1}\un)_{J_v(i_1)^c}\|_2 < \rho.
\]
But since $w\in \Sparse(m)$ this contradicts the assumption that $v\notin \Flat(m,\rho)$.
Thus, putting
\[
 \tilde{\cI}(v):= \left\{ (i_1,i_2)\in [n]^2: |v_{i_1}-v_{i_2}| \ge \frac{\rho}{\sqrt{n}} \right\}
\]
we have $|\tilde{\cI}(v) | = \sum_{i_1\in [m]} |J_v(i_1)| \ge nm$.
Set
\[
\cI'(u,v) := \cI(u) \cap \tilde{\cI}(v).
\]
Using the bound \eqref{my:bound} we have
\begin{equation}	\label{myp:bound}
|\cI'(u,v)| \ge |\tilde{\cI}(v)| - |\tilde{\cI}(v) \setminus \cI(u)| \ge |\tilde{\cI}(v)| - |\cI(u)^c| \ge mn - \frac{mn}{2} = \frac{mn}{2}.
\end{equation}

Now we show that $\cO_{i_1,i_2} ( c_{\ref{lem:avg}}md/n, \rho/4,t)$ holds for all $(i_1,i_2)\in \cI'(u,v)$ (in fact it holds for all $(i_1,i_2)\in \cI(u)$).
Indeed, the vector $u$ and the set $L=L_1(i_1)\cap L_2(i_2)$ witness the conditions (a)--(c) from Definition \ref{def:goodo}, as we now demonstrate.
The condition that $|L| \ge c_{\ref{lem:avg}}md/n$ follows from the definition of $\cI(u)$.
The condition (a) follows from \eqref{avg:Js} and the definitions of $L_1(i_1), L_2(i_2)$.
Finally, (b) and (c) follow easily from \eqref{avg:assume.uv} and the triangle inequality:
\[
\|(S_n^d+Z_n)^{(i_1,i_2)} u\|_2 \le \|(S_n^d+Z_n)u\|_2 \le \frac{t}{\sqrt{n}},
\]
\[
|(R_{i_1} + R_{i_2})\cdot u| \le |R_{i_1}\cdot u | + |R_{i_2}\cdot u| \le 2\|(S_n^d+Z_n)u\|_2 \le \frac{2t}{\sqrt{n}}.
\]
A key point here is that while $u$ and $L=L_1(i_1)\cap L_2(i_2)$ witness that the event $\cO_{i_1,i_2}( c_{\ref{lem:avg}}md/n, \rho/4,t)$ holds, we cannot take these to be $u^{(i_1,i_2)}$ and $L(i_1,i_2)$, respectively, as $u$ and $L$ are not themselves measurable with respect to $\cF(i_1,i_2)$.

Now it remains to show that occurrence of all the events on the left hand side of \eqref{bd:avg} implies also the occurrence of the event $\{ \big| (R_{i_1}-R_{i_2})\cdot u^{(i_1,i_2)}\big| \le \frac{8t}{\rho}\}$ for all $(i_1,i_2)\in \cI'(u,v)$.
By several applications of the Cauchy--Schwarz inequality and the fact that $\cO_{i_1,i_2}(c_{\ref{lem:avg}}md/n, \rho/4,t)$ holds, we have
\begin{align*}
\frac{t}{\sqrt{n}}  \ge \big\|v^*(S_n^d+Z_n)\big\|_2 &\ge \big|v^*(S_n^d+Z_n)u^{(i_1,i_2)}\big|   \\
&= \left| \sum_{i=1}^n \overline{v}_i R_i\cdot u^{(i_1,i_2)}\right|	\\
&\ge \big|(\overline{v}_{i_1} R_{i_1}+ \overline{v}_{i_2}R_{i_2})\cdot u^{(i_1,i_2)}\big| - \big\|(S_n^d+Z_n)^{(i_1,i_2)}u^{(i_1,i_2)}\big\|_2\\
&\ge \big|(\overline{v}_{i_1} R_{i_1}+ \overline{v}_{i_2}R_{i_2})\cdot u^{(i_1,i_2)}\big| - \frac{t}{\sqrt{n}},
\end{align*}
which implies that $|(\overline{v}_{i_1} R_{i_1}+ \overline{v}_{i_2}R_{i_2})\cdot u^{(i_1,i_2)}| \le \f{2t}{\sqrt{n}}$. Using the triangle inequality, recalling the definition of $\tilde{\cI}(v)$, and using the fact that $\max_i |v_i| \le \|v\|_2 =1$, we further obtain
\begin{align*}
 \left|(\overline{v}_{i_1} R_{i_1}+ \overline{v}_{i_2}R_{i_2})\cdot u^{(i_1,i_2)}\right|
&=\,  \frac12\left|(\overline{v}_{i_1} + \overline{v}_{i_2})(R_{i_1}+R_{i_2}) \cdot u^{(i_1,i_2)} + (\overline{v}_{i_1}-\overline{v}_{i_2})(R_{i_1}-R_{i_2})\cdot u^{(i_1,i_2)}\right| \\
&\ge  \,  \frac12 \left|(\overline{v}_{i_1}-\overline{v}_{i_2})(R_{i_1}-R_{i_2})\cdot u^{(i_1,i_2)}\right| - \left|(R_{i_1}+R_{i_2}) \cdot u^{(i_1,i_2)}\right|\\
&\ge  \,  \frac12 \left|(\overline{v}_{i_1}-\overline{v}_{i_2})(R_{i_1}-R_{i_2})\cdot u^{(i_1,i_2)}\right| - \frac{2t}{\sqrt{n}}\\
&\ge  \frac{\rho}{2\sqrt{n}} \left|(R_{i_1}-R_{i_2})\cdot u^{(i_1,i_2)}\right|-\frac{2t}{\sqrt{n}},
\end{align*}
where in the second-to-last inequality we have used the property (c) of the event $\cO_{i_1,i_2}(c_{\ref{lem:avg}}md/n, \rho/4,t)$. Combining and rearranging we have
\[
\big|(R_{i_1}-R_{i_2})\cdot u^{(i_1,i_2)}\big|\le \frac{8t}{\rho}.
\]

We have thus shown that on the event $\cE:=\left\{ s_n(S_n^d+Z_n)\le \frac{t}{\sqrt{n}}\right\}\cap \cG(m,\rho,t) \cap \cD\left( \frac{c_{\ref{lem:avg}} md}{n},\frac{m}{4}\right)$, the event $\cE(i_1,i_2):=\cO_{i_1,i_2}\left( \frac{cmd}{n}, \frac{\rho}2,t\right)\cap\left\{\big|(R_{i_1}-R_{i_2})\cdot u^{(i_1,i_2)}\big|\le \frac{8t}{\rho}\right\}$ holds for at least $mn/2$ values of $(i_1,i_2)\in [n]^2$.
By double counting,
\[
\sum_{i_1,i_2=1}^n \bI_{\cE(i_1,i_2)} \ge \frac{mn}{2} \bI_{\cE}.
\]
Taking expectations on both sides and rearranging yields the desired bound.
\end{proof}

\subsection{Anti-concentration for random walks}	\label{sec:walk}

In the previous section we essentially reduced our task to obtaining an anti-concentration estimate for the random variable $(R_{i_1}-R_{i_2})\cdot u^{(i_1,i_2)}$ for a fixed pair of distinct indices $i_1,i_2\in [n]$. We accomplish this in the following lemma (recall our notation \eqref{def:cF}).

\begin{lem}[Anti-concentration for row-pair random walk]	\label{lem:walk}
Let $i_1,i_2\in[n]$ be distinct, and suppose $\cO_{i_1,i_2}(k,\rho,t)$ holds for some $k\ge1$, $\rho,t>0$.
Then for all $r\ge 0$,
\begin{equation}
\P\left\{ \big| (R_{i_1}-R_{i_2})\cdot u^{(i_1,i_2)} \big| \le r \;\middle|\; \cF(i_1,i_2)\right\}
\le \corAa{C_{\ref{lem:walk}}} \left( 1+ \frac{r\sqrt{n}}{\rho}\right) \left(\frac{\log(n/\rho)}{k}\right)^{1/2}
\end{equation}
for some absolute constant \corAa{$C_{\ref{lem:walk}}$}.
\end{lem}

\begin{rmk}
In the proof we will only use the lower bound $|L(i_1,i_2)|\ge k$ and property (a) for $u^{(i_1,i_2)}$ and $L(i_1,i_2)$ from Definition \ref{def:goodo}, which is why the bound is independent of the parameter $t$.
\end{rmk}

We will need the following standard anti-concentration bound of Berry--Ess\'een-type; see for instance \cite[Lemma 2.7]{Cook:ssv} (the condition there of $\kappa$-controlled second moment is easily verified to hold with $\kappa=1$ for a Rademacher variable).

\begin{lem}[Berry--Ess\'een-type small-ball inequality]	\label{lem:BE}
Let $v\in \C^n$ be a fixed nonzero vector and let $\xi_1,\dots, \xi_n$ be independent Rademacher variables.
There exists an absolute constant $C_{\ref{lem:BE}}$ such that for any $r\ge 0$,
\[
\sup_{z\in \C} \P\bigg( \bigg| z+ \sum_{j=1}^n \xi_j v_j \bigg| \le r\bigg) \le C_{\ref{lem:BE}}\left(\frac{r+ \|v\|_\infty}{\|v\|_2}\right).
\]
\end{lem}

\vskip10pt

\begin{proof}[Proof of Lemma \ref{lem:walk}]
By symmetry we may take $(i_1,i_2)= (1,2)$.
Condition on a realization of $\{\pi^\ell(i): i\notin \{1,2\}, 1\le \ell \le d\}$ such that $\cO_{i_1,i_2}(k,\rho,t)$ holds. This fixes the vector $u^{(1,2)}$ and the set $L(1,2)\subset[d]$.
For ease of notation we write $u=u^{(1,2)}$ and $L= L(1,2)$ for the remainder of the proof.
Let $r\ge 0$. Our aim is to show
\begin{equation}	\label{walk:goal1}
\P\left( \big| (R_{1}-R_{2})\cdot u \big| \le r
\;\middle|\; \corAa{\cF(1,2)} \right)
\le C\left( 1+ \frac{r\sqrt{n}}{\rho}\right) \left(\frac{\log(n/\rho)}{k}\right)^{1/2}
\end{equation}
for some sufficiently large constant $C$.
Let $\xi_1,\dots, \xi_d$ be i.i.d.~Rademacher variables, independent of all other variables, and for each $\ell\in [d]$ put
\[
\tpi^\ell_n := \pi^\ell_n \circ \uptau_{(1,2)}^{\frac12(\xi_\ell+1)}
\]
where we recall $\uptau_{(i_1,i_2)}$ denotes the transposition that switches $i_1,i_2$, and we interpret $\uptau_{(i_1,i_2)}^1= \uptau_{(i_1,i_2)}$, $\uptau_{(i_1,i_2)}^0 = \Id$.
Now let $\tA^d$ be as in \eqref{def:S} but with each $\pi^\ell_n$ replaced by $\tpi^\ell_n$.
By the Haar distribution of $\pi^1_n,\dots, \pi^d_n$ and their independence from the Rademacher variables $\xi_1,\dots, \xi_d$, we have that $\tA^d\eqd S_n^d$,
even under conditioning on $ \cF(i_1,i_2)$. Moreover, it is clear from the construction that $\tpi^\ell_n(i)= \pi^\ell_n(i)$ for all $3\le i\le n$ and $1\le \ell\le d$, so that $\tA^d$ agrees with $S_n^d$ on the third through $n$-th rows. We denote the first two rows of $\tA^d$ by $\tR_1$ and $\tR_2$.
By replacing $S_n^d$ with $\tA^d$ in \eqref{walk:goal1}, it now suffices to show
\begin{equation}	\label{walk:goal2}
\P\left( \big| (\tR_{1}-\tR_{2})\cdot u \big| \le r
\;\middle|\; \cF(1,2) \right)
\le C \left( 1+ \frac{r\sqrt{n}}{\rho}\right) \left(\frac{\log(n/\rho)}{k}\right)^{1/2}.
\end{equation}
Turning to prove \eqref{walk:goal2} we note
\begin{align}
(\tR_{1}-\tR_{2})\cdot u
&= \sum_{\ell\in [d]} u_{\tpi^\ell_n(1)}-u_{\tpi^\ell_n(2)}	\notag\\
&= \sum_{\ell\in [d]} \big( u_{\pi^\ell(1)} - u_{\pi^\ell(2)}\big) \bI(\xi_\ell = -1) - \big( u_{\pi^\ell(1)} - u_{\pi^\ell(2)}\big) \bI(\xi_\ell = +1)	\notag\\
&=\sum_{\ell\in [d]}\xi_\ell \partial_\ell(u),	\label{walk:express}
\end{align}
where
\begin{equation}	\label{def:partial}
\partial_\ell(u):=u_{\pi^\ell(2)} - u_{\pi^\ell(1)}.
\end{equation}
For $j\ge-1$ let
\[
L^{(j)} := \big\{\, \ell\in L: 2^{-(j+1)} < |\partial_\ell(u)| \le 2^{-j}\,\big\}.
\]
By condition (a) in Definition \ref{def:goodo} we have that $|\partial_\ell(u)|\ge \rho/\sqrt{n}$ for all $\ell\in L$. Therefore,
\[
L\subset \bigcup_{j=-1}^{\log_2(\sqrt{n}/\rho)} L^{(j)}.
\]
Since $|L|=k$ by the pigeonhole principle there must exists some $j_\star$ such that
\[
|L^{(j_\star)}| \ge \frac{k}{2\log_2(\sqrt{n}/\rho)}.
\]
Set
\[
v:= \big( \partial_\ell(u) \bI(\ell\in L^{(j_\star)})\big)_{\ell\in [d]}\in \C^d.
\]
For all $\ell\in L^{(j)}$ we have
$|v_\ell| \ge \rho/\sqrt{n}	$
and so
\begin{equation}	\label{norm.lb}
\|v\|_2 \ge \frac{\rho}{\sqrt{n}}|L^{(j_\star)}|^{1/2} \ge \frac{\rho}{\sqrt{n}} \left( \frac{k}{2 \log_2(\sqrt{n}/\rho)}\right)^{1/2}.
\end{equation}
Moreover, since the components of $v$ vary by at most a factor of $2$ on $L^{(j_\star)}$ we also have
$|v_\ell| \ge \|v\|_\infty /2$
for all $\ell \in L^{(j_\star)}$. Therefore
\begin{equation}	\label{norm.ub}
\|v\|_\infty \le \frac{2\|v\|_2}{|L^{(j)}|^{1/2}} \le \left( \frac{8\log_2(\sqrt{n}/\rho)}{k}\right)^{1/2} \|v\|_2.
\end{equation}
Conditioning on $\{\pi^\ell: \ell\in [d]\}$ and applying Lemma \ref{lem:BE}, we have
\begin{align*}
  \sup_{z\in \C}\P_{L^{(j_\star)}}
  \left( \bigg| z+ \sum_{\ell\in L^{(j_\star)}} \xi_\ell \partial_\ell(u)\bigg| \le r\right)
&= \sup_{z\in \C} \P_{L^{(j_\star)}}
\left( \bigg| z+ \sum_{\ell=1}^d \xi_\ell v_\ell\bigg| \le r\right) \\
&\le C_{\ref{lem:BE}} \left(\frac{r}{\|v\|_2} + \frac{\|v\|_\infty}{\|v\|_2}\right)\\
&\le C_{\ref{lem:BE}} \left( 1+ \frac{r\sqrt{n}}{\rho}\right) \left( \frac{ 8 \log_2(\sqrt{n}/\rho)}{k}\right)^{1/2},
\end{align*}
where $\P_{L^{(j_\star)}}$ denotes the law of $\{\xi_\ell\}_{\ell\in L^{(j_\star)}}$.
Applying this bound to the expression \eqref{walk:express} (after conditioning on $\{\xi_\ell: \ell\notin L^{(j)}\}$ and absorbing the resulting deterministic summands into the scalar $z$), we obtain \eqref{walk:goal2} as desired.
\end{proof}

\subsection{Proof of Theorem \ref{thm:ssv}}
\label{sec:ssv_conclude}	

Now we combine the results of this section and Section \ref{sec:structured} to complete the proof of Theorem \ref{thm:ssv}.
Fix $\gamma_0\ge 1$ and let $\Gamma_0= \ol{C}_{\ref{thm:ssv}} \gamma_0\log_dn$ with $\ol{C}_{\ref{thm:ssv}}$ an absolute constant to be chosen sufficiently large.
We may and will assume that $n$ is sufficiently large depending on $\gamma_0$.
By Remark \ref{rmk:reduced} we may assume
\begin{equation}	\label{mayassumed}
\log^8n \le d\le n
\end{equation}
(the desired bound holds trivially for smaller values of $d$).
Recall the boundedness event $\cB(K)$ from \eqref{def:bdd}.
From our hypotheses and the fact that $\|S_n^d\|_\oneperp \le \|S_n^d\|= d$ we have
\[
\|S_n^d+Z_n\|_\oneperp \le \|S_n^d\|_\oneperp + \|Z_n\|_\oneperp  \le d+ n^{\gamma_0}\le 2n^{\gamma_0}\le n^{\gamma_0}\sqrt{d}.
\]
Thus the event $\cB(n^{\gamma_0})$ holds. 

Set
\begin{equation}	\label{choices.final}
m = \frac{c_{\ref{prop:struct}}n}{\gamma_0\log^2n}\;, \qquad  \rho= n^{-C_{\ref{prop:struct}}\gamma_0\log_dn}\;, \qquad t= \sqrt{n}(n^{-\Gamma_0}\wedge |d+\zeta|).
\end{equation}
Now using Lemma \ref{lem:avg} we have
\begin{align}
&\P\left( s_n(S_n^d+Z_n) < n^{-\Gamma_0}\wedge |d+\zeta|\right)
\le \P\left( \cG(m,\rho, t)^c\right)  +  \P\left( \cD\left(  \frac{c_{\ref{lem:avg}}md}{n}, \frac{m}{4} \right)^c\right) \notag
\\
&\qquad\qquad+ \frac{2}{mn} \sum_{i_1,i_2=1}^n \P\left\{ \cO_{i_1,i_2}\left( \frac{c_{\ref{lem:avg}}m d}{n},\frac{\rho}{4},t\right) \cap \left\{ \big| (R_{i_1}-R_{i_2})\cdot u^{(i_1,i_2)}\big| \le \frac{8t}{\rho}\right\}\right\}.		\label{final.breakdown}
\end{align}
Taking $\ol{C}_{\ref{thm:ssv}}\ge C_{\ref{prop:struct}}+1$ we have $t\le \rho$. Then by Proposition \ref{prop:struct} we see that
\beq\label{eq:T1_bd}
 \P\left( \cG(m,\rho, t)^c\right) \le e^{-\ol{c}_{\ref{prop:struct}} d}.
\eeq
Using Lemma \ref{lem:discrepancy} and 
the lower bound in \eqref{mayassumed} (here we only need $d=\omega(\log^6n)$), 
we see that
\beq\label{eq:T2_bd}
 \P\left( \cD\left(  \frac{c_{\ref{lem:avg}}md}{n}, \frac{m}{4} \right)^c\right)\le e^{-n}.
\eeq
Next applying Lemma \ref{lem:walk} yields that the third term in \eqref{final.breakdown} is bounded by
\begin{align}\label{eq:T3_bd}
&\frac{2}{mn}\times n^2 \times C_{\ref{lem:walk}} \left( 1+ \frac{128  
  t\sqrt{n}}{\rho^2}\right) \left( \log n + \log\frac4\rho\right)^{1/2} \sqrt{\frac{n}{c_{\ref{lem:avg}}md}} 	\notag\\
&\qquad\le \frac{256 C_{\ref{lem:walk}}}{\sqrt{c_{\ref{lem:avg}}}}
\frac1{\sqrt{d}}\left(\frac{n}{m}\right)^{3/2}\left( 1+ n^{-\Gamma_0+1+2C_{\ref{prop:struct}}\gamma_0\log_dn}\right) (2+C_{\ref{prop:struct}} \gamma_0\log_dn)^{1/2}(\log n)^{1/2}	\notag\\
&\qquad\le  \frac{C_{\gamma_0} \log^4n}{\sqrt{d}}\left( 1+ n^{-\Gamma_0+1+2C_{\ref{prop:struct}}\gamma_0\log_dn}\right)
\end{align}
for some constant $C_{\gamma_0}$ depending only on $\gamma_0$.
Taking $\ol{C}_{\ref{thm:ssv}}\ge 3C_{\ref{prop:struct}}$ and combining \eqref{final.breakdown}--\eqref{eq:T3_bd} we conclude
\begin{align}\label{eq:t_le_rho}
\P\left( s_n(S_n^d+Z_n) < n^{-\Gamma_0}\wedge |d+\zeta|\right) &\le e^{-\ol{c}_{\ref{prop:struct}}d} + e^{-n} + \frac{2C_{\gamma_0}\log^4n}{\sqrt{d}}\notag\\
\le & \frac{3C_{\gamma_0}\log^4n}{\sqrt{d}}.
\end{align}
The proof of Theorem \ref{thm:ssv} is now complete.

\section{Control on traces}
\label{sec-traces}
In this short section, we derive simple estimates on traces for
permutation matrices and for $S_n^d (S_n^d)^*$.
We begin with the following simple estimate. Let $\pi_n$ be a random, uniformly
chosen permutation on $[n]$, and let $P_n$ denote the corresponding permutation
matrix.
\begin{lem}
  \label{lem-singlepertraces}
With notation as above,
\begin{equation}
  \label{eq-perm1}\P(\Tr P_n\geq  k)\leq  \frac{1}{k!}, \;\;\; k\geq 1.
\end{equation}
\end{lem}

\vskip10pt

\begin{proof}
  Let $N_\ell$ denote the number of cycles of length $\ell$ in $\pi_n$.
Note that $\Tr P_n=N_1$.
Thus, the event
$\{\Tr P_n\geq k\}$ is the union of the events that
$k$ particular indices are fixed points
in the permutation $\pi_n$ and therefore
\begin{equation}
  \label{eq-perm1bis}
  \P(\Tr P_n\geq  k)=\P(N_1\geq k)\leq \left(\begin{array}{l} n\\ k\end{array}\right)
\frac{1}{n\cdot (n-1)\cdots (n-k+1)} = \frac{1}{k!}\,.
\end{equation}
\end{proof}

\vskip10pt

Let now $S_n^d$ be as in \eqref{def:S}.
We have the following lemma.

\vskip10pt

\begin{lem}
  \label{lem-tracebound}
With notation as above, there exists absolute constants  $c_{\ref{lem-tracebound}}$,  $C'_{\ref{lem-tracebound}}$, and
$C_{\ref{lem-tracebound}}$ so that
\begin{equation}
  \label{eq-perm2}\P(\Tr S_n^d (S_n^d)^*\geq nd+x d^2)\leq
  d e^{-d(x-e)}, \;\;\; x\geq e,
\end{equation}
for any $d \ge C_{\ref{lem-tracebound}}$.
In particular,  there exists an absolute constant $\ol{C}_{\ref{lem-tracebound}}$ so that
\begin{equation}
  \label{eq-perm2bbis}
  \E \Tr (S_n^d (S_n^d)^*)^2\leq 2nd^2+\ol{C}_{\ref{lem-tracebound}} d^4.
\end{equation}
\end{lem}
\begin{proof}
  Note that
  \begin{equation}
    \label{eq-perm4}
    S_n^d (S_n^d)^*= dI_n+\corN{\sum_{i\neq j=1}^d }P_n^{i} (P_n^{j})^*.
  \end{equation}
  Therefore, using that
  $P_n^{i}(P_n^{j})^*$ with $i\neq j$ is distributed
  like $P_n$, and that for fixed $i$ they are independent of each other,
  we get from \eqref{eq-perm1} that
\begin{align}
\P(\Tr S_n^d (S_n^d)^*\geq nd+x d^2)
\corN{\leq \P\left( \sum_{i\ne j = 1}^d \Tr P_n^i(P_n^j)^*\ge xd^2 \right)} &\corN{\leq d\,\P\left( \sum_{j=2}^d \Tr P_n^i(P_n^1)^* \ge xd \right) }\notag\\
&\leq
d\, \P\left(\sum_{i=1}^d \Tr P_n^{i}\geq xd\right). \label{eq-perm12}
\end{align}
From \eqref{eq-perm1} we have that
$\E(e^{\Tr P_n^{i}})\leq e^e$, and therefore, by
independence and Markov's inequality,
\beq
\label{eq-perm12bis}
\P\left(\sum_{i=1}^d \Tr P_n^{i}\geq xd\right)\leq e^{-xd}e^{ed}=e^{-(x-e)d}.\eeq
Substituting in \eqref{eq-perm12} we obtain that
\[\P(\Tr S_n^d (S_n^d)^*\geq nd+x d^2)\leq de^{-(x-e)d},\]
which completes the proof. 
\end{proof}

\vskip10pt

Note that Lemma \ref{lem-tracebound} together with \eqref{eq-perm1} imply that
with $Q_n = (z-S_n^d/\sqrt{d})(z-S_n^d/\sqrt{d})^*$,
\begin{equation}
  \label{eq-perm3}
  \P(\Tr Q_n> ((|z|^2+1)n+2|z|\sqrt{d}x+dx))\leq d e^{-c'dx},
\end{equation}
for some absolute constant $c'$, and $d$ and $x$ sufficiently large. Indeed,
$$ \Tr Q_n\leq |z|^2 n+ \f1d\Tr S_n^d (S_n^d)^*+2 |z| \f1{\sqrt{d}}\Tr S_n^d,$$
and the conclusion follows by a union bound and the estimates in
 \eqref{eq-perm2} and \eqref{eq-perm12bis}.

\section{Concentration for resolvent sub-traces}
\label{sec:conc_ineq}

In this section we derive concentration bounds on the traces of 
the diagonal and the off-diagonal blocks of the resolvent $\wt{G}(S_n^d)$. 
To prove Theorem \ref{thm:smallish_sing_control} we will need to consider 
the resolvent of $S_n^d$ shifted by some deterministic matrices. Hence, 
we introduce the following notation. Let $M_n:=M$ be 
a deterministic matrix of size $n \times n$. Fix 
$\xi \in \C \setminus \R$, $z \in \C$ and
define
\[
F^M(\xi)=: \begin{bmatrix}
F^M_{11}(\xi) & F^M_{12}(\xi) \\
F^M_{21}(\xi) & F^M_{22}(\xi)
\end{bmatrix}=:
 \wt{G}^M(S_n^d,\xi,z):= \left[\xi I - \begin{bmatrix} 0 & \left(z - \f{S_n^d}{\sqrt{d}}+M\right)\\
\left(z- \f{S_n^d}{\sqrt{d}}+M\right)^* & 0
\end{bmatrix} \right]^{-1}
.
\]
\begin{thm}\label{thm:hamming_conc}
Fix $z \in B_\C(0, R)$ and
$\xi \in \C \setminus \R$ such that $|\Im \xi| \le C_0$ for some 
$C_0 \ge 1$. Let $M_n:=M$ of size $n \times n$ be an $n \times n$ deterministic matrix with $\|M\| \le C_0$.
Then, for $i,j =1,2$ and $u \ge 0$ we have
\[
\P \left( \left| \f{1}{n}\Tr F_{ij}^M(\xi) - \E\left[\f{1}{n}\Tr F_{ij}^M(\xi)\right]\right| \ge u \right)
 \le 4 \exp \left(- c_{\ref{thm:hamming_conc}} n(\Im \xi)^4 u^2\right)
\]
for some constant $c_{\ref{thm:hamming_conc}}>0$, depending only on $C_0$.
\end{thm}
  The following is an immediate corollary of Theorem \ref{thm:hamming_conc}.
  \begin{cor}
    \label{cor-magic8}
   With notation as in Theorem \ref{thm:hamming_conc},
   there exists an $n_0$ so that if $\Im \xi>n^{-1/16}$ and
    $n>n_0$ then, for $i,j=1,2$,
    \beq
    \label{eq-magic8}
    \E\left[\left(\f{1}{n}\Tr F_{ij}^M(\xi) - \E\left[\f{1}{n}\Tr F_{ij}^M(\xi)\right]\right)^2\right]
      \leq \frac{1}{n^{3/4}}.
      \eeq
    \end{cor}
  We first prove Corollary \ref{cor-magic8} using Theorem \ref{thm:hamming_conc}. The proof of Theorem \ref{thm:hamming_conc} follows that.
  \begin{proof}[Proof of Corollary \ref{cor-magic8}]
    Let $Z\corA{:=}|\f{1}{n}\Tr F_{ij}(\xi) - \E\left[\f{1}{n}\Tr F_{ij}(\xi)\right]|$.
    Substituting $u=x  /n^{1/4}$ in Theorem \ref{thm:hamming_conc}
    gives that for $x>0$ we have
    \[\P(Z>u)\leq 4 \exp\corA{\left(-c_{\ref{thm:hamming_conc}}x^2 n^{1/4}\right)}.\]
    This completes the proof upon using integration by parts.
  \end{proof}
We next
establish Theorem \ref{thm:hamming_conc},
using a standard martingale approach. Specifically, we will apply a consequence of Azuma's inequality from \cite{Ledoux:phenom} that is conveniently phrased for our setting. This will reduce the task to bounding the change in $n^{-1}\Tr F_{ij}^M(\xi)$ under the application of a transposition to one of the permutations $\pi_n^\ell$.

Define the Hamming distance between two permutations $\pi,\sigma\in \bS_n$ as follows:
\beq\label{eq:hamm_dist}
d_H(\pi,\sigma):= \left| \left\{ i \in [n]: \pi(i) \neq \sigma(i) \right\} \right|.
\eeq
We extend to a Hamming metric on product space $\bS_n^d$ in the natural way: for two sequences $\bs\pi=(\pi^\ell)_{\ell\in [d]}$, $\bs \sigma=(\sigma^\ell)_{\ell\in [d]}$, set
\beq\label{eq:hamm_dist2}
d_H(\bs\pi,\bs\sigma):= \sum_{\ell=1}^d d_H(\pi^\ell,\sigma^\ell).
\eeq

\begin{lem}[Concentration for Hamming-Lipschitz functions]	\label{lem:hamming}
Let $f:\bS_n^d\to \C$ be an $L$-Lipschitz function with respect to the Hamming metric \eqref{eq:hamm_dist2}, and let $\bs\pi=(\pi^\ell)_{\ell\in [d]}$ be a uniform random element of $\bS_n^d$.
Then,
for any $u\ge0$,
\begin{equation}\label{eq:hamming}
\P( |f(\bs\pi)-\E f(\bs \pi)|\ge u) \le 4\exp\left( -\frac{u^2}{8 ndL^2}\right).
\end{equation}

\end{lem}

\begin{proof}
First we note that it is enough to prove that \eqref{eq:hamming} holds for $1$-Lipschitz function.
%
%
 Next, splitting $f(\bs\pi)$ into real and imaginary parts and applying the pigeonhole principle and the union bound, it suffices to show that for $f$ a real-valued $1$-Lipschitz function on $\bS_n^d$,
\beq \label{eq:Lip-goal-1}
\P( f(\bs\pi)-\E f(\bs \pi)\ge u) \le \exp\left( -\frac{u^2}{8nd}\right). 
\eeq
By Chebycheff's inequality, \eqref{eq:Lip-goal-1} would follow if, for any $\lambda>0$,
\begin{equation}
\label{eq-2018-a}
\E\left(e^{\lambda(f(\bs \pi)-\E f(\bs\pi))}\right)\leq e^{2nd \lambda^2}.
\end{equation} 
For $d=1$, the inequality \eqref{eq-2018-a} follows as in the proof of \cite[Corollary 4.3]{Ledoux:phenom}, using that in Lemma 4.1 there, one actually controls
the Laplace transform and not just the probabilities. To prove the case of general $d$,
we use tensorization. 
For an arbitrary  $1$-Lipschitz function $f: \mathbb{S}_n^d \mapsto \R$ and for $i \in [d]$, denote 
\[
h_i({\bm \pi}^{<i+1}):= \E[f| {\bm \pi}^{<i+1}]- \E[f| {\bm \pi}^{<i}],
\]
where we recall that ${\bm \pi}^{<k}:=(\pi^\ell)_{\ell \in [k-1]}$. For any fixed $i \in [d]$ and ${\bm \pi}^{<i}$,  the function
 $h_i$ viewed as a function of $\pi_i$  is a $1$-Lipschitz function with respect to the Hamming metric while $\E_i [h_i]=0$, where $\E_i$ denotes the expectation with respect to $\pi^i$. Therefore, applying the $d=1$ case of
 \eqref{eq-2018-a} we obtain, for any $i \in [d]$,
\[
\E \left[ \exp(\lambda h_i({\bm \pi}^{<i+1})) |{\bm \pi}^{<i}\right] \le \exp(2 \lambda^2 n).
\]
Since $f - \E f= \sum_{i=1}^d h_i$ and $h_i$ are measurable with respect to 
${\bm \pi}^{<i+1}$, iterating the above bound gives \eqref{eq-2018-a}.
\end{proof}

Lemma \ref{lem:hamming} reduces our task to showing the normalized traces of $F_{ij}(\xi)$ are $L$-Hamming-Lipschitz for an appropriate $L$. For this task we will make use of the following:

\begin{lem}[Resolvent identity]\label{lem:res-id}
Let $A$ and $B$ be two Hermitian matrices, and let $\xi \in \C \setminus \R$. Then
\[
(\xi - A)^{-1} -(\xi -B)^{-1}=(\xi-A)^{-1}(A-B)(\xi-B)^{-1}= (\xi-B)^{-1} (A-B) (\xi-A)^{-1}.
\]
More generally for any two invertible matrices $C$ and $D$, we have
\beq\label{eq:resolvent_2}
C^{-1} - D^{-1} = C^{-1} (D-C) D^{-1}.
\eeq
\end{lem}

\begin{proof}[Proof of Theorem \ref{thm:hamming_conc}]
Fix $i,j \in \{1,2\}$, $\ell \in [d]$ and set
$H_n(\xi):=\f{1}{n}\Tr F_{ij}^M(\xi)$.
As mentioned above we need to show that
$H_n(\cdot)$ is an $L$-Lipschitz function of $\bs\pi=(\pi_n^1,\dots, \pi_n^d)$ with
respect to the Hamming distance \eqref{eq:hamm_dist2} for an appropriate value of $L$. By the triangle inequality it suffices to show it is $L$-Lipschitz as a function of $\pi_n^\ell$ with respect to the Hamming distance \eqref{eq:hamm_dist} on $\bS_n$, for arbitrary fixed $\ell\in [d]$.

To this end, we define
\[
\wt{F}^M(\xi):=F^M(\xi,z,\wt{S}_n^d(\ell)), \text{ where } \wt{S}_n^d(\ell):=\sum_{k \in [d] \setminus \{\ell\}} \pi_n^k + \wt{\pi}_n^\ell,
\]
and $\wt{\pi}_n^\ell$ is some fixed but
arbitrary permutation over $[n]$. We similarly define
$\wt{F}_{ij}^M(\xi)$ and $\wt{H}_n(\xi)$. Now using the
resolvent identity we note that
\[
F^M(\xi) - \wt{F}^M(\xi) = \f{1}{\sqrt{d}} F^M(\xi)
\left(\Delta_n^\ell+ (\Delta_n^\ell)^*\right) \wt{F}^M(\xi),
\]
where
\[
\Delta_n^\ell:= \begin{bmatrix}
0 & (\wt{\pi}_n^\ell - \pi_n^\ell)\\
0  & 0
\end{bmatrix}.
\]
Therefore,
\begin{align}\label{eq:H_n}
H_n(\xi) - \wt{H}_n(\xi) = \f{1}{n\sqrt{d}} \Tr \left[ \begin{pmatrix} E_i^{\sf T} \\ {\bm 0}^{\sf T} \end{pmatrix} F^M(\xi) \left(\Delta_n^\ell+ (\Delta_n^\ell)^*\right) \wt{F}^M(\xi) (E_j \,\,  \, \, {\bm 0} )\right],
\end{align}
where
\[
E_1:= \begin{pmatrix} I_n\\ 0_n \end{pmatrix}, \quad E_2:= \begin{pmatrix}  0_n \\ I \end{pmatrix},  \quad {\bm 0}:= \begin{pmatrix} 0_n\\ 0_n \end{pmatrix},
\]
and $0_n$ is the $n \times n$ matrix of zeros.
To simplify \eqref{eq:H_n} further, we note that the $(k,n+k')$-th entry of $\Delta_n^\ell$ is non-zero for some $k, k' \in [n]$, if and only if $\pi_n^\ell(k) \ne \wt{\pi}_n^\ell(k)$ and one of $\pi_n^\ell(k)$ and $ \wt{\pi}_n^\ell(k)$ equals $k'$. Hence, using the triangle inequality and recalling the definition of $d_H(\cdot,\cdot)$, it follows that $|H_n(\xi)- \wt{H}_n(\xi)|$ is bounded by the sum of $4 d_H(\pi_n^\ell, \wt{\pi}_n^\ell)$ terms of the form
\beq\label{eq:H_n_one_term}
\f{1}{n\sqrt{d}}\left| \Tr \left[ \begin{pmatrix} E_i^{\sf T} \\ {\bm 0}^{\sf T}
\end{pmatrix} F^M(\xi) e_k e_{k'}^{\sf T}\wt{F}^M(\xi) (E_j \,\,  \, \,
{\bm 0} )\right] \right|,
\eeq
for some $k, k' \in [2n]$. Here $e_m$ denotes the canonical basis vector which has one in the $m$-th position.
Since $|\Im \xi |, \|M \| \le C_0$ we have the operator norm bounds
\[
\|F^M(\xi)\|,\|\wt{F}^M(\xi)\| \le \| M\| + |\Im \xi|^{-1} \le 2 C_0^2 |\Im \xi|^{-1} 
\]
As $\|E_i\|= 1$ for $i=1,2$, we have
\begin{align}
\left| \Tr \left[ \begin{pmatrix} E_i^{\sf T} \\ 
 {\bm 0}^{\sf T} \end{pmatrix} F^M(\xi) e_k e_{k'}^{\sf T}
 \wt{F}^M(\xi) (E_j \,\,  \, \, {\bm 0} )\right] \right| & = 
 \left|  e_{k'}^{\sf T}\wt{F}^M(\xi) (E_j \,\,  \, \, {\bm 0} )
 \begin{pmatrix} E_i^{\sf T} \\ {\bm 0}^{\sf T} \end{pmatrix} F^M(\xi) e_k 
 \right| \notag\\
 & \le \left\| \wt{F}^M(\xi) (E_j \,\,  \, \, {\bm 0} )\begin{pmatrix} 
   E_i^{\sf T} \\ {\bm 0}^{\sf T} \end{pmatrix} F^M(\xi)  \right\| \notag\\
 &  \le \f{4 C_0^4}{(\Im \xi)^2}. \label{eq:H_n_one_term_simplify}
\end{align}
Now combining \eqref{eq:H_n_one_term}-\eqref{eq:H_n_one_term_simplify} and
\eqref{eq:H_n}, we obtain
\beq\label{eq:H_n_lip}
|H_n(\xi) - \wt{H}_n(\xi)| \le \f{16 C_0^4 d_H(\pi_n^\ell, \wt{\pi}_n^\ell)}{n \sqrt{d} (\Im \xi)^2}.
\eeq
This shows that we can apply Lemma \ref{lem:hamming} with 
$f(\bs\pi)=H_n(\xi)$ and $L=16 C_0^4/n\sqrt{d}(\Im \xi)^2$, and the result follows.
\end{proof}

\section{Proof of the local law}
In this section we prove Theorem \ref{thm:smallish_sing_control}. The proof consists of two key components. First we derive an approximate fixed point equation for $\wt m_n(\xi)$, the Stieltjes transform of the symmetrized version of the empirical measure of the singular values of $z - S_n^d/\sqrt{d}$. Since the fixed point equation is an equation of degree three it is not apriori immediate that $\wt m_n(\xi)$ is close to the correct solution of the fixed point equation. To tackle this, we need certain properties of the roots of that cubic equation. We also need to employ a bootstrap argument to quantify the difference between $\wt m_n(\xi)$ and its limit $\wt m_\infty (\xi)$ when $\Im \xi$ approaches zero.

\subsection{Derivation of the approximate fixed point equation}\label{sec:sub-loop-eqn} The main technical result of this section is the following lemma.

\begin{lem}[Loop equation]\label{lem:loop-eqn}
Fix $\xi \in \C^+$ such that $ (\log n)^{-2} \le \Im \xi \le C_0$ for some $C_0 >0$. Fix $z \in B_\C(0,R)$ for some $R <\infty$. Then, there exists an event $\Omega_n(\xi)$ with
\[
\P (\Omega_n(\xi)^c) \le \exp \left(- c_{\ref{lem:loop-eqn}} (\log n)^2\right)
\]
such that on $\Omega_n(\xi)$ we have
\beq\label{eq:loop-eqn-wt}
\left| \wt m_n(\xi) ( \wt m_n(\xi) - \xi)^2 + (1-|z|^2)  \wt m_n(\xi) - \xi \right| \le C_{\ref{lem:loop-eqn}} \max\{d^{-1/2}, n^{-1/4}\log n\} (\Im \xi)^{-3} (1+ |\E \wt m_n(\xi)|),
\eeq
where $c_{\ref{lem:loop-eqn}}$ is an absolute constant and $C_{\ref{lem:loop-eqn}}$ depends only on $C_0$ and $R$.
\end{lem}

Since we have concentration bounds in Theorem \ref{thm:hamming_conc}, as we will see below, it will be enough to show that inequality \eqref{eq:loop-eqn-wt} holds for $\E \wt m_n(\xi)$. To show the same, it will be convenient to consider the Stieltjes transform of symmetrized version of the empirical measure of the singular values of $z - \wh S_n^d/\sqrt{d}$, where $\wh S_n^d$ is now centered. For ease of writing, let us denote ${\sf S}_n^d:=\sum_{\ell=1}^d {\sf P}_n^\ell$, where for $\ell \in [d]$,
\[
{\sf P}_n^\ell := \begin{bmatrix} 0 & (P_n^\ell - \E P_n^\ell) \\ (P_n^\ell - \E P_n^\ell)^* & 0 \end{bmatrix},
\]
and $\{P_n^\ell\}$ are i.i.d.~uniformly distributed permutation matrices. Define the resolvent as
\[
\wh{G}(S_n^d):=\wh{G}(S_n^d, \xi,z):= \left[ \xi I_{2n} -\begin{bmatrix} 0 & z I_n \\ \bar{z} I_n & 0 \end{bmatrix} + \f{{\sf S}_n^d}{\sqrt{d}} \right]^{-1}
\]
and denote $\wh m_n(\xi) := \f{1}{2n}\Tr \wh G (S_n^d)$.

\begin{lem}[Loop equation for the sum of centered permutation matrices]\label{lem:loop-eqn-exp}
Fix $\xi \in \C^+$ such that $ n^{-1/16}  \le \Im \xi \le C_0$ for some $C_0 >0$. Fix $z \in B_\C(0,R)$ for some $R <\infty$. Then, there exists a constant $C_{\ref{lem:loop-eqn-exp}}$, depending on $C_0$ and $R$, such that 
\beq\label{eq:loop-eqn}
\left|\E \wh m_n(\xi) (\E \wh m_n(\xi) - \xi)^2 + (1-|z|^2) \E \wh m_n(\xi) - \xi \right| \le C_{\ref{lem:loop-eqn-exp}} d^{-1/2} (\Im \xi)^{-3} (1+ |\E \wh m_n(\xi)|).
\eeq
\end{lem}

\vskip10pt

Recalling the definition of $\wt G(S_n^d)$ (see \eqref{eq:wt-G-def}) we observe that $\wt m_n(\xi)$ and $\wh m_n(\xi)$ are the normalized traces of the resolvent of two Hermitian matrices differed by a finite rank perturbation. Therefore, one can use the following result to bound the difference between $\wt m_n(\xi)$ and $\wh m_n(\xi)$. Its proof is a simple application of Cauchy's interlacing inequality. We include it for completeness.

\begin{lem}\label{lem:rank-stieltjes}
Let $A_i, \, i=1,2$, be two $n \times n$ Hermitian matrices such that ${\rm rank}(A_1-A_2) \le C_1$ for some absolute constant $C_1$. For $i=1,2$, and $\xi \in \C \setminus \R$, let $m_n^{A_i}(\xi)$ denote the Stieltjes transform of the empirical law of the eigenvalues of $A_i$. That is,
\[
m_n^{A_i}(\xi):= \int \f{1}{\xi-x}dL_{A_n^i}(x), \qquad i=1,2.
\]
Then 
\[
|m_n^{A_1}(\xi) - m_n^{A_2}(\xi)| \le \f{C_1 \pi}{n |\Im \xi|}.
\]
\end{lem}

\begin{proof}
Since
\(
\f{1}{\xi - x} = \int_{-\infty}^x \f{1}{(t-\xi)^2} dt
\) we observe that
\[
m_n^{A_i}(\xi) = \int_{-\infty}^\infty \int_{-\infty}^x \f{1}{(t-\xi)^2} dt dL_{A_n^i}(x) =\f{1}{n}\int_{-\infty}^\infty \f{\gn_i(-\infty, t]}{(\xi-t)^2}, 
\]
for $i=1,2$, where $\gn_i(-\infty,t]$ denotes the number of eigenvalues of $A_i$ in the interval $(-\infty,t]$. As ${\rm rank}(A_1 -A_2) \le C_1$, by Cauchy's interlacing inequality it also follows that $|\gn_1(-\infty,t]- \gn_2(-\infty,t]| \le C_1$. Therefore, 
\[
\left|m_n^{A_1}(\xi) - m_n^{A_2}(\xi)\right| \le \f{C_1}{n} \int \f{1}{|\xi- t|^2} dt =\f{C_1}{n} \int \f{1}{(\Im \xi)^2+ t^2} dt =    \f{C_1 \pi}{n |\Im \xi|}.
\]
\end{proof}

Equipped with Lemma \ref{lem:rank-stieltjes} and assuming Lemma \ref{lem:loop-eqn-exp} we now prove Lemma \ref{lem:loop-eqn}. 

\begin{proof}[Proof of Lemma \ref{lem:loop-eqn}]
Using Lemma \ref{lem:rank-stieltjes} and the trivial bounds $|\wh m_n(\xi)|, |\wt m_n(\xi)| \le 1/\Im \xi$ we obtain that $|\wt P (\E \wh m_n(\xi)) - \wt P (\E \wt m _n (\xi))| = O(n^{-1/2} (\Im \xi)^{-3})$. Therefore, Lemma \ref{lem:loop-eqn-exp} implies that 
\beq\label{eq:loop-eqn-wt-exp}
\wt P (\E \wt m_n(\xi)) = O\left(d^{-1/2} (\Im \xi)^{-3} (1+ |\E \wh m_n(\xi)|)\right) = O\left(d^{-1/2} (\Im \xi)^{-3} (1+ |\E \wt m_n(\xi)|)\right), 
\eeq
where we have used Lemma \ref{lem:rank-stieltjes} again and the fact that $n \Im \xi \ge 1$. It remains to show that
\beq\label{eq:loop-eqn-wt-exp-1}
\left| \wt P (\wt m_n(\xi)) - \wt P (\E \wt m_n(\xi))\right| = O\left(\f{\log n}{n^{1/4}(\Im \xi)^3}\right),
\eeq
with high probability. This will complete the proof of the lemma. 

To this end, applying Theorem \ref{thm:hamming_conc}, seting $M=0$ there, using the trivial bound $|\wt m_n(\xi)| \le 1/\Im \xi$ again, and the triangle inequality we obtain that
\beq\label{eq:prob-bound-1}
\P \left( |\wt m _n^3(\xi) - (\E \wt m_n(\xi))^3 | \ge \f{3 \log n  }{n^{1/2}(\Im \xi)^4}\right) \le \P \left( |\wt m_n(\xi) - \E \wt m_n(\xi) | \ge \f{ \log n}{n^{1/2} (\Im \xi)^2}\right) \le 4 e^{-c_{\ref{thm:hamming_conc}} (\log n)^2}
\eeq
and
\beq\label{eq:prob-bound-2}
\P \left( |\wt m _n^2(\xi) - (\E \wt m_n(\xi))^2 | \ge \f{2 \log n  }{n^{1/2}(\Im \xi)^3}\right) \le 4 e^{-c_{\ref{thm:hamming_conc}} (\log n)^2}.
\eeq
Since $ n^{-1/16} \le \Im \xi \le C_0$ we also have that 
\[
\f{\log n }{n^{1/2}(\Im \xi)^3} = O\left(\f{\log n }{n^{1/2}(\Im \xi)^4}\right) = O\left(\f{\log n}{n^{1/4}(\Im \xi)^3}\right),
\]
yielding \eqref{eq:loop-eqn-wt-exp-1}. The desired probability bounds \eqref{eq:prob-bound-1}-\eqref{eq:prob-bound-2}. The proof of the lemma now completes.
\end{proof}

Now it remains to prove Lemma \ref{lem:loop-eqn-exp}. As we will see below, to prove the same we will first derive an approximate fixed point equation involving $\E \wh m_n(\xi)$ and an auxiliary variable $\E \hat \nu_n(\xi)$ where
\[
\wh \nu_n(\xi):= \f{1}{n}\sum_{i=1}^n (\wh G(S_n^d))_{i,n+i}.
\]
Then an additional equation will be derived to eliminate $\E \nu_n(\xi)$ from the first equation. To obtain these two equations we will need to consider the expectation of the entries of product of matrices that are functions of centered permutation matrices. Hence, it will be useful to introduce the following notation. For ease of writing, for any permutation $\pi_n$ uniformly distributed on $\mathbb{S}_n$, we denote
\beq\label{eq:sfP}
{\sf P}:={\sf P}_n:= \begin{bmatrix} 0 & (P_n - \E P_n)\\ (P_n - \E P_n)^* & 0 \end{bmatrix}, \qquad \text{ where } \qquad P_n(i,j):=\bI(\pi_n(i)=j).
\eeq
Equipped with the above notation we have the following lemma.

\begin{lem}\label{lem:expectation}
Let $M:=M_n$ be a $2n \times 2n$ deterministic matrix. Then 
\begin{enumeratei}

\item 
\[
\E\left[ \left({\sf P} M {\sf P} M\right)_{i,i}\right] = M_{i,i}  \left[\f{1}{n}\sum_{j=1}^n M_{n+j,n+j}\right] + O\left(\f{\|M\|^2}{\sqrt{n}}\right)\corN{,}
\]
\item 
\[
\E\left[ \left({\sf P} M {\sf P} M\right)_{n+i,n+i}\right] = M_{n+i,n+i}  \left[\f{1}{n}\sum_{j=1}^n M_{j,j}\right] + O\left(\f{\|M\|^2}{\sqrt{n}}\right)\corN{,}
\]
\item 
\[
\E\left[ \left({\sf P} M {\sf P} M\right)_{n+i,i}\right] = M_{n+i,i}  \left[\f{1}{n}\sum_{j=1}^n M_{j,j}\right] + O\left(\f{\|M\|^2}{\sqrt{n}}\right)\corN{,}
\]

\item 
\[
\E\left[ \left({\sf P} M {\sf P} M\right)_{i,n+i}\right] = M_{i,n+i}  \left[\f{1}{n}\sum_{j=1}^n M_{n+j,n+j}\right] + O\left(\f{\|M\|^2}{\sqrt{n}}\right).
\]
\end{enumeratei}
\end{lem}

\begin{proof}
Recalling \eqref{eq:sfP}, we make the following observations:
\beq\label{eq:expP1}
\E[{\sf P}_{i,n+j} {\sf P}_{n+k,\ell}] = \left\{\begin{array}{ll}
\f{1}{n}(1-\f{1}{n}) &\mbox{if } i=\ell, j =k\\
-\f{1}{n^2} &\mbox{if } i \ne \ell, j=k, \mbox{or } i=\ell, j\ne k\\
\f{1}{n^2(n-1)} & \mbox{if } i \ne \ell, j \ne k
\end{array}
\right.
\eeq
and 
\beq\label{eq:expP2}
\E[{\sf P}_{i,n+j} {\sf P}_{k,n+\ell}] = \left\{\begin{array}{ll}
\f{1}{n}(1-\f{1}{n}) &\mbox{if } i=k, j =\ell\\
-\f{1}{n^2} &\mbox{if } i \ne k, j=\ell, \mbox{or } i=k, j\ne \ell \\
\f{1}{n^2(n-1)} & \mbox{if } i \ne k, j \ne \ell
\end{array}
\right.
.
\eeq
Since the diagonal blocks of ${\sf P}$ are zero it follows that
\begin{align}\label{eq:split-to-2}
\E [({\sf P} M {\sf P} M)_{i,i}] & = \E\left[ \sum_{j,k,\ell=1}^n {\sf P}_{i,n+j} M_{n+j, n+k} {\sf P}_{n+k, \ell} M_{\ell, i} \right] + \E\left[ \sum_{j,k,\ell=1}^n {\sf P}_{i,n+j} M_{n+j, k} {\sf P}_{k, n+ \ell} M_{n+\ell, i} \right] \notag\\
& =: {\rm Term \, I} + {\rm Term\, II}.
\end{align}
Using \eqref{eq:expP1} we have
\begin{align}
{\rm Term \, I} = \, &  M_{i,i} \left[\f{1}{n}\sum_{j=1}^n M_{n+j,n+j}\right] \cdot \left(1-\f{1}{n}\right) - \f{1}{n^2} \left[\sum_{\substack{j,\ell=1\\ \ell \ne i}}^n M_{n+j, n+j} M_{\ell, i} + M_{i,i} \sum_{\substack{j,k=1\\ j \ne k}}^n M_{n+j, n+k}\right] \notag\\
& + \f{1}{n^2(n-1)}\sum_{\substack{j,k,\ell=1\\ j \ne k, i \ne \ell}}^n M_{n+j,n+k} M_{\ell, i}. \notag
\end{align}
Since
\beq\label{eq:l-2-row}
\max_{\ell=1}^{2n} |M_{\ell,i}| \le \sum_{\ell=1}^{2n} |M_{\ell,i}| \le \sqrt{2n \sum_{\ell=1}^{2n} |M_{\ell,i}|^2} \le \sqrt{2n} \| M\|,  \quad \text{ for all } i \in [2n],
\eeq
we deduce from above that
\beq\label{eq:term1bd}
{\rm Term \, I} = M_{i,i} \left[\f{1}{n}\sum_{j=1}^n M_{n+j,n+j}\right] + O\left(\f{\|M\|^2}{\sqrt{n}}\right).
\eeq
Using \eqref{eq:expP2} and a similar argument as above we also deduce that
\begin{align}\label{eq:term2bd}
{\rm Term \, II} = \f{1}{n}\sum_{j=1}^n M_{n+j, i}^2  + O\left(\f{\|M\|^2}{\sqrt{n}}\right) = O\left(\f{\|M\|^2}{\sqrt{n}}\right),
\end{align}
where the last step follows from \eqref{eq:l-2-row}. Thus, the part (i) of the lemma now follows upon plugging the bounds \eqref{eq:term1bd}-\eqref{eq:term2bd} in \eqref{eq:split-to-2}. To prove (iii) we apply \eqref{eq:expP1}-\eqref{eq:expP2},  \eqref{eq:l-2-row}, and Cauchy--Schwarz inequality to deduce that 
\begin{align*}
\E\left[({\sf P} M{\sf P} M)_{n+i,i}\right] & = M_{n+i,i} \left[\f{1}{n}\sum_{j=1}^n M_{j,j}\right] + \f{1}{n}\sum_{j=1}^n M_{j,n+i} M_{n+j, i}  + O\left(\f{\|M\|^2}{\sqrt{n}}\right)\\
&= M_{i,n+i} \left[\f{1}{n}\sum_{j=1}^n M_{n+j,n+j}\right] +  O\left(\f{\|M\|^2}{\sqrt{n}}\right).
\end{align*}
This yields part (iii). The proofs of parts (ii) and (iv) follow from a similar argument as above and hence omitted. 
\end{proof}

We will apply Lemma by setting ${\sf P}={\sf P}_n^\ell$ for some $\ell \in [d]$ and $M$ will be functions of $\{{\sf P}_n^j\}$. 
For $\ell \in [d]$ denote
\[
\wh{G}^{(\ell)}(S_n^d) : =\left[ \xi -\begin{bmatrix} 0 & z \\ \bar{z} & 0 \end{bmatrix} + \frac1{\sqrt{d}}{\sf S}_n^{d,(\ell)} \right]^{-1},  
\]
$\text{ where } {\sf S}_n^{d,(\ell)}:= \sum_{j \ne \ell} {\sf P}_n^j$, and  
\[
\wh \nu_n^{(\ell)}(\xi):= \f{1}{n}\sum_{i=1}^n (\wh{G}^{(\ell)}(S_n^d))_{i,n+i}.
\]
Recall the following result regarding the inverse of a block matrix. 
\begin{lem}[Inverse of a block matrix]\label{lem:inverse_block_matrix}
\[
\begin{bmatrix}
A & B \\
C & D
\end{bmatrix}^{-1}= \begin{bmatrix} (A-B D^{-1}C)^{-1} & - A^{-1} B(D - C A^{-1} B)^{-1}\\
- D^{-1} C (A-B D^{-1} C)^{-1} & (D - C A^{-1} B)^{-1} \end{bmatrix}.
\]
\end{lem}
Applying Lemma \ref{lem:inverse_block_matrix} wth $A=D=\xi I$, where $\xi \in \C\setminus \R$, and $C=B^*$ we obtain that
\[
\f12\Tr \left(\begin{bmatrix}
\xi & B \\
B^* & \xi
\end{bmatrix}^{-1}\right) = \xi \Tr \left\{ (\xi^2 - B B^*)^{-1} \right\}= \xi \Tr \left\{ (\xi^2 - B^*B)^{-1} \right\}.
\]
This, in particular, implies that for every $\ell \in [d]$,
\beq\label{eq:half-trace}
\f{1}{n}\sum_{i=1}^n (\wh{G}^{(\ell)}(S_n^d))_{n+i,n+i}= \f{1}{n}\sum_{j=1}^n (\wh{G}^{(\ell)}(S_n^d))_{j,j} = \f{1}{2n}\Tr \wh{G}^{(\ell)}(S_n^d) =: \wh{m}_n^{(\ell)}(\xi).
\eeq

We now prove Lemma \ref{lem:loop-eqn-exp}.

\begin{proof}[Proof of Lemma \ref{lem:loop-eqn-exp}]
From the identity $\wh{G}(S_n^d)^{-1}\wh{G}(S_n^d)= I_{2n}$ we obtain
\beq\label{eq:reseq1}
\begin{bmatrix} \xi & -z \\ -\bar{z} & \xi \end{bmatrix} \wh{G}(S_n^d) = I_{2n} - \f{{\sf S}_n^d}{\sqrt{d}} \cdot  \wh{G}(S_n^d).
\eeq
Applying the resolvent identity (\eqref{eq:resolvent_2} in Lemma \ref{lem:res-id}) twice we further obtain that for any $\ell \in [d]$,
\[
\wh{G}(S_n^d)= \wh{G}^{(\ell)}(S_n^d) - \wh{G}^{(\ell)}(S_n^d) \cdot \f{{\sf P}_n^\ell}{\sqrt{d}} \cdot \wh{G}^{(\ell)}(S_n^d) + \wh{G}^{(\ell)}(S_n^d) \cdot \f{{\sf P}_n^\ell}{\sqrt{d}} \cdot \wh{G}^{(\ell)}(S_n^d)  \cdot \f{{\sf P}_n^\ell}{\sqrt{d}}  \cdot \wh{G}^(S_n^d).
\]
Fixing $\ell \in [d]$ we use the above identity to expand $\wh{G}(S_n^d)$, which we then plug in \eqref{eq:reseq1}. Therefore, now summing over $\ell \in [d]$, from \eqref{eq:reseq1} we deduce that
\begin{align}\label{eq:reseq2}
\begin{bmatrix} \xi & -z \\ -\bar{z} & \xi \end{bmatrix} \wh{G}(S_n^d) = \,  &  I_{2n}-  \sum_{\ell=1}^d \f{{\sf P}_n^\ell}{\sqrt{d}} \cdot  \wh{G}^{(\ell)}(S_n^d) + \sum_{\ell=1}^d \f{{\sf P}_n^\ell}{\sqrt{d}} \cdot \wh{G}^{(\ell)}(S_n^d) \cdot \f{{\sf P}_n^\ell}{\sqrt{d}} \cdot \wh{G}^{(\ell)}(S_n^d)\notag\\
& - \sum_{\ell=1}^d \f{{\sf P}_n^\ell}{\sqrt{d}} \cdot \wh{G}^{(\ell)}(S_n^d) \cdot \f{{\sf P}_n^\ell}{\sqrt{d}} \cdot \wh{G}^{(\ell)}(S_n^d) \cdot \f{{\sf P}_n^\ell}{\sqrt{d}} \cdot \wh{G}(S_n^d).
\end{align}
Next we need to simplify \eqref{eq:reseq2}. To this end, for every $\ell \in [d]$, let $\E_\ell(\cdot)$ denote the expectation with respect to the randomness of $P_n^\ell$ and $\E_{-\ell}(\cdot)$ denote the expectation with respect to the randomness of $\{P_n^j\}_{j \ne \ell}$. Since $\{P_n^\ell\}$ are independent we have that $P_n^\ell$ and $\wh{G}^{(\ell)}(S_n^d)$ are independent for every $\ell\in [d]$, which in particular implies that 
\beq\label{eq:exp=0}
\E[ \wh{G}^{(\ell)}(S_n^d) {\sf P}_n^\ell]=\E_{-\ell}\left[\E_\ell (\wh{G}^{(\ell)}(S_n^d) {\sf P}_n^\ell)\right]=0, \qquad \ell \in [d],
\eeq
where we have used the fact that the entries of ${\sf P}_n^\ell$ are centered. Applying Lemma \ref{lem:expectation} we also note that 
\begin{align}
& \f1n\sum_{i=1}^n \sum_{\ell=1}^d \E \left[\left\{\f{{\sf P}_n^\ell}{\sqrt{d}} \cdot \wh{G}^{(\ell)}(S_n^d) \cdot \f{{\sf P}_n^\ell}{\sqrt{d}} \cdot \wh{G}^{(\ell)}(S_n^d)\right\}_{n+i,n+i}\right]\notag\\
= &  \f1d \sum_{\ell=1}^d \E_{-\ell}\left[ \left\{\f{1}{n}\sum_{i=1}^n (\wh{G}^{(\ell)}(S_n^d))_{n+i,n+i}\right\} \cdot \left\{\f{1}{n}\sum_{j=1}^n (\wh{G}^{(\ell)}(S_n^d))_{j,j}\right\}\right] \notag\\ = &   \f{1}{d} \sum_{\ell=1}^d \E_{-\ell} \left[\wh{m}_n^{(\ell)} (\xi)^2\right]+ O(n^{-1/2} (\Im \xi)^{-2}),
\end{align}
where the last step follows from \eqref{eq:half-trace} and the standard operator norm bound $\|\wh{G}^{(\ell)}(S_n^d)\| \le 1/ \Im \xi$. Therefore, considering the $(n+i,n+i)$-th entry of the both sides of \eqref{eq:reseq2}, taking an average over $i \in [n]$, followed by taking an expectation over the randomness  of $\{P_n^\ell\}$, upon using \eqref{eq:exp=0}, we obtain
\begin{align}\label{eq:reseq3}
- \bar z \E \widehat\nu_n(\xi) + \xi \E \widehat{m}_n(\xi) 
& = 1+ \f{1}{d} \sum_{\ell=1}^d \E_{-\ell} \left[\wh{m}_n^{(\ell)}(\xi)^2\right] - {\rm Term \, E_1} + O(n^{-1/2} (\Im \xi)^{-2}),
\end{align}
where
\[
{\rm Term \, E_1} := \sum_{\ell=1}^d \E \left[\left(\f{{\sf P}_n^\ell}{\sqrt{d}} \cdot \wh{G}^{(\ell)}(S_n^d) \cdot \f{{\sf P}_n^\ell}{\sqrt{d}} \cdot \wh{G}^{(\ell)}(S_n^d) \cdot \f{{\sf P}_n^\ell}{\sqrt{d}} \cdot \wh{G}(S_n^d)\right)_{n+i,n+i}\right].
\]
Using the resolvent identity once again we observe that for any $\ell \in [d]$,
\beq\label{eq:remove-ell}
\|\wh{G}^{(\ell)}(S_n^d) - \wh{G}(S_n^d) \| \le \f{1}{\sqrt{d}}\| \wh{G}(S_n^d)\| \cdot \|P_n^\ell- \E P_n^\ell\| \cdot \| \wh{G}^{(\ell)}(S_n^d)\| \le 2 d^{-1/2} (\Im \xi)^{-2},
\eeq
where the last inequality follows from the facts that $\|\wh{G}(S_n^d)\|, \|\wh{G}^{(\ell)}(S_n^d)\| \le (\Im \xi)^{-1}$ and $\| P_n^\ell - \E P_n^\ell\| \le 2$. Thus ${\rm Term \, E_1}= O(d^{-1/2} (\Im \xi)^{-3})$, which in particular implies that the first term in the \abbr{RHS} of \eqref{eq:reseq3} is the dominant term. 
Using \eqref{eq:remove-ell} we also note that  $|\wh{m}_n(\xi)^2 - \wh{m}^{(\ell)}(\xi)^2| \le 4 d^{-1/2} (\Im \xi)^{-3}$. Hence, from \eqref{eq:reseq3}, upon using the facts that $d=O(n)$ and $\Im \xi \le C_0$, we deduce
\begin{align}\label{eq:eq1}
-\bar{z} \E \widehat\nu_n(\xi) + \xi \E \widehat{m}_n(\xi) & =1+\f{1}{d} \sum_{\ell=1}^d \E_{-\ell} \left[\wh{m}_n^{(\ell)}(\xi)^2\right] + O(d^{-1/2} (\Im \xi)^{-3})\\ & = 1+ \E \left[ \wh{m}_n(\xi)^2\right] + O(d^{-1/2} (\Im \xi)^{-3}) = 1+\left(\E \wh{m}_n(\xi)\right)^2 + O(d^{-1/2} (\Im \xi)^{-3}), \notag
\end{align}
where the last step follows from Corollary \ref{cor-magic8} upon taking 
\[
M= \begin{bmatrix} 0 & \f{\E S_n^d}{\sqrt{d}} \\ \f{\E (S_n^d)^*}{\sqrt{d}} & 0 \end{bmatrix}  = \begin{bmatrix} 0 & \f{\sqrt{d}}{n} {\bm 1} {\bm 1}^{\sf T} \\ \f{\sqrt{d}}{n} {\bm 1} {\bm 1}^{\sf T} & 0 \end{bmatrix}
\]
(recall that ${\bm 1}$ is the $n$-dimensional vector consisting of all ones) and observing that $d^{1/2} (\Im \xi)^3 = O(n^{1/2})= o(n^{3/4})$. 

Note that \eqref{eq:eq1} involves $\E\wh{\nu}_n(\xi)$.  To derive the desired approximate fixed point equation for $\E \wh{m}_n(\xi)$ one needs eliminate $\E\wh{\nu}_n(\xi)$ from \eqref{eq:eq1}. To this end, consider  the $(i,n+i)$-th entry of the both sides of \eqref{eq:reseq2}, take an average over $i \in [n]$, and proceed similarly as in the steps leading to \eqref{eq:reseq3} to deduce that
\beq\label{eq:eq2.1}
\xi \E \wh\nu_n(\xi) - z \E \wh{m}_n(\xi) =  \f1d\sum_{\ell=1}^d\E_{-\ell} \left[\wh\nu_n^{(\ell)}(\xi) \wh m_n^{(\ell)}(\xi)\right]+ O(n^{-1/2} (\Im \xi)^{-2}) -{\rm Term \, E_2},
\eeq
where 
\[
{\rm Term \, E_2}:=  \sum_{\ell=1}^d \E \left[\left(\f{{\sf P}_n^\ell}{\sqrt{d}} \cdot \wh{G}^{(\ell)}(S_n^d) \cdot \f{{\sf P}_n^\ell}{\sqrt{d}} \cdot \wh{G}^{(\ell)}(S_n^d) \cdot \f{{\sf P}_n^\ell}{\sqrt{d}} \cdot \wh{G}(S_n^d)\right)_{i,n+i}\right]= O(d^{-1/2}(\Im \xi)^{-3})
\]
and the last step follows from the operator norm bounds $\|\wh{G}(S_n^d)\|, \|\wh{G}^{(\ell)}(S_n^d)\| \le (\Im \xi)^{-1}$ and $\| P_n^\ell - \E P_n^\ell\| \le 2$. Using \eqref{eq:remove-ell} and the resolvent identity we also have that
\[
|\wh\nu_n^{(\ell)}(\xi) \wh m_n^{(\ell)}(\xi) - \wh\nu_n(\xi) \wh m_n(\xi) | = O(d^{-1/2} (\Im \xi)^{-3}). 
\]
On the other hand an application of Corollary \ref{cor-magic8} and Cauchy-Schwarz inequality yield that
\[
\left| \E (\wh m_n(\xi) \wh \nu_n(\xi) ) - \E \wh m_n(\xi) \E \wh \nu_n(\xi) \right| \le n^{-3/4}.
\]
Therefore, the approximate equation \eqref{eq:eq2.1} simplifies to 
\beq\label{eq:eq2}
(\E \wh m_n(\xi) -\xi) \E \wh\nu_n(\xi) = -z \E \wh{m}_n(\xi)  + O(d^{-1/2} (\Im \xi)^{-3}).
\eeq
Finally multiplying both sides of \eqref{eq:eq1} by $(\E \wh m_n(\xi) -\xi)$, using  \eqref{eq:eq2}, and recalling that $\Im \xi \le C_0$, $|z| \le R$, we arrive at \eqref{eq:loop-eqn}. This completes the proof of the lemma.
\end{proof}

\subsection{Proof of Theorem \ref{thm:smallish_sing_control}}\label{sec:prf-thm-26} In Section \ref{sec:sub-loop-eqn} we have shown that for $\xi \in \C^+$ with $(\log n)^{-2} \le \Im \xi \le C_0$ we have $\wt P (\wt m_n(\xi)) =o(1)$, with high probability, where
\[
\wt P(m):=\wt P (m, \xi):= \wt P (m, \xi,z)= m(m-\xi)^2 +(1-|z|^2) m - \xi.
\]
Since $\wt P(\cdot)$ it is not evident from Lemma \ref{lem:loop-eqn} that $|\wt m_n(\xi) - \wt m_\infty(\xi)| =o(1)$, where $\wt m_\infty(\xi)$ is the Stieltjes transform of the desired limit. Hence, it requires some additional properties of the roots of the equation $\wt P(m,\xi)=0$. 

During the proof of Theorem \ref{thm:main} we will see that it is enough to show that $|\wt m_n(\xi) - \wt m_\infty(\xi)| =o(1)$ holds for $\xi$ purely imaginary, that is $\xi= {\rm i} \eta$ for some $\eta >0$.  On the other hand, for any symmetric probability measure $\mu$ on $\R$ (i.e.~$\mu((a,b)) = \mu(((-b,-a))$ for any $0 < a< b \le \infty$), denoting $m_\mu(\xi)$ to be its Stieltjes transform, we have
\[
m_\mu({\rm i} \eta) = \int \f{1}{{\rm i} \eta - x} d\mu(x) = -\int \f{x+{\rm i} \eta}{x^2+\eta^2} d\mu(x) = -{\rm i} \eta \int \f{1}{x^2+\eta^2} d\mu(x).
\]
This means that $m_{\mu}({\rm i} \eta) = -{\rm i} x$ for some $x >0$. Therefore
\beq\label{eq:wt-P-on-pure-imag}
{\wt P} (m_\mu({\rm i} \eta), {\rm i} \eta) = {\rm i} x(x+\eta)^2-{\rm i} (1-|z|^2) x - {\rm i} \eta.
\eeq
Thus for any symmetric probability measure $\mu$ on $\R$, the map $\eta \mapsto \wt P(m_\mu({\rm i} \eta, {\rm i} \eta))$ is essentially a cubic polynomial over the reals. Since $\wt m_n(\xi)$ and $\wt m_\infty(\xi)$ are both Stieltjes transforms of symmetric probability measures and we need to control their differences only when $\xi$ is purely imaginary, it is enough to derive properties of the roots of the equation
\[
Q(x):=Q(x,\eta):=Q(x,\delta,\eta):= x(x+\eta)^2 - \delta x -\eta.
\]
where $\delta , \eta >0$.

\begin{lem}[Stability of the fixed point equation]\label{lem:stability}
Fix any $\delta, \eta >0$. Then the following properties hold for the cubic equation $Q(x)=0$.

\begin{enumerate}

\item[(i)] There exists a unique positive root $x_\star$ of the equation $Q(x)=0$. 

\item[(ii)] For any $c_0 >0$, 
\[
\inf_{x \ge c_0} \f{|Q(x)|}{|x-x_\star|} \ge c_0^2.
\]
\end{enumerate}
\end{lem}

\begin{proof}
Since $Q(0)= -\eta <0$ and $\lim_{x \to \infty} Q(x) =\infty$, the number of roots of the equation $Q(x)=0$ in the interval $(0,\infty)$ is either one or three. If the number of positive roots of the equation $Q(x)=0$ is three, then the Rolle's theorem implies that there exists $x_0 \in (0,\infty)$ such that $Q''(x_0)=0$ which is clearly a contradiction, as we note that $Q''(x)= 3x^2 + 4 \eta >0$ for all $x \in \R$. Thus there exists a unique $x_\star \in (0,\infty)$ such that $Q(x_\star)=0$. Turning to prove the second part of the lemma we note that
\begin{align}
Q(x)  = Q(x)- Q(x_\star) & = (x- x_\star) ((x+\eta)^2+x_\star(x+x_\star+2\eta)-\delta)\notag\\
& =  (x- x_\star) (x^2+2x \eta + x x_\star+ x_\star^2+ \eta^2+2x_\star\eta-\delta) \notag\\
& = (x- x_\star) \left(x^2+2x \eta + x x_\star+\f\eta x_\star\right)
\end{align}
where the last equality follows from the fact that $Q(x_\star)=0$. Since $x, x_\star, \eta >0$, we have that
\[
x^2+2x \eta + x x_\star+\f\eta x_\star \ge c_0^2,
\]
for all $x \ge c_0$. This completes the proof of the lemma. 
\end{proof}

Recalling \eqref{eq:wt-P-on-pure-imag} we see that for any symmetric probability measure $\mu$, $\wt P(m_\mu({\rm i} \eta, {\rm i} \eta)) = {\rm i} Q(x,\eta)$ where $m_\mu({\rm i}\eta)=-{\rm i} x$. Therefore, Lemma \ref{lem:stability}(i) implies that there is a unique symmetric probability measure $\wt{\mu}_\infty$ such that its Stieltjes transform $\wt{m}_\infty(\xi)$ satisfies the fixed point equation $\wt P (m)=0$. The second part of Lemma \ref{lem:stability} ensures that 
\beq\label{eq:wt-P-to-diff1}
|\wt P(\wt m_n({\rm i} \eta))|  \ge |\wt m_n({\rm i} \eta)|^2 |\wt m_n({\rm i} \eta) - \wt m_\infty({\rm i} \eta)|
\eeq
for all $\eta >0$ and in particular 
\beq\label{eq:wt-P-to-diff2}
|\wt P(\wt m_n({\rm i} \eta))|  \ge c_0^2 |\wt m_n({\rm i} \eta) - \wt m_\infty({\rm i} \eta)|,
\eeq
provided $|\wt m_n({\rm i} \eta)| \ge c_0$. The inequalities \eqref{eq:wt-P-to-diff1}-\eqref{eq:wt-P-to-diff2} will be crucially used to derive a bound on the difference between $\wt{m}_n(\xi)$ and $\wt m_\infty(\xi)$ from a bound on $|\wt P(\wt m_n(\xi))|$. However, these inequalities need apriori lower bound on $|\wt m_n(\xi)|$. Hence, to initiate the bootstrap argument we need to show that such lower bounds hold when $\xi$ far away from the real line. 

\begin{lem}[Preliminary lower bound]\label{lem:prelim_estimate}
Fix $z \in \C$ such that $|z| \le 1$. For any positive constant $C$ denote
\[
\mathbb{H}_C:= \{\xi \in \C: \Im \xi \ge C, \xi \in B_\C(0,2C)\}.
\]
Then there exist a sufficiently large absolute constant $\bar{C}_{\ref{lem:prelim_estimate}}$ 
and a sufficiently small absolute constant
$c'_{\ref{lem:prelim_estimate}}$
such that the following holds: There exists a set
  $\Omega_{\ref{lem:prelim_estimate},n}$ such that for every $\xi \in \mathbb{H}_ {\bar{C}_{\ref{lem:prelim_estimate}}}$, we have \(|\wt m_n(\xi)| \ge 1/10\)
on the event
  $\Omega_{\ref{lem:prelim_estimate},n}$, where
\[
\P(\Omega_{\ref{lem:prelim_estimate},n}^c) \le \exp(-c'_{\ref{lem:prelim_estimate}}d).
\]
\end{lem}


\begin{proof}
We set
\[
\Omega_{\ref{lem:prelim_estimate},n}^c:=\left\{\f{1}{n}\Tr (z- \hat{S}_n^d)^*(z-\hat{S}_n^d)>C\right\},
\]
where $C$ is chosen to be sufficiently large \corAa{and for brevity we write $\hat{S}_n^d:=S_n^d/\sqrt{d}$}. Recalling that $d=O(n)$ and $|z| \le 1$
it follows from \eqref{eq-perm3} that for $C$ large,
\[
\P\left(\f{1}{n}\Tr (z- \hat{S}_n^d)^*(z-\hat{S}_n^d)>C\right) \le d \exp(-c' d),
\]
for some absolute constant $c'$
establishing the desired assertion on the probability bound of $\Omega_{\ref{lem:prelim_estimate},n}^c$.

Now note that
\[
-\Im \wt{m}_n(\xi)= \f{1}{2n} \sum_{i=1}^{2n} \f{\Im \xi}{|\xi - \lambda_i|^2}
\geq \f{\Im \xi}{\f{1}{2n}\sum_{i=1}^{2n} |\xi-\lambda_i|^2}
\geq \f{\Im \xi}{2|\xi|^2+\f1n \sum_{i=1}^{2n} \lambda_i^2},
\]
where $\{\lambda_i\}_{i=1}^{2n}$ are the eigenvalues of
$$
\begin{bmatrix} 0 & (z- \hat{S}_n^d) \\
(z- \hat{S}_n^d)^* & 0 \end{bmatrix}.
$$
Thus,
\[
-\Im \wt{m}_n(\xi)
\geq \f{\bar C_{\ref{lem:prelim_estimate}}}{8\bar C_{\ref{lem:prelim_estimate}}^2+\f2n
\corAa{\Tr (z-\hat{S}_n^d)^* (z-\hat{S}_n^d)}}. \]
The desired lower bound on $\wt{m}_n(\xi)$ on the event $\Omega_{\ref{lem:prelim_estimate},n}$ now follows upon setting ${\bar C_{\ref{lem:prelim_estimate}}} =C$.
\end{proof}

When $\Im \xi$ is close to zero we cannot use Lemma \ref{lem:prelim_estimate}. In that case, the desired bound $|\wt m_n(\xi)|$ can be obtained by showing that it is close to $\wt m_\infty(\xi)$ and then obtaining bounds on $|\wt m_\infty(\xi)|$ which we derive in the lemma below. 

\begin{lem}[Properties of $\wt{m}_\infty$]\label{lem:prop_wtm_infty}
Fix any $\vep >0$ and let $z \in B_\C(0,1-\vep)$. Fix any $\xi \in \C^+$ such that $|\xi| \le \vep^{-1}$. Then there exist $\vep_0>0$ such that for any $\vep <\vep_0$ there exists constants $c_{\ref{lem:prop_wtm_infty}}$ and $ C_{\ref{lem:prop_wtm_infty}}$, depending only on $\vep$, such that 
\(
c_{\ref{lem:prop_wtm_infty}} \le |\wt{m}_\infty(\xi)| \le C_{\ref{lem:prop_wtm_infty}}
\).
\end{lem}

\begin{proof}
The proof of this lemma follows from \cite[Lemma 4.3]{bourgade2014local}. There they analyzed properties of the solution $m_c(\xi)$ of the cubic equation
\[
m(1+m)^2 \xi +(1-|z|^2)m + 1 =0,
\]
which has nonnegative imaginary part for all $\xi \in \C$.
In \cite{bai1997circular} it was shown that for any $\xi \in \C^+$, $-m_c(\xi)$ is the Stieltjes transform of the limiting distribution of the empirical measure of the singular values of $z- A_n/\sqrt{n}$ where $A_n$ is an $n \times n$ matrix of i.i.d.~entries with certain moment assumptions on its entries. Note that the limiting measure is the same in our set-up. Therefore $m_\infty(\xi)=-m_c(\xi)$ on $\C^+$. Since $\wt{m}_\infty(\xi)=\xi {m}_\infty(\xi^2)$, we use the relation between $m_\infty(\xi)$and $ m_c(\xi)$ to extract the properties of $\wt{m}_\infty(\xi)$.

From \cite[Eqn.~(4.9)]{bourgade2014local} we note that
\beq\label{eq:m_infty_bd}
c|\xi|^{-1} \le |m_\infty(\xi^2)|  \le C|\xi|^{-1},
\eeq
whenever $\Im(\xi^2) >0$, for some constants \corAa{$c$ and $C$} depending only on $\vep$. When $\Im (\xi^2) <0$ then we note that $m_\infty(\xi^2)=\ol{m_\infty(\bar{\xi}^2)}=\ol{m_c(\bar{\xi}^2)}$ and therefore \eqref{eq:m_infty_bd} also holds for all $\xi$ such that $\Im(\xi^2) <0$. Multiplying both sides of \eqref{eq:m_infty_bd} by $|\xi|$ and using the relation between $\wt{m}_\infty(\cdot)$ and $m_\infty(\cdot)$ we establish the desired conclusion for $\wt{m}_\infty(\cdot)$ for all $\xi$ such that $\Re \xi \ne 0$. We extend our conclusion for all $\xi$ such that $\Re \xi=0$ by continuity of $\wt{m}_\infty(\cdot)$ on $\C^+$.  
\end{proof}

Equipped with all ingredients we are now ready to prove Theorem \ref{thm:smallish_sing_control}.

\begin{proof}[Proof of Theorem \ref{thm:smallish_sing_control}]
Recall that
\[
\sS_{\vep,\varpi}:= \left\{ \xi ={\rm i}\eta: 0 < \eta \le \bar{C}_{\ref{thm:smallish_sing_control}},  \min\{{d}^{1/2}, n^{1/4} (\log n)^{-1}\} \eta^3 \ge \varpi_n\right\},
\]
where we set $\bar{C}_{\ref{thm:smallish_sing_control}}=2\bar{C}_{\ref{lem:prelim_estimate}}$. 
We need to show that $\wt{m}_n(\xi)$ is close to $\wt{m}_\infty(\xi)$ uniformly for all $\xi \in \sS_{\vep,\varpi}$. 
 Consider a decreasing sequence of positive reals $\{\eta_i\}_{i =0}^N$ such that $\eta_0 = \bar{C}_{\ref{thm:smallish_sing_control}}$, $ 1/(2n) < \eta_i - \eta_{i+1} <  1/n$ and $\eta_N \in \sS_{\vep,\varpi}$. Note that $N= O(n)$. 
 Denote
 \beq\label{eq:upsilon}
 \Upsilon_n(\xi):= 3 C_{\ref{lem:loop-eqn}} \max\{d^{-1/2}, n^{-1/4}\log n\} (\Im \xi)^{-3} (1+ 4 C_{\ref{lem:prop_wtm_infty}})
 \eeq
 and set
 \beq\label{eq:c_0}
 c_0= \f14 \min\{c_{\ref{lem:prop_wtm_infty}}, 1/10\}.
 \eeq
 Note that $\Upsilon_n(\xi) =o(1)$ for all $\xi \in \sS_{\vep, \varpi}$. Now applying Lemma {\ref{lem:loop-eqn}} we see that on the event $\Omega_n(\xi_0)$ we have 
 \[
 \wt{P}(\wt m_n(\xi_0)) = O\left( \max\{d^{-1/2}, n^{-1/4}\log n\} \right),
 \]
 as $|\E \wt m_n(\xi_0)| \le 1/\eta_0 < 1$. From \eqref{eq:wt-P-to-diff1} we have that
 \[
 |\wt P (\wt m_n(\xi_0)| \ge |\wt m_n(\xi_0)|^2 |\wt m_n (\xi_0) - \wt m_\infty(\xi_0)|.
 \]
This together with Lemma \ref{lem:prelim_estimate} further implies that on the event  $\Omega_{\ref{lem:prelim_estimate},n}\cap \Omega_n(\xi_0)$ we have
\[
|\wt m_n (\xi_0) - \wt m_\infty(\xi_0)| = O\left( \max\{d^{-1/2}, n^{-1/4}\log n\} \right) =o (1).
\]
Therefore, Lemma \ref{lem:prop_wtm_infty} and the triangle inequality yields
\beq\label{eq:wt-m_n-lbd}
2 c_0 \le |\wt m _n(\xi_0)| \le 2 C_{\ref{lem:prop_wtm_infty}}
\eeq
on the event $\Omega_{\ref{lem:prelim_estimate},n}\cap \Omega_n(\xi_0)$, for all large $n$. Note that we also have
\beq\label{eq:e-wt-m}
|\E [ \wt m_n(\xi_0)]| \le 2 C_{\ref{lem:prop_wtm_infty}} + (\Im \xi_0)^{-1} \P (\Omega_{\ref{lem:prelim_estimate},n}^c\cup \Omega_n(\xi_0)^c) \le 3 C_{\ref{lem:prop_wtm_infty}},
\eeq
for all large $n$, where we use the fact that $\Im \xi_0 > \Im \xi_N \ge (\log n)^{-2}$.

Now we are ready to carry out the bootstrap argument. Indeed, applying Lemma {\ref{lem:loop-eqn}} again and using the inequality $|\wt m_n(\xi) - \wt m_n(\xi')| \le |\xi -\xi'|/((\Im \xi)\cdot (\Im \xi'))$ we deduce that 
\begin{align}\label{eq:boot-1}
|\wt P(\wt m_n(\xi))|  & \le |\wt P(\wt m_n(\xi_0))|+  |\wt P(\wt m_n(\xi))- \wt P(\wt m_n(\xi_0))| \notag\\
& \le C_{\ref{lem:loop-eqn}} \max\{d^{-1/2}, n^{-1/2}(\log n)^3\} (\Im \xi_0)^{-3} (1+ |\E \wt m_n(\xi_0)|) +O\left(\f{(\log n)^8}{n}\right) \notag\\
&  \le 2 C_{\ref{lem:loop-eqn}} \max\{d^{-1/2}, n^{-1/2}(\log n)^3\} (\Im \xi_0)^{-3} (1+ |\E \wt m_n(\xi_0)|)\notag\\
& \le 2 C_{\ref{lem:loop-eqn}} \max\{d^{-1/2}, n^{-1/2}(\log n)^3\} (\Im \xi)^{-3} \left(1+O\left(\f{(\log n)^8}{n}\right)\right) (1+ |\E \wt m_n(\xi_0)|)\notag\\
& \le  \Upsilon_n(\xi),
\end{align}
for all $\xi ={\rm i}\eta$ with $\eta \in [\eta_1,\eta_0]$, on the event $\Omega_{\ref{lem:prelim_estimate},n}\cap \Omega_n(\xi_0)$, where in the last step we have used \eqref{eq:e-wt-m}. On other hand, from \eqref{eq:wt-m_n-lbd} and the inequality $|\wt m_n(\xi) - \wt m_n(\xi')| \le |\xi -\xi'|/((\Im \xi)\cdot (\Im \xi'))$ we obtain $|\wt{m}_n(\xi)| \ge c_0$ for all $\xi ={\rm i}\eta$ with $\eta \in [\eta_1,\eta_0]$, on the event $\Omega_{\ref{lem:prelim_estimate},n}\cap \Omega_n(\xi_0)$. This together with \eqref{eq:wt-P-to-diff2} implies that
\beq\label{eq:induction-wt-m-n}
|\wt m_n(\xi)-\wt m_\infty(\xi)| \le c_0^{-2} \Upsilon_n(\xi)
\eeq
for all $\xi ={\rm i}\eta$ with $\eta \in [\eta_1,\eta_0]$, on the event $\Omega_{\ref{lem:prelim_estimate},n}\cap \Omega_n(\xi_0)$.

We complete the  proof by induction. Indeed, we denote $\Omega_j:=\cap_{i=0}^{j-1} \Omega_n(\xi_i) \cap \Omega_{\ref{lem:prelim_estimate},n}$. By the induction hypothesis we assume that \eqref{eq:induction-wt-m-n} holds for all $\xi={\rm i} \eta$ with $\eta \in [\eta_k, \eta_0]$ on the event $\Omega_k$. To finish the proof we need to show that \eqref{eq:induction-wt-m-n} continue to hold for all $\xi={\rm i} \eta$ with $\eta \in [\eta_{k+1}, \eta_k]$ on the event $\Omega_{k+1}$. 

First we note that using Lemma \ref{lem:prop_wtm_infty} and proceeding similarly as in \eqref{eq:e-wt-m} we obtain $|\E \wt m_n(\xi_k)| \le 3 C_{\ref{lem:prop_wtm_infty}}$. Therefore, arguing similarly as in \eqref{eq:boot-1} we deduce that the conclusion of \eqref{eq:boot-1} continue to hold for all $\xi={\rm i} \eta$ with $\eta \in [\eta_{k+1}, \eta_k]$ on the event $\Omega_{k+1}$. Using this we also get that \eqref{eq:induction-wt-m-n} holds for all $\xi={\rm i} \eta$ with $\eta \in [\eta_{k+1}, \eta_k]$ on the event $\Omega_{k+1}$. 

Thus by induction we have shown that for all $\xi ={\rm i} \eta$ with $\eta \in [\eta_N, \eta_0]$ the inequality \eqref{eq:induction-wt-m-n} holds on the event $\Omega_N$. Since, \[\P(\Omega_N^c) \le \P(\Omega_{\ref{lem:prelim_estimate},n}^c)+ \sum_{j=0}^{N-1} \P (\Omega_n(\xi_j)^c), \]
the proof of the theorem now completes from the probability bounds obtained in Lemma {\ref{lem:loop-eqn}} and Lemma \ref{lem:prelim_estimate}, and using the fact $N=O(n)$. This finishes the proof.
\end{proof}

\section{Proof of Theorem \ref{thm:main}}
\label{sec:proof_thmmain}

Recall that the key step in Girko's method is the integrability of $\log(\cdot)$ with respect to the empirical distribution of the singular values of ${S_n^d}/{\sqrt{d}} - z I_n $ for Lebesgue almost every $z \in \C$. From Theorem \ref{thm:ssv} we have quantitive bounds on the smallest singular value of ${S_n^d}/{\sqrt{d}} - z I_n$.
The conclusion of Theorem \ref{thm:smallish_sing_control} will show that there are not too many singular values in small intervals near zero. However, we note that Theorem \ref{thm:smallish_sing_control} holds only for 
$z \in B_\C(0,1-\vep)$, where $\vep>0$ is arbitrary but fixed. So the steps of Girko's method, as stated  in Section \ref{sec:prelim_outline}, cannot be carried out.
\corO{To overcome this difficulty we use the
{\em replacement principle}, already present e.g. in
the works of Tao and Vu, see in particular
\cite[Theorem 2.1]{tao_vu_2010}. However,
their proof  requires
control on the small singular values for Lebesgue almost every $z \in \C$.
Below we adapt their proof to obtain a version
of the replacement principle, which is suited to our purpose.
Before stating the result we introduce more definitions.}

If $\{X_n\}$ is a sequence of random variables, we say that $X_n$ is bounded in probability if we have
\[
\lim_{K \ra \infty} \liminf_{n\ra \infty} \P (|X_n| \le K)=1.
\]
Next for a matrix $B_n$, we denote $\norm{B_n}_2$ to be its Frobenius norm, i.e. $\norm{B_n}_2:= \sqrt{\Tr(B_n^* B_n)}$.
Now we are ready to state the result on replacement principle.

\begin{lem}[Replacement lemma]
Let $B_n^{(1)}$ and $B_n^{(2)}$ are two sequences of $n \times n $ random matrices, such that

\noindent
(i) The expression
\beq
\f{1}{n}\norm{B_n^{(1)}}_2^2 + \f{1}{n}\norm{B_n^{(2)}}_2^2, \noindent
\eeq
is bounded in probability

\noindent
and

\noindent
(ii) For Lebesgue almost all $z \in \D \subset B_\C(0,R) \subset \C$, for some domain $\D$ and some $R$ finite,
\beq
\f{1}{n} \log |\det (B_n^{(1)}-z I_n)| - \f{1}{n} \log |\det (B_n^{(2)}-z I_n)|  \ra 0, \notag
\eeq
in probability.

Then for every $f \in C_c^2(\C)$ supported on $\D$,
\beq
\int f(z)dL_{B_n^{(1)}}(z) - \int f(z)dL_{B_n^{(2)}}(z) \ra 0, \notag
\eeq
in probability.
\label{lem:replacement}
\end{lem}

\bigskip

Since Theorem \ref{thm:smallish_sing_control} holds for all 
$z \in B_\C(0,1-\vep)$, for every $\vep >0$, we can set $\D_\vep:=B_\C(0,1-\vep)$ and apply Lemma \ref{lem:replacement} to conclude that $\int f d L_{S_n^d/\sqrt{d}} \ra \f{1}{2\pi}\int f d \gm$ for all smooth functions $f$ supported on $\D_\vep$, where we recall $\gm(\cdot)$ is the Lebesgue measure on $\C$. Since $\vep>0$ is arbitrary and the circular law is supported on $B_\C(0,1)$, the above is enough to conclude the weak convergence of $L_{S_n^d/\sqrt{d}}$ (for more details see the proof of Theorem \ref{thm:main}).

We now turn our attention to the proof of Lemma \ref{lem:replacement}. A key tool is the following dominated convergence theorem.
\begin{lem}\label{prop:dom_conv}{\em {\bf (\cite[Lemma 3.1]{tao_vu_2010})}}
Let $(\cX,\mu)$ be a finite measure space. For each integer $n \ge 1$, let $f_n: \cX \ra \R$ be random functions which are jointly measurable with respect to $\cX$ and the underlying probability space. Assume that

\noindent
(i) There exists $\delta >0$ such that $\int_\cX |f_n(x)|^{1+\delta} d\mu(x)$ is bounded in probability.

\noindent
(ii) For $\mu$-almost every $x\in \cX$, $f_n(x)$ converges to zero in probability.

\noindent
Then $\int_\cX f_n(x)d\mu(x)$ converges to zero in probability.
\end{lem}

\vskip10pt

With the help of Lemma \ref{prop:dom_conv}, one can check that the proof of Lemma \ref{lem:replacement} actually follows from an easy adaptation of the alternative proof of \cite[Theorem 2.1]{tao_vu_2010} \corN{sketched in \cite[Section 3.6]{tao_vu_2010}}. We provide a short proof for
completeness.

\begin{proof}[Proof of Lemma \ref{lem:replacement}]
From \eqref{eq:greens_theorem}, it follows that for any $f \in C_c^2(\C)$,
\begin{align}
&\int f(z)dL_{B_n^{(1)}}(z) - \int f(z)dL_{B_n^{(2)}}(z)\label{eq:ESD&laplacian}\\
=& \f{1}{2 \pi n} \int \Delta f(z) \Big( \log |\det (B_n^{(1)}-z I_n)| -  \log |\det (B_n^{(2)}-z I_n)| \Big) d\gm(z).\notag
\end{align}
Set $\cX:= \D$,
\beq
f_n(z):=\f{1}{2 \pi n} \Delta f(z) \Big( \log |\det (B_n^{(1)}-z I_n)| -  \log |\det (B_n^{(2)}-z I_n)| \Big), \notag
\eeq
and $\mu$ to be the Lebesgue measure on $\D$ in Lemma \ref{prop:dom_conv}. We see that with these choices the assumption (ii) of Lemma \ref{prop:dom_conv} is satisfied. To prove assumption (i) of Lemma \ref{prop:dom_conv} note that, for any $\lambda \in \C$,
\begin{align}
&\int |\Delta f(z)|^2 \Big(\log |\lambda - z|\Big)^2 d\gm(z)\notag\\
 \le &\int_{z \in B_\C(\lambda,1)} |\Delta f(z)|^2 \Big(\log |\lambda - z|\Big)^2 d\gm(z)+ \int_{z \notin B_\C(\lambda,1)} |\Delta f(z)|^2 2(|\lambda|^2+|z|^2) d\gm(z) \le C(1+ |\lambda|^2), \notag
\end{align}
for some positive finite constant $C$ depending on $f$. Here the last step follows from the fact that $f \in C_c^2(\C)$. Therefore using Cauchy--Schwarz inequality, denoting $\lambda_i^{B_n^{(j)}}, i=1,2,\ldots,n$, to be the eigenvalues of $B_n^{(j)}$, for $j=1,2$, we have that
\beq
\int_{\cX}|f_n(z)|^2 d\gm(z) \le C'\bigg(1 + \f{1}{n}\sum_{i=1}^n \big|\lambda_i^{B_n^{(1)}}\big|^2+ \f{1}{n}\sum_{i=1}^n \big|\lambda_i^{B_n^{(2)}}\big|^2\bigg), \notag
\eeq
for some another positive finite constant $C'$. Finally, using assumption (i) of Lemma \ref{lem:replacement}, and Weyl's comparison inequality for second moment (cf.~\cite[Lemma A.2]{tao_vu_2010}), we see that the assumption (i) of Lemma \ref{prop:dom_conv} is satisfied. Thus, recalling \eqref{eq:ESD&laplacian}, the proof now completes upon applying Lemma \ref{prop:dom_conv}.
\end{proof}

\vskip10pt
Now we are almost ready to complete the proof of Theorem \ref{thm:main}. Recall that we earlier mentioned that the control on the Stieltjes transform derived in Theorem \ref{thm:smallish_sing_control} provides us necessary estimates on the number of singular values near zero. Indeed, the following lemma does that job.

\begin{lem}\label{lem:stieltjes_bound}{\em {\bf (\cite[Lemma 15]{gkz})}}
Let $\mu$ be a probability measure on $\R$. Then for any real $y>0$,
\beq
\mu\Big((-y,y)\Big) \le 2y |\Im G_\mu(\imag y)|. \notag
\eeq
\end{lem}

\vskip10pt
We now proceed to the proof of Theorem \ref{thm:main}. The idea behind the proof is the following. From Theorem \ref{thm:ssv0} we have that $s_n(S_n^d/\sqrt{d}-z)$ is not very small with large probability. Therefore we can exclude a small region near zero while computing $\langle \Log, \nu_n^z \rangle$ where we recall $\nu_n^z$ is be the \abbr{ESD} of ${\bm S}_n^{d,z}$ and ${\bm S}_n^{d,z}$ was defined in \eqref{eq:bmS}. Then we use Theorem \ref{thm:smallish_sing_control} to show that the integration of $\log(|\cdot|)$ around zero, with respect to the probability measure $\nu_n^z$, is negligible. Using Theorem \ref{thm:smallish_sing_control} we also deduce that $\{\nu_n^z\}$ converges weakly, which therefore in combination with the last observation yields {\bf Step 2} of Girko's method. Then applying the replacement lemma we finish the proof. Below we make this idea precise.

\vskip10pt


\begin{proof}[Proof of Theorem \ref{thm:main}]

Fix $\vep >0$ and $z \in \D_\vep:=B_\C(0,1-\vep)$. Denote $c_n:=e^{-\ol{C}_{\ref{thm:ssv0}}(\log n)^2/\log d}$ and let
\[
\Omega_n':= \left\{s_n\left(\f{S_n^d}{\sqrt{d}}- z\right) \ge c_n\right\}.
\]
Fixing any $\tau>0$, on the event $\Omega_n'$, we see that
\begin{align}\label{eq:log_int_split}
\int_{-\tau}^\tau |\log (|x|)|d\nu_n^z(x)& = 2 \int_0^\tau  |\log (x)|d\nu_n^z(x)\notag\\
& =   2 \int_{c_n}^\tau  |\log (x)|d\nu_n^z(x) \notag\\
&= 2\int_{c_n}^{\tau_n} |\log(x)| d\nu_n^z(dx) + 2\int_{\tau_n}^\tau |\log(x)| d\nu_n^z(dx),
\end{align}
where we set $\tau_n:= \f{(\log \log n)^{5/6}}{(\log n)^2}$.

Using Theorem \ref{thm:smallish_sing_control}, Lemma \ref{lem:stieltjes_bound}, Lemma \ref{lem:prop_wtm_infty}, and the triangle inequality \corA{we obtain}
\begin{align}\label{eq:int_small_1}
\int_{c_n}^{\tau_n} |\log(x)| d\nu_n^z(dx) & \le |\log c_n| 
\times \nu_n^z((0, \tau_n))\notag\\
& \le {2|\log c_n|}\cdot \tau_n\left|\wt{m}_n(\imag \, \tau_n)\right|, \notag\\
& \le {2|\log c_n|}\cdot \tau_n \left( \left|\wt{m}_\infty(\imag 
\tau_n)\right|+ \wt{C}_{\ref{thm:smallish_sing_control}}
\tau_n^{-3}\cdot \max\left\{\f{1}{d^{1/2}}, \f{{\log n}}{n^{1/4}}
\right\}\right)\notag\\
& \le  {4 C_{\ref{lem:prop_wtm_infty}} |\log c_n|}\cdot \tau_n =o(1),
\end{align}
on the event $\Omega_n \cap \Omega_n'$ (recall the definition of $\Omega_n$ from the statement of Theorem \ref{thm:smallish_sing_control}), 
where we used the fact $d \ge \f{(\log n)^{12}}{(\log \log n)^4}$.

Next using integration by parts it is easy to check that for any probability measure $\mu$ on $\R$ and $0\le  a_1< a_2 <1$,
\beq
\int_{a_1}^{a_2} |\log (x)| d\mu(x) \le |\log(a_2)| \mu((0,a_2))+ \int_{a_1}^{a_2} \f{\mu((0,t))}{t} dt.\label{eq:log_byparts}
\eeq
Therefore arguing similarly as above and using \eqref{eq:log_byparts} we further deduce
\begin{align}
 \int_{\tau_n}^\tau  |\log (x)|d\nu_n^z(x)
&\le |\log(\tau)| \nu_n^z\left((0,\tau)\right)+ \int_{\tau_n}^\tau \f{\nu_n^z\left((0,t)\right)}{t} dt\notag\\
 &\le  \tau |\log(\tau)| |\wt{m}_n(\imag \tau)| +  \int_{\tau_n}^\tau |\wt{m}_n(\imag t)| dt \notag\\
 &\le  \tau |\log(\tau)| |\wt{m}_\infty(\imag \tau)| +  \int_{\tau_n}^\tau |\wt{m}_\infty(\imag t)| dt \notag\\
 & \qquad \qquad \qquad \qquad +{2\wt{C}_{\ref{thm:smallish_sing_control}} \tau |\log (\tau)|\tau_n^{-3}}\cdot \max\left\{
   \f{1}{d^{1/2}}, \f{{\log n}}{n^{1/4}}\right\}
  \notag\\
 & \le 2 C_{\ref{lem:prop_wtm_infty}} \tau |\log(\tau)|  +{2\wt{C}_{\ref{thm:smallish_sing_control}} \tau |\log (\tau)|
 \tau_n^{-3}}\cdot \max\left\{
   \f{1}{d^{1/2}}, \f{{\log n}}{n^{1/4}}\right\} 
\label{eq:int_small_2}
\end{align}
on the event $\Omega_n \cap \Omega_n'$. Hence, combining \eqref{eq:int_small_1}-\eqref{eq:int_small_2} from \eqref{eq:log_int_split} we see that for any given $\delta >0$ there exists a $\tau_\delta:=\tau(\delta)$, with the property $\lim_{\delta \to 0} \tau_\delta=0$, such that
\begin{align}\label{eq:log_near0_prb_bd}
\limsup_{n \ra \infty} \P\left(\int_{-\tau_\delta}^{\tau_\delta} |\log |x|| d\nu_n^z(x) \ge \delta\right) & \le \limsup_{n \ra \infty} \P\left(\left\{\int_{-\tau_\delta}^{\tau_\delta} |\log |x|| d\nu_n^z(x) \ge \delta\right\} \cap \Omega_n \cap \Omega_n'\right)=0.
\end{align}

We next recall that
Theorem \ref{thm:smallish_sing_control} also implies that, for any $\corAa{\delta'}>0$,
  \[\lim_{n\to\infty}
    \P\left(\sup_{\xi = {\rm i} \eta: \bar{C}_{\ref{thm:smallish_sing_control}}/2 \le \eta \le \bar{C}_{\ref{thm:smallish_sing_control}}}
  |\wt{m}_n(\xi) - \wt{m}_\infty(\xi)|>\delta'\right)=0.\]
This in particular implies that $\nu_n^z$ converges weakly to $\nu_\infty^z$,
in probability
(for example, \corO{apply Montel's theorem in conjunction with}
\cite[Theorem 2.4.4(c)]{agz}), where $\nu_\infty^z$ is the probability measure corresponding to the Stieltjes transform $\wt{m}_\infty(\xi)$. Therefore
\beq\label{eq:weak_conv}
\int_{(-R, - \tau_\delta) \cup (\tau_\delta, R)} |\log |x|| d\nu_n^z(x) \ra \int_{(-R, - \tau_\delta) \cup (\tau_\delta, R)} |\log |x| | d\nu_\infty^z(x)
\qquad \text{ \corO{in probability}},
\eeq
for any $R$ positive. Recall that for $z \in \D_\vep$ the support of $\nu_\infty^z$ is contained in $[-7,7]$.
\corO{
  On the other hand, using that $\log |x|/|x|$ is decreasing for $|x|>e$,
  we have that
  \[\E\int_{(-R,R)^c} |\log |x|| d\nu_n^z(x)\leq
    \frac{\log R}{R}
  \E\int |x| d\nu_n^z(x)\leq C\frac{\log R}{R},\]
 where $C$ is an absolute constant, and \eqref{eq-perm2} was used in the last inequality.}
Therefore, choosing $R_\delta:=R(\delta)$ sufficiently large we obtain
from Markov's inequality that
\corO{
  \beq\label{eq:log_near_infty}
  \lim_{\delta\to 0}\limsup_{n \ra \infty} \P\left(
\left| \int_{(-R_\delta,R_\delta)^c} |\log |x|| d\nu_n^z(x) - \int_{(-R_\delta,R_\delta)^c} |\log |x||
d\nu_\infty^z(x) \right|>\delta\right) =0.
\eeq
}
From Lemma \ref{lem:prop_wtm_infty}, using Lemma \ref{lem:stieltjes_bound} and \eqref{eq:log_byparts} one can also check that
\beq\label{eq:log_near_0_limit}
\int_{-\tau_\delta}^{\tau_\delta} |\log |x|| d\nu_\infty^z(x) \le 4C_{\ref{lem:prop_wtm_infty}} \tau_\delta |\log \tau_\delta|.
\eeq
Since $\delta >0$ is arbitrary and $\tau_\delta \to 0$ as $\delta \to 0$, combining \eqref{eq:log_near0_prb_bd}-\eqref{eq:log_near_0_limit} we deduce that
\beq\label{eq:log_integrates}
\langle \Log, \nu_n^z \rangle \ra \langle \Log, \nu^z_\infty \rangle, \text{ in probability}.
\eeq
Now the remainder of the proof is completed using Lemma \ref{lem:replacement}. Indeed, consider ${A}_n$ the $n \times n $ matrix with i.i.d.~centered Gaussian entries with variance one. It is well-known that, for Lebesgue almost all $z$,
\beq\label{eq:log_integrates_sub_gaussian}
\f{1}{n} \log |\det(A_n/\sqrt{n} - z I_n)| \ra \langle \Log , \nu^z_\infty \rangle, \text{ almost surely}.
\eeq
For example, one can obtain a proof of \eqref{eq:log_integrates_sub_gaussian} using \cite[Lemma 4.11, Lemma 4.12]{bordenave2012around}, \cite[Theorem 3.4]{bourgade2014local}, and \cite[Lemma 3.3]{R}.

Thus setting $\D=\D_\vep$, $B_n^{(1)}=S_n^d/\sqrt{d}$, and $B_n^{(2)}=A_n/\sqrt{n}$ in Lemma \ref{lem:replacement} we see that assumption (ii) there is satisfied. The assumption (i) of Lemma \ref{lem:replacement} follows from
\corO{
  \eqref{eq-perm2}.}
Hence, using Lemma \ref{lem:replacement} and the
\corN{circular law for i.i.d.~complex Gaussian matrices (which follows from e.g.~\cite{bai1997circular}, but essentially goes back to Ginibre \cite{Ginibre65})}, we obtain that for every $\vep >0$ and every $f_\vep \in C_c^2(\C)$, supported on $\D_\vep$,
\beq\label{eq:f_tau}
\int f_\vep(z) d\mu_n(z) \ra \f{1}{\pi}\int f_\vep(z) d\gm(z), \text{ in probability},
\eeq
where for brevity we denote $\mu_n:=L_{S_n^d/\sqrt{d}}$.
To finish the proof it now remains to show that one can extend the convergence of \eqref{eq:f_tau} to all $f \in C_c^2(\C)$. That is we need to show that for any $\delta >0$ and $f \in C_c^2(\C)$
\beq
\P\left(\left| \int f(z) d\mu_n(z)- \f1\pi\int_{B_\C(0,1)} f(z) d \gm(z) \right| \ge \delta \right) \to 0 \qquad \text{ as } \qquad n \to \infty.
\eeq
To this end, for any $\vep >0$ define a function $i_\vep \in C_c^2(\C)$ such that $i_\vep$ is supported on $\D_\vep$, $i_\vep\equiv 1$ on $\D_{2\vep}$ and $i_\vep \in [0,1]$ on $\D_\vep \setminus \D_{2\vep}$. Next fix $\vep$ such that $M(1 -(1-2\vep)^2) \le \delta/8$ where $M:=\sup_x |f(x)|$. Denote $f_\vep:=f i_\vep$ and $\bar{f}_\vep:= f - f_\vep$. Applying \eqref{eq:f_tau} for the function $i_\vep$ and the triangle inequality we note that 
\begin{align}\label{eq:f_tau_1}
 \P\left( \left|\int \bar f_\vep(z) d\mu_n(z)\right| \ge \delta/4 \right) & \le
\P\left( \left|\int  (1-i_\vep(z)) d\mu_n(z)\right| \ge \f{\delta}{4M} \right)\notag\\
& \le \P\left( \left|\int  i_\vep(z) d\mu_n(z) - \f1\pi \int i_\vep(z) d\gm (z)\right| \ge  \f{\delta}{8M} \right)  \to 0,
\end{align}
as $n \to \infty$, where we have used the fact that 
\beq\label{eq:f_tau_2}
\left|\f1\pi \int_{B_\C(0,1)} (1-i_\vep(z)) d\gm (z)\right| \le \f{1}{\pi} \int_{B_\C(0,1)\setminus \D_{2\vep}} d\gm(z) \le \f{\delta}{8M}, 
\eeq
by our choice of $\vep$. 
Therefore combining \eqref{eq:f_tau}, \eqref{eq:f_tau_1}-\eqref{eq:f_tau_2} and the triangle inequality we deduce
\begin{align*}
&\,  \P\left(\left| \int f(z) d\mu_n(z)- \f1\pi\int f(z) d \gm(z) \right| \ge \delta \right) \\
 \le & \, \P\left(\left| \int f_\vep(z) d\mu_n(z)- \f1\pi\int f_\vep(z) d \gm(z) \right| \ge \delta/2 \right)  + \P\left( \left|\int \bar f_\vep(z) d\mu_n(z)\right| \ge \delta/4 \right) \to 0,
\end{align*}
as $n \to \infty$. 
This completes the proof of the theorem.
\end{proof}


\noindent
\corO{ \subsection*{Acknowledgement} O.Z.~thanks Alice Guionnet for her help in
formulating the loop equations for the Ginibre ensemble. The authors also thank
Amir Dembo for useful discussions. Finally, we thank an anonymous referee for 
suggesting a significant 
simplification of our original derivation of the loop equations.}

\bibliographystyle{plain}
\bibliography{circ_law_infinite_permutation}

\end{document}